\documentclass[journal]{IEEEtran}
\usepackage[margin=1in]{geometry}
\usepackage{amsmath}
\usepackage{amssymb,amsthm}
\usepackage{mathtools}
\usepackage{bm}
\usepackage{algorithm}
\usepackage{algorithmic}
\usepackage{graphicx}
\usepackage{booktabs}
\usepackage{hyperref}
\usepackage{cite}
\usepackage{setspace}
\usepackage[usenames,dvipsnames]{color}
\usepackage{multirow}
\usepackage{array}
\usepackage{siunitx}
\usepackage{enumitem}
\newcommand{\rev}[1]{#1}
\newenvironment{revblock}{}{}
\newcommand{\nrev}[1]{{\color{blue}#1}}

\newcommand{\oldrev}[1]{#1}
\newenvironment{oldrevblock}{}{}

\newtheorem{theorem}{Theorem}[section]
\newtheorem{lemma}[theorem]{Lemma}
\newtheorem{proposition}[theorem]{Proposition}
\newtheorem{corollary}[theorem]{Corollary}
\newtheorem{definition}{Definition}[section]
\newtheorem{assumption}{Assumption}
\newtheorem{remark}{Remark}[section]
\newtheorem{example}{Example}[section]

\newcommand{\calK}{\mathcal{K}}
\newcommand{\calF}{\mathcal{F}}
\newcommand{\calX}{\mathcal{X}}

\newcommand{\calT}{\mathcal{T}}
\newcommand{\calA}{\mathcal{A}}
\newcommand{\calE}{\mathcal{E}}
\newcommand{\calV}{\mathcal{V}}
\newcommand{\calI}{\mathcal{I}}
\newcommand{\calC}{\mathcal{C}}
\newcommand{\bPhi}{\boldsymbol{\Phi}}
\newcommand{\bpsi}{\boldsymbol{\psi}}
\newcommand{\bPsi}{\boldsymbol{\Psi}}
\newcommand{\bF}{\mathbf{F}}
\newcommand{\bE}{\mathbf{E}}
\newcommand{\bGamma}{\boldsymbol{\Gamma}}
\newcommand{\bH}{\mathbf{H}}
\newcommand{\bA}{\mathbf{A}}
\newcommand{\bT}{\mathbf{T}}
\newcommand{\bK}{\mathbf{K}}
\newcommand{\bL}{\mathbf{L}}
\newcommand{\bP}{\mathbf{P}}
\newcommand{\bQ}{\mathbf{Q}}
\newcommand{\bR}{\mathbf{R}}
\newcommand{\bC}{\mathbf{C}}
\newcommand{\bD}{\mathbf{D}}
\newcommand{\bG}{\mathbf{G}}
\newcommand{\bI}{\mathbf{I}}
\newcommand{\bO}{\mathbf{O}}
\newcommand{\bone}{\mathbf{1}}
\newcommand{\bzero}{\mathbf{0}}
\newcommand{\bW}{\mathbf{W}}

\newcommand{\bS}{\mathbf{S}}

\newcommand{\bJ}{\mathbf{J}}
\newcommand{\bN}{\mathbf{N}}
\newcommand{\bPi}{\boldsymbol{\Pi}}
\newcommand{\bLambda}{\boldsymbol{\Lambda}}

\newcommand{\bX}{\mathbf{X}}
\newcommand{\bY}{\mathbf{Y}}
\newcommand{\bM}{\mathbf{M}}
\newcommand{\bU}{\mathbf{U}}
\newcommand{\normF}[1]{\left\lVert#1\right\rVert_{\mathrm{F}}}
\newcommand{\normone}[1]{\left\lVert#1\right\rVert_{1}}

\newcommand{\norm}[1]{\left\lVert#1\right\rVert}
\newcommand{\Lnorm}[1]{\left\lVert#1\right\rVert_{L^2}}
\newcommand{\inner}[2]{\left\langle #1,\, #2\right\rangle}
\newcommand{\expect}{\mathbb{E}}
\newcommand{\tr}{\mathrm{tr}}
\newcommand{\rank}{\mathrm{rank}}
\newcommand{\vecop}{\mathrm{vec}}
\newcommand{\prox}{\mathrm{prox}}
\newcommand{\sign}{\mathrm{sign}}
\newcommand{\diag}{\mathrm{diag}}
\newcommand{\blkdiag}{\mathrm{blkdiag}}
\newcommand{\clip}{\mathrm{clip}}
\DeclareMathOperator*{\argmin}{arg\,min}

\newcommand{\by}{\bm{y}}
\newcommand{\ba}{\bm{a}}
\newcommand{\bb}{\bm{b}}
\newcommand{\bc}{\bm{c}}
\newcommand{\be}{\bm{e}}
\newcommand{\br}{\bm{r}}
\newcommand{\bs}{\bm{s}}
\newcommand{\bt}{\bm{t}}
\newcommand{\bu}{\bm{u}}
\newcommand{\bphi}{\bm{\phi}}
\newcommand{\bchi}{\bm{\chi}}
\newcommand{\bzeta}{\bm{\zeta}}
\newcommand{\bdelta}{\bm{\delta}}
\newcommand{\bxi}{\bm{\xi}}
\newcommand{\bvarsigma}{\bm{\varsigma}}
\newcommand{\bvarrho}{\bm{\varrho}}
\newcommand{\bvartheta}{\bm{\vartheta}}
\newcommand{\bepsilon}{\bm{\epsilon}}
\newcommand{\bv}{\bm{v}}
\newcommand{\bq}{\bm{q}}
\newcommand{\calD}{\mathcal{D}}
\newcommand{\calZ}{\mathcal{Z}}
\newcommand{\calS}{\mathcal{S}}
\newcommand{\bTheta}{\boldsymbol{\Theta}}

\newcommand{\bB}{\bm{B}}
\newcommand{\bDelta}{\bm{\Delta}}

\newcommand{\supp}{\operatorname{supp}}

\newcommand{\prob}{\mathbb{P}}

\newcommand{\Pip}{\bPi_{\!p}}
\newcommand{\bJPhi}{\mathbf{J}_{\Phi}}
\newcommand{\hG}{\hat{\mathbf{G}}}
\newcommand{\hD}{\hat{\mathbf{D}}}
\newcommand{\hU}{\hat{\mathbf{U}}}
\newcommand{\bd}{\bm{d}}

\title{Joint State and Topology Estimation of Nonlinear Networks via Koopman Lifting}

\date{}

\author{Chuansen Peng, Qiaojia Zhu, Xiaojing Shen and Yunmin Zhu
\thanks{The work was supported by the Sichuan Provincial Natural Science Foundation Project (Innovative Research Group) 2026NSFSCZY0056 and in part by the National Natural Science
 Foundation of China (NSFC) under Grant U2133208, 62203313. \textit{(Corresponding author: Xiaojing Shen)}\\
 \indent Chuansen Peng, Qiaojia Zhu, Xiaojing Shen, and Yunmin Zhu are with the College of Mathematics, Sichuan University, Chengdu, Sichuan, 610064, China. (e-mail: pengchuansen@stu.scu.edu.cn; zhuqiaojiaa@163.com; shenxj@scu.edu.cn; ymzhu@scu.edu.cn).\\
 \indent This work has been submitted to the IEEE for possible publication. Copyright may be transferred without notice, after which this version may no longer be accessible.}}
\begin{document}
\maketitle

%% ===============================================================
\begin{abstract}
Joint estimation of the latent node states and the unknown, possibly
time-varying topology of a nonlinear network from partial, noisy
streaming measurements forms a feedback loop of two estimators. We
study this loop when the node dynamics are learned from data.
Koopman-GKFA uses an \emph{interaction-aware structured Koopman
lifting}: the node dynamics evolve linearly in lifted coordinates, and
the coupling, represented by pairwise interaction features, stays
affine in the physical edge weights. Additive node-separable lifting
with nonlinear observables cannot represent the cross-node terms
created by coupling, whereas the proposed lifting admits a tight
network-level residual bound without second-order terms on the state
coordinates. The filter re-lifts its estimate at every step, since
undetectable lifted modes cannot be stabilized; for linear sensors it
is the extended Kalman filter of the learned surrogate. The topology is
estimated by an elastic-net regression of measurement prediction errors
on the physical weights, solved by an ADMM with certified accuracy. A
small-gain theorem for the filter--regression loop gives ultimate error
bounds that are explicit in noise, model mismatch, topology drift,
regularization and solver accuracy, and a multi-step condition extends
it to intermittent sensing, where one-step contraction can fail.
Experiments on synthetic networks, a real-world traffic dataset and an
in silico gene-regulatory benchmark compare the method with
representative baselines.
\end{abstract}


%% ===============================================================
\section{Introduction}
\label{sec:intro}
%% ===============================================================

\IEEEPARstart{N}{etworked} dynamical systems pervade modern science and engineering. In power transmission grids, synchronous generators are coupled through transmission-line admittances, and real-time estimation of rotor angles and
angular velocities is indispensable for stability monitoring and protection \cite{zhao2019power}. In gene regulatory networks, the expression level of each gene is driven by
its own transcriptional kinetics and modulated by activators and repressors
whose interaction strengths define the regulatory topology \cite{bansal2007infer}.
In vehicular platoons operating under nonlinear car-following dynamics,
the coupling strength between consecutive vehicles determines whether
upstream disturbances are amplified or attenuated along the chain
\cite{helly1959simulation}.
In multi-robot formations and autonomous agent systems, the coordination
protocol implicitly defines an interaction graph that is rarely pre-specified
and must be inferred from local observations \cite{piperakis2018nonlinear}.
Across all these domains, a networked system is governed by two inseparable
constituents: the \emph{nodal state}, tracking the instantaneous
configuration of each agent, and the \emph{network topology} (with $p$ nodes), encoding the
direction and strength of inter-agent coupling through a weighted adjacency
matrix $\bA \in \mathbb{R}^{p \times p}$.

In practice, $\bA$ is frequently unknown or time-varying: transmission-line
failures may alter the power grid topology without triggering sensor alarms;
regulatory interactions are obscured by transcriptional noise; robot
communication graphs change as agents enter or leave proximity.
Recovering $\bA$ alongside the latent node states from partial,
noise-corrupted streaming observations constitutes the
\emph{joint state estimation and topology inference} problem.
This problem is intrinsically harder than either task in isolation,
because accurate state estimation presupposes knowledge of the coupling
structure, whereas reliable topology inference demands well-estimated state
trajectories, a mutual dependence that calls for a unified algorithmic
framework rather than independent solutions to two separate problems.

A substantial body of literature addresses topology inference under the simplifying assumptions that node states are directly observable and that the underlying graph is static. Graph signal processing (GSP) methods exploit spectral templates, diffusion models, or signal smoothness priors to recover the Laplacian or adjacency matrix~\cite{dong2019learning,segarra2017network,mateos2019connecting}. For instance, Dong et al.~\cite{dong2019learning} formulate graph learning as a signal reconstruction problem under Laplacian constraints, while Segarra et al.~\cite{segarra2017network} infer network structure from spectral templates induced by graph stationarity. Mateos et al.~\cite{mateos2019connecting} provide a unified perspective linking graph topology identification with signal processing on graphs. Recent work has also examined topology inference from causality and limited-excitation viewpoints, with nonasymptotic guarantees developed in~\cite{li2023topology} and probability-based fusion under few excitations studied in~\cite{jiao2025inferring}\oldrev{; connectivity of nonlinear oscillator networks can also be revealed from their response dynamics~\cite{timme2007revealing}, and neural relational inference~\cite{kipf2018neural} and causal discovery for nonlinear time series~\cite{runge2019detecting} learn interaction graphs from fully observed trajectories}. Most of these methods use batch observations of fully observed node states and do not infer latent states recursively. Parallel to graph topology learning, Kalman filter (KF) theory
\cite{kalman1960new} provides the minimum mean-squared error (MMSE)
estimator for linear Gaussian state-space systems and is widely adopted for
online state estimation in large-scale networked systems
\cite{sayed2014adaptation}.
The KF achieves recursive, computationally efficient updates but operates
under the assumption that the state-transition matrix, which implicitly
encodes the coupling topology, is known.
When $\bA$ is unknown, the KF cannot be applied directly, and the joint
estimation problem requires a fundamentally different treatment.

The joint state estimation and topology inference problem for \emph{linear}
graphical dynamical systems (GDS) was pioneered by Ramezani-Mayiami and
Beferull-Lozano~\cite{ramezani2018joint,ramezani2018topology}, who proposed
alternating between a Kalman update for the latent state and a sub-gradient
step on an $\ell_1$-regularized least-squares objective for the topology.
Elvira and Chouzenoux~\cite{elvira2022graphical} developed GraphEM, a
batch expectation-maximization algorithm for graphical inference in
linear-Gaussian state-space models that achieves high statistical efficiency
but operates offline.
Fang et al.~\cite{fang2025joint} resolved the real-time constraint with
their KF-ADMM algorithm, which interleaves at each time step a Kalman
filter update for state estimation with an ADMM
iteration~\cite{boyd2011distributed} for $\ell_1$-regularized topology
inference, and is the current reference point for online joint
estimation in linear GDS.
\oldrev{Online sparse recursive least squares~\cite{babadi2010sparls,angelosante2010online}
and the stability theory of forgetting-factor tracking
algorithms~\cite{guo1995exponential} provide tools for the topology step,
but they assume exactly measured regressors.}

These results concern linear models, whereas many networks are nonlinear. Power system swing equations are nonlinear; gene regulatory kinetics follow
Hill functions; car-following dynamics exhibit saturation nonlinearities;
and coupled phase oscillators interact through the phase difference of
the receiving and the sending node.
In the latter two examples the coupling depends jointly on the
receiver and the sender, a feature that any principled extension must
accommodate.

Existing approaches to nonlinear state estimation do not close this gap.
The Extended Kalman Filter (EKF)~\cite{schmidt1966application} and Unscented Kalman Filter
(UKF)~\cite{julier2004unscented} are the two most widely deployed
strategies: the EKF linearizes the nonlinear dynamics at each step via a
first-order Taylor expansion, while the UKF propagates a deterministically
chosen set of sigma points through the full nonlinear map, achieving
second-order accuracy without explicit Jacobian computation. Particle filters (PFs)~\cite{arulampalam2002tutorial} offer asymptotically exact
Bayesian inference under arbitrary nonlinear, non-Gaussian dynamics, but suffer from particle degeneracy in high dimension. Sagi et al.~\cite{sagi2023extended} adapted the EKF to graph signals
(GEKF) for nonlinear GDS with a \emph{known} topology, while
Li et al.~\cite{li2023unscented} developed the analogous graph UKF (GUKF)
under the same known-topology assumption.
\begin{oldrevblock}
When the topology must also be inferred, two routes are common. If
$\bA$ is appended to the state (joint EKF/UKF), the topology enters the
dynamics multiplicatively with the state, and the convergence of the
resulting estimator is delicate, as known since Ljung's analysis of the
EKF as a parameter estimator~\cite{ljung1979asymptotic} and the dual
filtering literature~\cite{wan1997dual}. Alternatively, for a
\emph{known} interaction function the coupling is affine in the edge
weights, so the topology can be fitted by a convex regression on
filtered states, with or without lifting. The difficulty is then not
convexity but the feedback: the regression is fed with estimated states
(an errors-in-variables problem~\cite{soderstrom2007errors}), and the
filter is fed with the estimated topology. To our knowledge, this
closed loop has not been analyzed for nonlinear networks.
\end{oldrevblock}
State-space topology identification via augmented state
vectors~\cite{coutino2020state} preserves a state-space form but
offers no guarantee on the topology block.

Graph neural networks (GNNs)~\cite{hamilton2017inductive} and their
recurrent extensions have demonstrated empirical success in learning
nonlinear graph dynamics.
Buchnik et al.~\cite{buchnik2024gspkalmannet} proposed GSP-KalmanNet,
which integrates GSP priors with a neural-augmented Kalman gain learned
from data, achieving strong performance on nonlinear filtering tasks. GKNet~\cite{sabbaqi2026gknet} learns graph-based state-space
models jointly with Kalman-type inference by model-based deep learning;
the graph enters as a (possibly uncertain) inductive bias rather than as
a set of physical edge weights to be estimated with error guarantees. Neural approaches perform well empirically, but they typically
need substantial training data, provide few analytical guarantees, and
treat the graph as an inductive bias rather than as physical edge
weights estimated online.

Koopman operator theory~\cite{koopman1931hamiltonian} provides a principled framework for representing nonlinear dynamics through linear evolution in a lifted observable space. Its spectral interpretation, established by Mezi\'{c}~\cite{mezic2005spectral}, shows that Koopman eigenvalues and eigenfunctions encode the global structure of the underlying flow and support finite-dimensional approximation. This viewpoint has led to data-driven identification methods such as Dynamic Mode Decomposition (DMD)~\cite{schmid2010dynamic} and Extended DMD (EDMD)~\cite{williams2015data}, with convergence guarantees for EDMD established in~\cite{korda2018convergence}. Subsequent work extended the Koopman framework to identification, control, and bilinearization of nonlinear systems~\cite{brunton2016koopman,proctor2018generalizing,mauroy2019koopman,goswami2021bilinearization,korda2020optimal}, while deep learning-based parameterizations further enlarged the class of admissible observables~\cite{lusch2018deep}, and stability-preserving parameterizations of learned Koopman embeddings were recently proposed in~\cite{fan2026learning}. Koopman-based observers and Kalman filters have been designed for
state estimation with known models~\cite{surana2016linear,netto2018robust},
and online DMD tracks slowly varying linear surrogates~\cite{zhang2019online}. Koopman Kalman filters with error bounds relative to the optimal
filter have been derived from kernel EDMD approximations and used for
parameter estimation~\cite{olguin2025koopman}; there, the system is
not a network with an unknown interaction structure.

Koopman methods have also been applied to networks. Mauroy and
Goncalves~\cite{mauroy2020koopman} recast nonlinear identification as
linear identification of the Koopman generator and used it to
reconstruct network topologies; Mei and
Oguchi~\cite{mei2022network,mei2023realtime} identified network
structures by combining Koopman representations with sparse
identification, including real-time variants that relax the need for
isolated-node data\rev{, and extracted coupling functions from measured
output data~\cite{mei2023output}}; Anantharaman and
Mauroy~\cite{anantharaman2025koopman} identified heterogeneous local
dynamics and coupling functions of continuous-time networks; and Li
et al.~\cite{li2024parallel} built interacting node-wise (parallel)
Koopman subsystem models for distributed moving-horizon state
estimation. Node-wise lifting and Koopman-based network identification
are therefore \emph{not} new by themselves. The cited
identification methods operate on recorded full-state or output data
and do not estimate latent states recursively, whereas the cited
estimation methods assume a known interconnection. To the best of our
knowledge, no Koopman-based method estimates latent states and an
unknown, time-varying physical topology \emph{jointly and online} with error guarantees for the closed loop formed by the two estimators.

From a cybernetic viewpoint, the joint problem is a feedback
interconnection of two learning and estimation modules that exchange
information online: the filter passes estimated states to the topology
learner, which returns an estimated interaction structure to the filter.
Each module can be designed with classical tools, but the reliability
of the interconnection is not inherited from that of its components,
and it is further affected by the mismatch of a model learned from data.
It is the reliability of this loop, rather than the accuracy of either
module in isolation, that we seek to quantify.

Moving to this joint setting exposes three interlocking difficulties,
which constitute the research gap addressed here.
\begin{enumerate}[leftmargin=*]
\item[\textbf{(D1)}] \textbf{Structure preservation.} For a known
interaction function, the physical model is already affine in the edge
weights; the question is whether this property survives lifting. An additive node-wise lifting with nonlinear observables loses it:
coupling creates cross-node terms that no node-wise dictionary removes.
\item[\textbf{(D2)}] \textbf{Observability of the lifted filter.}
Lifting adds coordinates that the sensors may not see. When such a
coordinate evolves along a non-decaying mode of the lifted
dynamics, as the rotation modes of trigonometric dictionaries do, no gain can stabilize the lifted estimation error. A lifted Kalman filter needs an additional mechanism to be certifiable.
\item[\textbf{(D3)}] \textbf{Closed-loop coupling.} The topology
regression uses estimated states and the filter uses the estimated
topology. If filtered lifted increments are regressed on the
interaction features, the regressand contains the previous topology
estimate itself, information about $\bA$ enters only through the
Kalman gain, and an artificial error floor appears. Guarantees for each
subproblem in isolation do not compose into a guarantee for the loop.
\end{enumerate}
Accordingly, a satisfactory method should satisfy: \textbf{(R1)} it
handles nonlinear node dynamics and nonlinear interactions that may
depend on both the receiving and the sending node; \textbf{(R2)} it
infers an unknown, possibly time-varying topology online from partial,
noisy measurements; and \textbf{(R3)} its guarantees are stated for the
algorithm that is actually implemented, and they separate the
quantities that are computed from data from the a priori bounds on
model mismatch, noise and topology variation, so that their scope is explicit.

This paper addresses (D1)--(D3) with the \textbf{Koopman
Graph-structured Kalman Filter--ADMM (Koopman-GKFA)} framework. Each node is lifted by a state-inclusive dictionary and the coupling is represented through features of \emph{node pairs} multiplied by the physical edge weights (D1). The lifted filter re-lifts its posterior at every step, so that its error is governed by the physical error and non-decaying lifted modes are removed by construction (D2). The topology is estimated from the \emph{one-step prediction error of the measurements}, which contains no explicit previous topology estimate, by an
elastic-net regression on the free edge weights (D3); the loop is closed
by a small-gain argument, which is extended to intermittent sensing
through a multi-step contraction condition. The contributions are
organized along three levels, representation, inference and
certification of the estimator loop.

\begin{itemize}
\item \textbf{Representation.} Additive node-separable lifting with a
nonlinear observable has a network residual bounded away from zero for
any dictionary size. The
proposed interaction-aware lifting is affine in the physical edge
weights, and its network-level residual admits a tight bound whose
state block has no second-order remainder; the lifted noise and the identifiability
of the calibration are characterized in
Propositions~\ref{prop:lifted_noise} and~\ref{prop:calib}.
\item \textbf{Inference.} Undetectable lifted modes cannot be
stabilized by any gain; the
re-lifted filter is, for linear sensors, the EKF of the learned
surrogate. The topology is estimated from
measurement prediction errors, without an explicit previous topology
estimate, by an elastic-net regression solved by a certified ADMM; multi-step prediction errors
handle sensors of higher relative degree.
\item \textbf{Certification of the estimator loop.} The topology error
under estimated states is bounded without self-referential terms, and a small-gain theorem gives
explicit ultimate bounds for the loop. When one-step contraction fails, a multi-step theorem covers intermittent sensing.
Restricted-eigenvalue, tracking and support-recovery results complete
the analysis; all guarantees are conditional on explicit excitation and
contraction assumptions.
\item \textbf{Empirical evaluation.} On synthetic Kuramoto networks,
Koopman-GKFA gives the most accurate joint estimates at moderate
coupling, whereas a joint UKF with the exact model is more accurate in
state under strong coupling. A controlled comparison with a
physical-space surrogate that shares its state rows, topology estimator
and tuning shows that the lifting does not improve accuracy: the two
are identical for linear sensors and practically indistinguishable for
a sensor of nonlinear features, where the lifting provides a measurement
model that is covered by the analysis. Experiments on the NGSIM US-101
traffic dataset and the in silico DREAM4 benchmark illustrate the
behavior of the method on practical nonlinear dynamics.
% \item \textbf{Empirical evaluation.} On synthetic Kuramoto networks,
% Koopman-GKFA gives the most accurate joint estimates at moderate
% coupling, whereas a joint UKF with the exact model is more accurate in
% state under strong coupling. Experiments on the NGSIM US-101 traffic
% dataset and the in silico DREAM4 benchmark illustrate the behavior of
% the method on practical nonlinear dynamics.
\end{itemize}

\begin{oldrevblock}
The remainder of this paper is organized as follows.
Section~\ref{sec:prelim} reviews graph theory, Koopman operator theory,
and EDMD. Section~\ref{sec:problem} formulates the problem and discusses
identifiability. Section~\ref{sec:lifting} constructs the
interaction-aware lifting and characterizes its residual and noise.
Section~\ref{sec:algorithm} presents Koopman-GKFA: representation
learning \rev{and its identifiability}, the re-lifted filter, the
prediction-error topology estimator and its ADMM solver.
Section~\ref{sec:theory} establishes the solver, topology, closed-loop,
\rev{multi-step,} sparse-regime, tracking and support-recovery
results. Section~\ref{sec:simulation} reports experiments, and
Section~\ref{sec:conclusion} concludes. Proofs are given in the
Appendix.
\end{oldrevblock}

\textbf{Notation}. $\bI_p$ denotes the $p\times p$ identity matrix and
$\bone_p\in\mathbb{R}^p$ the all-ones vector; $\bX\otimes\bY$ is the
Kronecker product; $\vecop(\bX)$ stacks the columns of $\bX$.
$\norm{\cdot}$ denotes the Euclidean norm of a vector and the induced
(spectral) norm of a matrix; $\normF{\bX}$, $\normone{\bX}=\sum_{ij}|x_{ij}|$,
$\norm{\bX}_\infty=\max_i\sum_j|x_{ij}|$ and
$\norm{\bX}_{1\to1}=\max_j\sum_i|x_{ij}|$ are the Frobenius, entrywise
$\ell_1$, maximum row-sum and maximum column-sum norms, and $|\bX|$ is
the entrywise absolute value; for vectors, $\norm{\cdot}_1$ and
$\norm{\cdot}_\infty$ are the $\ell_1$ and $\ell_\infty$ norms. $\lambda_{\min}(\cdot)$ and
$\lambda_{\max}(\cdot)$ are the extreme eigenvalues of a symmetric
matrix, $\rho(\cdot)$ is the spectral radius, and $\bX\succeq\bO$ ($\bX\succ\bO$) denotes positive
semidefiniteness (definiteness); for $\bG\succ\bO$,
$\norm{\bv}_{\bG}\triangleq(\bv^\top\bG\bv)^{1/2}$. $\blkdiag(\bM_1,\ldots,\bM_p)$ is
block-diagonal and $\mathrm{col}\{x_{i}\}=[x_{1}^\top,\ldots,x_{p}^\top]^\top$.
For a random vector $X$, $\Lnorm{X}\triangleq(\expect\norm{X}^2)^{1/2}$;
Minkowski's inequality $\Lnorm{X+Y}\le\Lnorm{X}+\Lnorm{Y}$ is used
throughout. $\expect[\,\cdot\mid\calF]$ is conditional expectation.
$[\bX]_{(i,j)}$ denotes the $(i,j)$-th block of a block matrix. Bold $\bK_t$ always denotes a filter gain; the Kuramoto coupling
strength is the scalar $K$.

%% ===============================================================
\section{Preliminaries}
\label{sec:prelim}
%% ===============================================================

\subsection{Graph Theory}
Let $\mathcal{G}=(\calV,\calE,\bA)$ be a weighted directed graph with
node set $\calV=\{1,\ldots,p\}$, edge set
$\calE\subseteq\calV\times\calV$, and weighted adjacency matrix
$\bA=[a_{ij}]\in\mathbb{R}^{p\times p}$, where $a_{ij}\neq0$ if and only
if $(i,j)\in\calE$, i.e., node $j$ influences node $i$. Self-effects are
modeled by the node dynamics, hence $a_{ii}=0$. Signed weights are
allowed; nonnegativity, symmetry or known absent edges are imposed only
when physically justified. We write $s\triangleq|\calE|$ for the number
of edges, $d\triangleq\max_i|\{j:a_{ij}\neq0\}|$ for the maximum
in-degree, $a_{\max}\triangleq\max_{ij}|a_{ij}|$, and
$a_{\min}\triangleq\min_{(i,j)\in\calE}|a_{ij}|$ for the minimum edge
magnitude. When the topology varies in time we write $\bA_k$, $\calE_k$.

\subsection{Koopman Operator Theory}
\label{subsec:koopman_prelim}

\begin{definition}[Koopman Operator]
\label{def:koopman}
Let $\mathcal{F}$ be a vector space of observables
$\varphi:\mathcal{X}\to\mathbb{R}$ closed under composition with $f$.
For the dynamical system $x_{k+1}=f(x_k)$, the Koopman operator
$\calK:\mathcal{F}\to\mathcal{F}$ is the composition operator
\begin{equation}
  (\calK\varphi)(x) \;\triangleq\; \varphi(f(x)),
  \quad\forall\,\varphi\in\mathcal{F},\; x\in\mathcal{X}.
\end{equation}
\end{definition}

$\calK$ is linear regardless of the nonlinearity of $f$. A
finite-dimensional dictionary $\bPhi(x)=[\varphi_1(x),\ldots,
\varphi_N(x)]^\top$ and a matrix $\bF^\phi\in\mathbb{R}^{N\times N}$
define the finite-dimensional approximation
\begin{equation}
  \bPhi(f(x)) = \bF^\phi\bPhi(x) + \boldsymbol{\varepsilon}(x),
  \label{eq:koopman_finite}
\end{equation}
where $\boldsymbol{\varepsilon}(x)$ is the truncation residual,
which vanishes identically for a suitable $\bF^\phi$ when the subspace
$\mathrm{span}\{\varphi_1,\ldots,\varphi_N\}$ is Koopman-invariant.

\begin{definition}[Separable Dictionary]
\label{def:separable}
A dictionary $\bPhi:\mathbb{R}^{np}\to\mathbb{R}^{Np}$ for a
$p$-node network with node states $x_i\in\mathbb{R}^n$ is
\emph{separable} if
\begin{equation}
  \bPhi(x) = \bigl[\bpsi_1(x_1)^\top,\,\bpsi_2(x_2)^\top,\,\ldots,
  \bpsi_p(x_p)^\top\bigr]^\top,
  \label{eq:separable}
\end{equation}
where $\bpsi_i:\mathbb{R}^n\to\mathbb{R}^N$ is a per-node local
dictionary acting only on $x_i$.
\end{definition}

\subsection{Extended Dynamic Mode Decomposition}
\label{subsec:edmd}

Given $M$ snapshot pairs $\{(x_\ell,x_\ell^+)\}_{\ell=1}^{M}$
with $x_\ell^+=f(x_\ell)$, EDMD~\cite{williams2015data} computes
\begin{equation}
  \hat{\bF}^\phi = \argmin_{\bF}\;\sum_{\ell=1}^{M}
  \|{\bPhi(x_\ell^+) - \bF\,\bPhi(x_\ell)}\|_2^2
  = \bPsi_+\bPsi_-^\dagger,
  \label{eq:edmd}
\end{equation}
where $\bPsi_-=[\bPhi(x_1),\ldots,\bPhi(x_M)]$,
$\bPsi_+=[\bPhi(x_1^+),\ldots,\bPhi(x_M^+)]$, and $(\cdot)^\dagger$
denotes the Moore--Penrose pseudoinverse.
Korda and Mezi\'c~\cite{korda2018convergence} proved that, for
data drawn from a measure $\mu$ and $M\to\infty$, EDMD converges to the
$L^2(\mu)$-orthogonal compression of $\calK$ onto the dictionary span,
and that these compressions converge to $\calK$ in the strong operator
topology as $N\to\infty$. These are $L^2(\mu)$ statements; they do not
imply that the residual in~\eqref{eq:koopman_finite} is uniformly small
on the state space, nor anything about networks. This distinction is
central in Section~\ref{sec:lifting}.

%% ===============================================================
\section{Problem Formulation}
\label{sec:problem}
%% ===============================================================

\subsection{Nonlinear Graphical Dynamical System}
\label{subsec:model}

Consider the graph $(\calV,\calE_k,\bA_k)$ with $p$ nodes. The state
$x_{i,k}\in\mathbb{R}^n$ of node $i$ and its measurement
$z_{i,k}\in\mathbb{R}^q$ satisfy
\begin{align}
  x_{i,k+1} &= f_i(x_{i,k})
               + \sum_{j=1}^p a_{ij,k}\,g(x_{i,k},x_{j,k})
               + w_{i,k},
               \label{eq:state}\\
  z_{i,k}   &= h_i(x_{i,k}) + v_{i,k},
               \label{eq:meas}
\end{align}
where $f_i:\mathbb{R}^n\to\mathbb{R}^n$ is the self-transition map of
node $i$, $g:\mathbb{R}^n\times\mathbb{R}^n\to\mathbb{R}^n$ is the
interaction function, which may depend on both the receiving node $i$
and the sending node $j$, $h_i:\mathbb{R}^n\to\mathbb{R}^q$ is the
measurement map, and $a_{ij,k}$ is the unknown topology weight
($a_{ii,k}=0$). The sequences $\{w_{i,k}\}$ and $\{v_{i,k}\}$ are
zero-mean, independent across time and of each other and of $x_0$, with
covariances $\bQ_{i,k}$ and $\bR_{i,k}$; Gaussianity is \emph{not}
required except in Proposition~\ref{prop:sparse}. We set $x_k\triangleq\mathrm{col}\{x_{i,k}\}$,
$z_k\triangleq\mathrm{col}\{z_{i,k}\}$, $\bQ_k=\blkdiag(\bQ_{i,k})$,
$\bR_k=\blkdiag(\bR_{i,k})$,
$\sigma_w^2\triangleq\sup_k\tr(\bQ_k)$ and
$\sigma_v^2\triangleq\sup_k\tr(\bR_k)$. When $\bA_k\equiv\bA$ the
topology is static. Partial observation is described by a sensing schedule: at time
$k$ only the nodes in $\calV_k^{\rm obs}\subseteq\calV$ are measured.
The schedule is deterministic or independent of $(x_0,\{w_k\},\{v_k\})$;
in the latter case all statements below hold conditionally on it. The same holds for the topology sequence $\{\bA_k\}$ (e.g., the
random-walk drift of Section~\ref{subsec:synthetic}).

\begin{remark}
Model~\eqref{eq:state}--\eqref{eq:meas} contains the linear graphical
model of~\cite{fang2025joint} ($f_i(x)=\bF_i x$, $g(x_i,x_j)=\bGamma x_j$,
$h_i(x)=\bH_i x$), the sender-only nonlinear model $g(x_i,x_j)=g(x_j)$,
which covers Hill-type gene regulation~\cite{bansal2007infer}, and
receiver-dependent interactions such as the Kuramoto coupling
$\sin(\theta_j-\theta_i)$ and the car-following coupling
$\tanh(\kappa(v_j-v_i))$ used in Section~\ref{sec:simulation}. The
interaction function $g$ is common to all edges; heterogeneity is
carried by $a_{ij,k}$ and $f_i$. Higher-order and delayed
interactions are not covered. For known $g$, the right-hand side
of~\eqref{eq:state} is \emph{affine in $\bA_k$}; this property belongs to
the physical model, and the lifting of Section~\ref{sec:lifting} is
designed to preserve it rather than to create it.
\end{remark}

\begin{remark}[Identifiability]
\label{rem:identifiability}
Three limitations are intrinsic to~\eqref{eq:state}. (a) \emph{Scale.}
$(a_{ij,k},g)$ and $(c\,a_{ij,k},g/c)$, $c\neq0$, generate the same
dynamics, so the weights are identifiable only up to the scale fixed by
$g$; the edge set is scale invariant. (b) \emph{Self-loops.} For
difference couplings, $g(x,x)\equiv0$ (e.g., $\sin(\theta_i-\theta_i)$,
$\tanh(\kappa(v_i-v_i))$), so a weight $a_{ii}$ would have no effect on
the data; it is therefore fixed to zero and excluded from estimation. For sender-only couplings, $a_{ii}$ is
confounded with $f_i$. (c) \emph{Operating regime.} Identifiability of
row $i$ of $\bA_k$ requires the regressors
$\{g(x_{i,t},x_{j,t})\}_{j}$ to be linearly independent along the data
window. If, e.g., a Kuramoto network phase-locks,
$\theta_{j,t}-\theta_{i,t}\equiv c_{ij}$, these regressors are constant
in $t$ and only $\sum_ja_{ij}\sin c_{ij}$ is identifiable from row $i$.
Identifiability is thus a property of the trajectory, which is why all
excitation constants below are evaluated along the data. What
matters for the estimator is the excitation of the measured regressors
$\bH_t\bU(\hat x_{t-1})$ (Remarks~\ref{rem:design}(b)
and~\ref{rem:reldeg}); uniqueness of a regularized estimate is not
identifiability.
\end{remark}

\subsection{Problem Statement}
\label{subsec:problem_statement}

\textbf{Information pattern.} The following information is assumed
throughout. (I1) The self-dynamics $f_i$ are \emph{unknown} and are
represented through a node dictionary $\bphi$; the interaction function
$g$ is known only through an interaction dictionary $\bchi$ that is
expected to represent it on the state coordinates. (I2) Calibration
trajectories with full-state records are available
(Section~\ref{subsec:pretrain}). In the main setting (P1) they were
generated under a \emph{known} reference topology $\bA^{\rm ref}$; the
setting (P2) with unknown $\bA^{\rm ref}$ requires the additional prior
information specified in Proposition~\ref{prop:calib}. (I3) A box of
admissible weights, which encodes the candidate edge set and known
absent edges, is given (Section~\ref{subsec:admm_update}). (I4) The
measurement maps $h_i$ and the sensing schedule are known. (I5) The
topology $\bA_k$ during operation may differ from $\bA^{\rm ref}$;
neither the operational data nor the operational topology are used for
calibration, dictionary selection or tuning. The experiments test
transfer across topologies, not across node dynamics.

\textbf{Offline (representation learning):} From the calibration data,
learn the lifted node dynamics $\bF_i$ and the interaction matrices
$\bB_i$ of the structured lifting introduced in
Section~\ref{sec:lifting}; the measurement matrices $\bH_i$ follow from
$h_i$ and $\bphi$ (Assumption~\ref{ass:dictionary}(A5)).

\textbf{Online (main problem):} Given the learned matrices and the
streaming measurements $\{z_k\}_{k\ge1}$, recursively compute estimates
$\hat x_k$ and $\hat\bA_k$ of the state and of the (possibly
time-varying) topology, at a per-step cost that is polynomial in $p$ and
in the lifted dimension, and quantify how model mismatch, noise,
topology variation and the coupling between the two estimators
determine the resulting errors.

%% ===============================================================
\section{Structure-Preserving Koopman Representation}
\label{sec:lifting}
%% ===============================================================

This section answers the first question raised in
Section~\ref{sec:intro}: for which nonlinear networks does a lifted
model exist whose coupling is linear in the \emph{physical} edge
weights, and how large is its error? We first show that the separable
construction cannot provide such a model, and then construct one. The construction is a structured, edge-affine surrogate of the
one-step map, not an exact finite-dimensional Koopman-invariant model
of the network; its purpose is to
combine a node model that can be learned from data with a coupling
that stays linear in the physical weights, and its error is
quantified.

\subsection{Why Separable Lifting Is Insufficient}
\label{subsec:separable_fails}

A separable lifting with the
sender-only coupling $g(x_i,x_j)=g(x_j)$ postulates
\begin{equation}
  \bpsi_i(x_{i,k+1})\approx\bF_i^\phi\bpsi_i(x_{i,k})
  +\sum_{j}a_{ij}\bGamma^\phi\bpsi_j(x_{j,k}),
  \label{eq:separable_model}
\end{equation}
and justifies it by \emph{local} conditions:
$\norm{\bpsi(f_i(x))-\bF_i^\phi\bpsi(x)}\le\varepsilon_f$ and
$\norm{\bpsi(g(x))-\bGamma^\phi\bpsi(x)}\le\varepsilon_g$. The network
residual of~\eqref{eq:separable_model}, however, is
$\bpsi(f_i(x_i)+\sum_ja_{ij}g(x_j))-\bF_i^\phi\bpsi(x_i)-\sum_ja_{ij}\bGamma^\phi\bpsi(x_j)$,
and $\bpsi$ is evaluated at a \emph{sum} of node contributions. The
local conditions do not control it.

\begin{proposition}[Local accuracy does not imply network accuracy]
\label{prop:separable_fails}
Let $n=1$, $p=2$, $f_1(x)=x$, $g(x_1,x_2)=x_2$, $a_{12}=a\neq0$, and let
the local dictionary contain the observable $x\mapsto x^2$, e.g.,
$\bpsi(x)=[x,x^2]^\top$. Then:
\begin{enumerate}[label=(\roman*),leftmargin=*]
\item the local conditions hold \emph{exactly}, $\varepsilon_f=\varepsilon_g=0$,
with $\bF_1^\phi=\bGamma^\phi=\bI_2$;
\item for every lifted model of node $1$ in which the one-step
prediction of the observable $x_1^2$ is an additive function
$u(x_{1,k})+v(x_{2,k})$ of the current states, and in particular for
\eqref{eq:separable_model} with any dictionary size $N$ and any matrices
$\bF_1^\phi,\bGamma^\phi$, the network residual $r$ of that observable
satisfies
\begin{equation}
  \sup_{(x_1,x_2)\in[-1,1]^2}|r(x_1,x_2)|\;\ge\;2|a|.
  \label{eq:separable_lower}
\end{equation}
\end{enumerate}
\end{proposition}
\begin{proof}
See Appendix~\ref{app_prop_sep}.
\end{proof}

Proposition~\ref{prop:separable_fails} shows that a uniform
network-level condition of the form
$\norm{\bpsi(f_i(x_i)+\sum_ja_{ij}g(x_j))-\bF_i^\phi\bpsi(x_i)-\sum_ja_{ij}\bGamma^\phi\bpsi(x_j)}\allowbreak\le\varepsilon_{\rm net}$
with small $\varepsilon_{\rm net}$ is \emph{not} implied by small
$\varepsilon_f,\varepsilon_g$, and cannot be enforced by enlarging the
node-wise dictionary: the obstruction is the cross term $2ax_1x_2$,
which is not additively separable. The
remedy is to represent the coupling through features of node
\emph{pairs}. Proposition~\ref{prop:separable_fails} concerns the structured
separable model~\eqref{eq:separable_model}, whose coupling is linear in
$\bA$ and acts through node-wise features. It motivates the construction below; it is not a statement about
Koopman-based network identification methods that use non-separable
observables~\cite{mauroy2020koopman,mei2022network,anantharaman2025koopman}. Nor does it apply to dictionaries without nonlinear observables:
for $\bphi(x)=x$ and linear $f_i,g$, the separable model is exact
(Corollary~\ref{cor:consistency}(ii)).

\subsection{Interaction-Aware Structured Lifting}
\label{subsec:decoupled}

Let $\bphi:\mathbb{R}^n\to\mathbb{R}^N$ be a node dictionary shared by
all nodes and $\bchi:\mathbb{R}^n\times\mathbb{R}^n\to\mathbb{R}^{N_e}$
an interaction dictionary. The lifted node state is
\begin{equation}
  \bpsi_{i,k} \;\triangleq\; \bphi(x_{i,k}) \;\in\;\mathbb{R}^N,
  \label{eq:lifted_def}
\end{equation}
and the \emph{interaction-aware structured lifting} models its
evolution as
\begin{equation}
  \bpsi_{i,k+1}=\bF_i\bpsi_{i,k}
  +\sum_{j=1}^pa_{ij,k}\,\bB_i\,\bchi(x_{i,k},x_{j,k})+\bzeta_{i,k},
  \label{eq:lifted_dynamics}
\end{equation}
with $\bF_i\in\mathbb{R}^{N\times N}$, $\bB_i\in\mathbb{R}^{N\times N_e}$,
and residual $\bzeta_{i,k}$. The coupling in~\eqref{eq:lifted_dynamics}
is linear in the physical weights $a_{ij,k}$ by construction. Writing
$\tilde x_{i,k+1}\triangleq f_i(x_{i,k})+\sum_ja_{ij,k}g(x_{i,k},x_{j,k})$
for the noise-free successor, the residual splits as
$\bzeta_{i,k}=\bdelta_{i,k}+\bxi_{i,k}$ with the \emph{representation
residual} and the \emph{lifted noise}
\begin{align}
  \bdelta_{i,k}&\triangleq\bphi(\tilde x_{i,k+1})-\bF_i\bphi(x_{i,k})
  -\sum_ja_{ij,k}\bB_i\bchi(x_{i,k},x_{j,k}),
  \label{eq:delta_def}\\
  \bxi_{i,k}&\triangleq\bphi(x_{i,k+1})-\bphi(\tilde x_{i,k+1}).
  \label{eq:xi_def}
\end{align}

\begin{assumption}[Regularity of the structured lifting]
\label{ass:dictionary}
\begin{enumerate}
\item[(A1)] \textit{State inclusion:} $\bphi(x)=[x^\top,\tilde\bphi(x)^\top]^\top$,
so that $\bPi\bphi(x)=x$ with $\bPi=[\bI_n\;\;\bzero_{n\times(N-n)}]$,
$\norm{\bPi}=1$.
\item[(A2)] \textit{Smoothness of the node dictionary:} $\bphi$ is
$C^1$, $L_\phi$-Lipschitz, and its Jacobian $\bJ_\phi$ is
$L_2$-Lipschitz: $\norm{\bJ_\phi(y)-\bJ_\phi(x)}\le L_2\norm{y-x}$.
\item[(A3)] \textit{Interaction dictionary:} $\bchi$ is bounded and
Lipschitz, $\norm{\bchi(x,y)}\le b_\chi$ and
$\norm{\bchi(x,y)-\bchi(x',y')}\le L_\chi(\norm{x-x'}+\norm{y-y'})$.
\item[(A4)] \textit{Common state block:} $\bPi\bB_i=\bB_x$ for all $i$,
i.e., the rows of $\bB_i$ that act on the state coordinates are shared\rev{,
and $\bB_x\neq\bO$. In the calibration (P1) $\bB_x$ is a decision
variable; in (P2) it is fixed a priori
(Section~\ref{subsec:pretrain})}.
\item[(A5)] \textit{Measurement representability:} $h_i(x)=\bH_i\bphi(x)$
for some $\bH_i\in\mathbb{R}^{q\times N}$ (e.g., $\bH_i=\bC_i\bPi$ for
linear $h_i(x)=\bC_ix$).
\end{enumerate}
\end{assumption}

\begin{remark}
(A1)--(A3) are properties of user-chosen dictionaries and can be
verified analytically; they do \emph{not} postulate any network-level
accuracy. Bounded, Lipschitz interaction features include
trigonometric, sigmoidal, Hill-type and radial basis functions. (A4)
reflects the fact, proved in Theorem~\ref{thm:lifting}(ii), that the
state block of the ideal interaction term equals $g$ for every node. A \emph{fixed} nonzero $\bB_x$ removes the common scale ambiguity
of the weights (Remark~\ref{rem:identifiability}(a)); a $\bB_x$ that is
fitted jointly with $\bA^{\rm ref}$ does not, and neither condition
alone makes the calibration identifiable
(Proposition~\ref{prop:calib}). If (A5) holds only
approximately, the mismatch is added to $v_{i,k}$. The global Lipschitz requirement in (A2) excludes polynomial
dictionaries on $\mathbb{R}^n$; such dictionaries satisfy (A2) on bounded
sets. Lemma~\ref{lem:localization} transfers bounds on the
representation residual from bounded sets to mean-square bounds, but it
does not control the regressor constants $L_\phi$, $b_\chi$ and $L_\chi$,
which enter Section~\ref{sec:theory} through $c_e$, $\ell_U$ and $b_U$.
For polynomial dictionaries, the results of Section~\ref{sec:theory}
therefore apply only while the true and the estimated states remain in
a bounded set on which these constants are evaluated, which the
algorithm does not enforce. Theorem~\ref{thm:lifting} and the
bound~\eqref{eq:xi_mean} use only (A1) and the Lipschitz continuity of
$\bJ_\phi$, which $\bphi(x)=[x,x^2]^\top$ satisfies globally with
$L_2=2$.
\end{remark}

For given $(\bF_i,\bB_i)$ define the \emph{local Koopman error} and the
\emph{interaction-feature error}
\begin{align}
  \varepsilon_{f,i}(x_i)&\triangleq\norm{\bphi(f_i(x_i))-\bF_i\bphi(x_i)},
  \label{eq:eps_f}\\
  \varepsilon_{\chi,i}(x_i,x_j)&\triangleq\norm{\bJ_\phi(f_i(x_i))\,g(x_i,x_j)
  -\bB_i\bchi(x_i,x_j)}.
  \label{eq:eps_chi}
\end{align}
The function $\bJ_\phi(f_i(x_i))g(x_i,x_j)$ in~\eqref{eq:eps_chi} is the
\emph{ideal interaction term}: it is the first-order effect of the
coupling on the lifted coordinates, and it depends on the receiver
state $x_i$ even when $g$ does not.

\subsection{Network-Level Residual Bound}
\label{subsec:approx_linear}

\begin{theorem}[Constructive network-level representation]
\label{thm:lifting}
Let Assumption~\ref{ass:dictionary}(A1) hold, let $\bphi$ be $C^1$
with $L_2$-Lipschitz Jacobian $\bJ_\phi$, and let $\bA$ and
$(\bF_i,\bB_i)$ be arbitrary. For every network state
$x=\mathrm{col}\{x_i\}$, the representation residual~\eqref{eq:delta_def}
satisfies:
\begin{enumerate}[label=(\roman*),leftmargin=*]
\item \emph{(Residual decomposition)}
\begin{align}
  \norm{\bdelta_i(x)}\le\;&\varepsilon_{f,i}(x_i)
  +\sum_{j}|a_{ij}|\,\varepsilon_{\chi,i}(x_i,x_j)\nonumber\\
  &+\frac{L_2}{2}\Bigl(\sum_j|a_{ij}|\,\norm{g(x_i,x_j)}\Bigr)^2.
  \label{eq:lifted_noise_bound}
\end{align}
\item \emph{(Exactness of the state coordinates)}
\begin{align}
  \bPi\bdelta_i(x)=\;&\bigl[f_i(x_i)-\bPi\bF_i\bphi(x_i)\bigr]\nonumber\\
  &+\sum_ja_{ij}\bigl[g(x_i,x_j)-\bPi\bB_i\bchi(x_i,x_j)\bigr],
  \label{eq:state_block}
\end{align}
which is affine in $\bA$ and contains no second-order term. Consequently, for $p\ge2$, $\bPi\bdelta_i(x)=\bzero$ for all $x$ and
all $\bA$ if and only if $\bPi\bF_i\bphi\equiv f_i$ and
$\bPi\bB_i\bchi\equiv g$. Matrices with this property exist if and only
if $f_i\in\mathrm{span}\{\bphi\}$ and $g\in\mathrm{span}\{\bchi\}$
componentwise; the span condition alone does not imply that given
(e.g., learned) matrices have it.
\item \emph{(Tightness)} For the system of
Proposition~\ref{prop:separable_fails} with
$\bchi(x_1,x_2)=[x_2,\,x_1x_2]^\top$ and
$\bB_1=\mathrm{diag}(1,2)$, one has $\varepsilon_{f,1}=\varepsilon_{\chi,1}=0$,
$L_2=2$, and~\eqref{eq:lifted_noise_bound} holds with equality,
$\bdelta_1(x)=[0,\,a^2x_2^2]^\top$.
\end{enumerate}
\end{theorem}
\begin{proof}
See Appendix~\ref{app_thm1}.
\end{proof}

\begin{corollary}[Uniform residual and special cases]
\label{cor:special}
Suppose that, on a set $\calX$ containing the states of interest,
$\varepsilon_{f,i}\le\varepsilon_f$, $\varepsilon_{\chi,i}\le\varepsilon_\chi$
and $\norm{g}\le g_\infty$. Then, for all $x\in\calX^p$,
\begin{equation}
  \norm{\bdelta_i(x)}\le\varepsilon_f+d\,a_{\max}\varepsilon_\chi
  +\tfrac{L_2}{2}\,(d\,a_{\max}g_\infty)^2.
  \label{eq:uniform_residual}
\end{equation}
Moreover, in the sender-only case $g(x_i,x_j)=g(x_j)$ with
$\bchi(x_i,x_j)=\bphi(x_j)$ and $\bB_i=\bGamma$, the coupling term
of~\eqref{eq:lifted_dynamics} is $[(\bA\otimes\bGamma)\bPsi]_i$ with
$\bPsi=\mathrm{col}\{\bpsi_i\}$, and the interaction-feature error is
$\varepsilon_{\chi,i}=\norm{\bJ_\phi(f_i(x_i))g(x_j)-\bGamma\bphi(x_j)}$,
which in general depends on the receiver state through
$\bJ_\phi(f_i(x_i))$ while $\bGamma\bphi(x_j)$ does not; it vanishes
identically, e.g., when $\bphi$ is affine and $g$ is linear, but not
in general.
\end{corollary}

Theorem~\ref{thm:lifting} bounds the network-level residual by
quantities that are computable for known $f_i,g$ and controllable
through the dictionaries:
$\varepsilon_{f,i}$ is a single-node Koopman error, $\varepsilon_{\chi,i}$
is an approximation error of a \emph{known} function of $(x_i,x_j)$ once
$f_i,g$ are modeled, and the last term is of second order in the
coupling strength and vanishes on the state coordinates. For
unknown $f_i$, validation residuals are
only empirical proxies for them. Choosing
$\bchi$ to contain products $u_\ell(x_i)v_m(x_j)$ of receiver and
sender features gives a tensor-product dictionary, for which
$\varepsilon_\chi$ is a standard bivariate approximation error.

\begin{remark}[First-order nature of the lifting]
\label{rem:first_order}
The interaction-aware lifting is a first-order (in the coupling)
representation of the lifted one-step map, not a Koopman-invariant
model of the network: the remainder
$\frac{L_2}{2}(\sum_j|a_{ij}|\norm{g})^2$ in~\eqref{eq:lifted_noise_bound}
is attained and cannot be removed by
any pairwise interaction dictionary when a node has two or more
in-neighbours. For an Euler
discretization of a continuous-time network with step $\Delta t$, one
has $g=\mathcal{O}(\Delta t)$, so the remainder is $\mathcal{O}(\Delta t^2)$,
of the same order as the discretization error. It enters the
guarantees through $\varepsilon_{\rm rep}$ in~\eqref{eq:fx}; with linear
sensors it affects neither the physical estimate nor the topology bound, and it matters for sensors of
nonlinear features under coarse sampling or strong coupling. Removing it requires features that are
quadratic in the weights and is outside the scope of this paper.
\end{remark}

\begin{example}[Constructions for the benchmark models]
\label{ex:constructions}
\emph{Kuramoto network.} With the Euler step,
$\theta_{i,k+1}=\theta_{i,k}+\omega_i\Delta t+\sum_ja_{ij}g(\theta_i,\theta_j)+w_{i,k}$
and $g(\theta_i,\theta_j)=\frac{K\Delta t}{p}\sin(\theta_j-\theta_i)$, where $\theta_{i,k}\in\mathbb{R}$ is the unwrapped phase. Take $\bphi(\theta)=[\theta,1,\cos\theta,\sin\theta]^\top$. Then
$\bphi(f_i(\theta))=\bF_i\bphi(\theta)$ exactly, with $\bF_i$ containing
the rotation by $\omega_i\Delta t$, so $\varepsilon_{f,i}=0$. Since
$\sin(\theta_j-\theta_i)=\sin\theta_j\cos\theta_i-\cos\theta_j\sin\theta_i$
and the products $\sin(\theta_i+c)\cos\theta_i$,
$\sin(\theta_i+c)\sin\theta_i$, $\cos(\theta_i+c)\cos\theta_i$,
$\cos(\theta_i+c)\sin\theta_i$ lie in
$\mathrm{span}\{1,\cos2\theta_i,\sin2\theta_i\}$, the choice
$\bchi=\{\sin\theta_j,\cos\theta_j\}\otimes\{1,\cos\theta_i,\sin\theta_i,\cos2\theta_i,\sin2\theta_i\}$
($N_e=10$) yields $\varepsilon_{\chi,i}=0$ for node-specific $\bB_i$
with a common state row. Here $L_2=1$ and
$|g|\le K\Delta t/p$, so~\eqref{eq:lifted_noise_bound} gives
$\norm{\bdelta_i}\le\frac12(d\,a_{\max}K\Delta t/p)^2$, and the state
coordinate admits an exact representation, attained by
$\bPi\bF_i\bphi(\theta)=\theta+\omega_i\Delta t$ and $\bPi\bB_i\bchi=g$.
The interaction \emph{function} can be represented exactly, but the lifted dynamics are not exactly closed:
the remainder is the second-order coupling term. Two consequences are used later. First, $g(\theta,\theta)=0$
implies $\bB_i\bchi(\theta,\theta)=\bJ_\phi(f_i(\theta))g(\theta,\theta)=\bzero$,
so the regressor column of $a_{ii}$ vanishes identically and $a_{ii}$
must be excluded from estimation.
Second, if only $\theta$ is measured, the coordinates $(\cos\theta,\sin\theta)$
evolve under a rotation that the sensor does not see; this is the
situation of Proposition~\ref{prop:undetectable}.

\emph{Car-following and Hill networks.} Including
$\tanh(\kappa(v_j-v_i))$, respectively $x_j^h/(\vartheta^h+x_j^h)$, as
a component of $\bchi$ makes an exact state
block~\eqref{eq:state_block} attainable; the remaining rows are approximated by bivariate features, and
their residual is quantified by~\eqref{eq:lifted_noise_bound}. In all three benchmarks the self-dynamics $f_i$ are affine; the
significance of this fact is clarified by Proposition~\ref{prop:ekf}. Non-affine self-dynamics and a sensor of nonlinear features are
considered in Section~\ref{subsec:controlled}.
\end{example}

\subsection{Lifted Noise and the Operating Domain}
\label{subsec:lifted_noise}

Let $\calF_k$ denote the $\sigma$-algebra generated by $x_0$,
$\{w_t\}_{t<k}$, $\{v_t\}_{t\le k}$ and the initialization of the
estimator (conditionally on the sensing schedule and the topology
sequence, Section~\ref{subsec:model}), i.e., the history of the system and of the estimator up
to time $k$. It contains the latent states and is therefore larger than
the information $\calZ_k\triangleq\sigma(z_1,\ldots,z_k)$ available to
the estimator. Then
$x_k$, $\tilde x_{k+1}$ and every quantity computed by the estimator up
to time $k$ are $\calF_k$-measurable, while $w_k$ is independent of
$\calF_k$.

\begin{proposition}[Moments of the lifted noise]
\label{prop:lifted_noise}
Under Assumption~\ref{ass:dictionary}(A1)--(A2), the lifted
noise~\eqref{eq:xi_def} satisfies
\begin{align}
  \bigl\|\expect[\bxi_{i,k}\mid\calF_k]\bigr\|&\le\tfrac{L_2}{2}\tr(\bQ_{i,k}),
  \label{eq:xi_mean}\\
  \expect\bigl[\norm{\bxi_{i,k}}^2\mid\calF_k\bigr]&\le L_\phi^2\tr(\bQ_{i,k}),
  \label{eq:xi_var}
\end{align}
and $\bPi\bxi_{i,k}=w_{i,k}$. The bound~\eqref{eq:xi_mean} is tight: for $n=1$, $\bphi(x)=[x,x^2]^\top$ ($L_2=2$),
$w\sim\mathcal{N}(0,\sigma^2)$ and $\expect|\tilde x_{i,k+1}|<\infty$,
$\expect[\bxi_{i,k}\mid\calF_k]=[0,\sigma^2]^\top$, and the second
component $\xi=2\tilde x w+w^2\ge-\tilde x^2$ is, conditionally on
$\calF_k$, bounded from below and hence not Gaussian.
\end{proposition}
\begin{proof}
See Appendix~\ref{app_noise}.
\end{proof}

Consequently, the lifted model~\eqref{eq:lifted_dynamics} is \emph{not}
a linear-Gaussian state-space model, even for Gaussian $w_{i,k}$: its
residual has a state-dependent conditional bias of order
$L_2\tr(\bQ_{i,k})$ and is correlated with the state through
$\bdelta_{i,k}$. We therefore use the Kalman recursion as a linear
recursive estimator and make no MMSE-optimality claim; the analysis of
Section~\ref{sec:theory} uses only the moment bounds above. Because $\bPi\bxi_{i,k}=w_{i,k}$ has zero conditional mean, the bias
affects only non-state coordinates; with linear sensors it does not
enter the topology estimate.

Uniform approximation guarantees are typically available only on a
bounded set, whereas Gaussian noise has unbounded support. The next
lemma converts such guarantees into the mean-square bounds used in
Section~\ref{sec:theory}.

\begin{lemma}[Localization]
\label{lem:localization}
Let $\bdelta:\mathbb{R}^{np}\to\mathbb{R}^{N}$ be measurable with
$\norm{\bdelta(x)}\le\varepsilon$ on $\calX_R\triangleq\{\norm{x}\le R\}$
and $\norm{\bdelta(x)}\le c_0(1+\norm{x}^m)$ for all $x$. If
$\expect\norm{x}^{4m}<\infty$, then
\begin{equation}
  \Lnorm{\bdelta(x)}\le\varepsilon
  +c_0\bigl(\expect(1+\norm{x}^m)^4\bigr)^{1/4}
  \prob(\norm{x}>R)^{1/4},
  \label{eq:localization}
\end{equation}
and $\prob(\norm{x}>R)\le\expect\norm{x}^2/R^2$.
\end{lemma}
\begin{proof}
See Appendix~\ref{app_noise}.
\end{proof}

\subsection{Global Lifted Model and Structural Correspondence}
\label{subsec:compact}

Stack $\bPsi_k\triangleq\mathrm{col}\{\bpsi_{i,k}\}\in\mathbb{R}^{Np}$,
$\bL\triangleq\blkdiag(\bF_1,\ldots,\bF_p)$,
$\bH\triangleq\blkdiag(\bH_1,\ldots,\bH_p)$, and let
$\ba_{i,k}\in\mathbb{R}^p$ be the transpose of the $i$-th row of $\bA_k$
and $\ba_k\triangleq\mathrm{col}\{\ba_{i,k}\}=\vecop(\bA_k^\top)\in\mathbb{R}^{p^2}$.
\oldrev{We also write $\Pip\triangleq\bI_p\otimes\bPi$ (so that
$x=\Pip\bPsi$), $\bPhi(x)\triangleq\mathrm{col}\{\bphi(x_i)\}$ and
$\bJPhi(x)\triangleq\blkdiag(\bJ_\phi(x_1),\ldots,\bJ_\phi(x_p))$, the
Jacobian of $\bPhi$; by Assumption~\ref{ass:dictionary}(A2), $\bPhi$ is $L_\phi$-Lipschitz. At time $k$,
$\bH_k$ collects the rows of $\bH$ of the measured nodes.}
Define the \emph{interaction regressor}
\begin{align}
  \bU_i(x)&\triangleq\bigl[\bB_i\bchi(x_i,x_1),\ldots,\bB_i\bchi(x_i,x_p)\bigr]
  \in\mathbb{R}^{N\times p},\nonumber\\
  \bU(x)&\triangleq\blkdiag\bigl(\bU_1(x),\ldots,\bU_p(x)\bigr)
  \in\mathbb{R}^{Np\times p^2}.
  \label{eq:regressor}
\end{align}
Then $\sum_ja_{ij,k}\bB_i\bchi(x_{i,k},x_{j,k})=\bU_i(x_k)\ba_{i,k}$ and,
with $\bzeta_k=\bdelta_k+\bxi_k$ stacked,
\begin{align}
  \bPsi_{k+1} &= \bL\,\bPsi_k + \bU(x_k)\,\ba_k + \bzeta_k,
  \label{eq:global_lifted_state}\\
  z_k         &= \oldrev{\bH_k}\,\bPsi_k + v_k,\ x_k=\Pip\bPsi_k,
  \,\oldrev{\bPsi_k=\bPhi(x_k)}.
  \label{eq:global_lifted_meas}
\end{align}
In the sender-only special case of Corollary~\ref{cor:special},
$\bU(x_k)\ba_k=(\bA_k\otimes\bGamma)\bPsi_k$, and
\eqref{eq:global_lifted_state}--\eqref{eq:global_lifted_meas} become the
linear graphical model of~\cite{fang2025joint} in the lifted variable.
In general, the coupling is a known nonlinear function of the state that
is \emph{linear in $\ba_k$}; this is the structural property exploited
below.

When a lifted coupling matrix is estimated without structure, edges
must be read off its blocks. The next lemma states exactly what is,
and what is not, guaranteed in this case.

\begin{lemma}[Exact and robust structural correspondence]
\label{lem:homomorphism}
Let $\bGamma\in\mathbb{R}^{N\times N}$ with $\bGamma\neq\bO$
(invertibility is not needed) and $\bT=\bA\otimes\bGamma$.
\begin{enumerate}[label=(\roman*),leftmargin=*]
\item \emph{(Exact)} $[\bT]_{(i,j)}=a_{ij}\bGamma$,
$[\bT]_{(i,j)}=\bO\iff a_{ij}=0$, and
$\normF{[\bT]_{(i,j)}}=|a_{ij}|\normF{\bGamma}$; hence
$\sum_{i,j}\normF{[\bT]_{(i,j)}}=\normF{\bGamma}\normone{\bA}$.
\item \emph{(Robust)} If $\hat\bT$ satisfies
$\normF{[\hat\bT]_{(i,j)}-a_{ij}\bGamma}\le\delta$ for all $(i,j)$, then
$\bigl|\normF{[\hat\bT]_{(i,j)}}-|a_{ij}|\normF{\bGamma}\bigr|\le\delta$.
If moreover $a_{\min}\normF{\bGamma}>2\delta$, then for every
$\tau\in(\delta,\,a_{\min}\normF{\bGamma}-\delta)$ the thresholded set
$\hat\calE_\tau\triangleq\{(i,j):\normF{[\hat\bT]_{(i,j)}}>\tau\}$
equals $\calE$.
\item \emph{(Kronecker projection)}
$\hat a_{ij}\triangleq\inner{[\hat\bT]_{(i,j)}}{\bGamma}_{\rm F}/\normF{\bGamma}^2$
minimizes $\normF{[\hat\bT]_{(i,j)}-a\bGamma}$ over $a\in\mathbb{R}$ and
satisfies $|\hat a_{ij}-a_{ij}|\le\normF{[\hat\bT]_{(i,j)}-a_{ij}\bGamma}/\normF{\bGamma}$.
\item \emph{(Block constraints do not enforce structure)} The block-row
constraint $\sum_j[\hat\bT]_{(i,j)}=\bGamma$ does not imply
$[\hat\bT]_{(i,j)}\in\mathrm{span}\{\bGamma\}$: for $\bGamma=\bI_2$,
$[\hat\bT]_{(1,1)}=\diag(1,0)$ and $[\hat\bT]_{(1,2)}=\diag(0,1)$ satisfy
it, while neither block is a multiple of $\bGamma$.
\end{enumerate}
\end{lemma}
\begin{proof}
See Appendix~\ref{app_lemma1}.
\end{proof}

\begin{remark}
Lemma~\ref{lem:homomorphism}(i) shows that, under exact Kronecker
structure, a block-Frobenius (group) penalty on $\bT$ equals
$\normF{\bGamma}\normone{\bA}$; it is a rescaled $\ell_1$ penalty and
brings no additional sparsity-inducing power. Part~(iv) shows that an
optimizer over a generic $\bT$ with block-row constraints returns, in
general, a matrix that is \emph{not} of the form $\bA\otimes\bGamma$,
so reading $\hat a_{ij}$ from a single entry of $[\hat\bT]_{(i,j)}$ is
not justified; part~(iii) gives the principled projection, and part~(ii)
the margin condition under which edges are still recovered. These
observations motivate estimating $\bA$ directly in the physical
parameter space, which also reduces
the number of unknowns from $N^2p^2$ to at most $p(p-1)$.
\end{remark}

%% ===============================================================
\section{Koopman-GKFA: Algorithm Design}
\label{sec:algorithm}
%% ===============================================================

Koopman-GKFA consists of an offline representation-learning stage and
an online loop that, at every time step, performs one re-lifted filter
step with the current topology estimate and then re-estimates the
topology in the physical parameter space from a forgetting-factor
window of measurement prediction errors.

\subsection{Representation Learning Without Isolated-Node Data}
\label{subsec:pretrain}

Let $\calD_{\rm cal}$ be calibration trajectories of the network with
full-state records (e.g., a commissioning experiment, a high-fidelity
simulator, or historical logs), split into training and validation
parts, and let $\bA^{\rm ref}$ denote the topology under which they
were generated. Test-phase data and topologies are never used.

\emph{(P1) Reference-topology calibration.} If $\bA^{\rm ref}$ is known,
$(\bF_i,\bB_i)$ are obtained from the ridge-regularized least-squares
problem
\begin{align}
  \min_{\{\bF_i,\bB_i\}:\,\bPi\bB_i=\bB_x}\;
  &\sum_{k}\sum_i\Bigl\|\bphi(x_{i,k+1})-\bF_i\bphi(x_{i,k})\nonumber\\
  &\quad-\bB_i\sum_ja^{\rm ref}_{ij}\bchi(x_{i,k},x_{j,k})\Bigr\|^2\nonumber\\
  &+\varrho\sum_i\bigl(\normF{\bF_i}^2+\normF{\bB_i}^2\bigr),
  \label{eq:calibration}
\end{align}
which is convex and solved in closed form (the non-state rows decouple
over $i$; the shared state block $\bB_x$ is fitted jointly over all
nodes). Isolated-node data correspond to $\bA^{\rm ref}=\bO$ and
identify only $\bF_i$; in that case $\bB_i$ is obtained by regressing
the ideal interaction term $\bJ_\phi(\hat f_i(x_i))g(x_i,x_j)$, with
$\hat f_i\triangleq\bPi\bF_i\bphi$, on $\bchi(x_i,x_j)$ over sampled
pairs, which requires a model of $g$ but no coupled data.

\emph{(P2) Coupled data with unknown topology.} This setting requires
additional prior information. The state block $\bB_x$ is fixed a
priori by a representation $g=\bB_x\bchi$ of the interaction function
on the state coordinates; without it,
the weights are determined only up to a common factor that is selected
by the regularization. With $\bB_x$
fixed, the state rows of~\eqref{eq:calibration} are linear in
$(\bPi\bF_i,\ba^{\rm ref}_i)$, and $\bA^{\rm ref}$ is first estimated
from them by the full-state version of the topology regression of
Section~\ref{subsec:admm_update} with the penalty
$\lambda\sum_{i\neq j}|a^{\rm ref}_{ij}|$; $(\bF_i,\bB_i)$ are then
obtained from~\eqref{eq:calibration} with $\bPi\bB_i$ held at the
fixed $\bB_x$, so that the scale fixed in the first step is retained.
The joint objective can
subsequently be decreased by alternating the two convex block updates with $\bB_x$ fixed;
the objective is nonincreasing along the alternation, but convergence to
a global minimizer of this biconvex problem is not claimed.
Proposition~\ref{prop:calib}(ii) characterizes when the state rows
determine $\bA^{\rm ref}$.

\begin{proposition}[Scale and identifiability of the calibration]
\label{prop:calib}
Let $\calK_{\rm cal}$ index the calibration snapshots and let
$\mathcal{J}(\{\bF_i,\bB_i\},\bA)$ denote the data-fit term
of~\eqref{eq:calibration} with $\bA$ in place of $\bA^{\rm ref}$.
\begin{enumerate}[label=(\roman*),leftmargin=*]
\item \emph{(Scale)} For every $c\neq0$,
$\mathcal{J}(\{\bF_i,\bB_i/c\},c\bA)=\mathcal{J}(\{\bF_i,\bB_i\},\bA)$, and
$(\{\bB_i/c\},\bB_x/c)$ satisfies (A4) whenever $(\{\bB_i\},\bB_x)$
does. Consequently, if $\bB_x$ is a decision variable, then for
$\varrho,\lambda>0$ every minimizer of
$\mathcal{J}+\varrho\sum_i(\normF{\bF_i}^2+\normF{\bB_i}^2)+\lambda\sum_{i\neq j}|a_{ij}|$ over a set of admissible $\bA$ that is a cone (e.g., support and
sign constraints without magnitude bounds)
with $\bA\neq\bO$ satisfies
\begin{equation}
  \lambda\sum_{i\neq j}|a_{ij}|=2\varrho\sum_i\normF{\bB_i}^2,
  \label{eq:scale_balance}
\end{equation}
i.e., the scale of the weights is fixed by $\lambda/\varrho$ and not by
the data.
\item \emph{(Identifiability from the state rows)} Let $\bB_x\neq\bO$ be
fixed and suppose that the calibration states satisfy, for some
$(\bF_i^o,\bA^o)$ supported on $\calI$,
$x_{i,k+1}=\bPi\bF^o_i\bphi(x_{i,k})+\bB_x\sum_{j\in\calI_i}a^o_{ij}\bchi(x_{i,k},x_{j,k})$
for all $i$ and $k\in\calK_{\rm cal}$. Then $(\bPi\bF^o_i,\ba^o_i)$ is the unique
pair $(\bTheta,\ba)\in\mathbb{R}^{n\times N}\times\mathbb{R}^{|\calI_i|}$ for
which the state-row residual
$x_{i,k+1}-\bTheta\bphi(x_{i,k})-\bB_x\sum_{j\in\calI_i}a_j\bchi(x_{i,k},x_{j,k})$
vanishes for all $k\in\calK_{\rm cal}$ if and only if the linear map
% \begin{equation}
%   \mathcal{L}_i:(\bTheta,\ba)\mapsto\Bigl(\bTheta\bphi(x_{i,k})
%   +\bB_x\sum_{j\in\calI_i}a_j\bchi(x_{i,k},x_{j,k})\Bigr)_{k\in\calK_{\rm cal}}
%   \label{eq:calib_map}
% \end{equation}
\begin{equation}
  \resizebox{0.9\columnwidth}{!}{$
  \mathcal{L}_i:(\bTheta,\ba)\mapsto\Bigl(\bTheta\bphi(x_{i,k})
  +\bB_x\sum_{j\in\calI_i}a_j\bchi(x_{i,k},x_{j,k})\Bigr)_{k\in\calK_{\rm cal}}
  $}
  \label{eq:calib_map}
\end{equation}
is injective. In particular, identifiability fails whenever, for some
$j\in\calI_i$, the sequence $\bB_x\bchi(x_{i,k},x_{j,k})$ lies in the span
of the sequences $\{\bTheta\bphi(x_{i,k})\}$ and
$\{\bB_x\bchi(x_{i,k},x_{l,k})\}_{l\neq j}$.
\end{enumerate}
\end{proposition}
\begin{proof}
See Appendix~\ref{app_calib}.
\end{proof}

Proposition~\ref{prop:calib} makes the role of the prior information
explicit. Fixing $\bB_x$ removes the scale ambiguity but does not by
itself ensure identifiability, which is an excitation property of the
calibration trajectory: for the Kuramoto dictionary of
Example~\ref{ex:constructions}, which contains the constant observable,
a phase-locked calibration run makes
$\sin(\theta_{j,k}-\theta_{i,k})$ constant in $k$, and $\mathcal{L}_i$ is
not injective. Once
$\bA^{\rm ref}$ is fixed, the ridge term makes~\eqref{eq:calibration}
strictly convex, so $(\bF_i,\bB_i)$ is unique; this uniqueness is
produced by the regularization and is not an identifiability statement
for the non-state rows. For these reasons, (P1) is the main setting of
the analysis, and (P2) is an extension under the conditions of
Proposition~\ref{prop:calib}(ii).

The quantity entering all guarantees is the residual of the matrices
actually used online. It is estimated on the validation part by
\begin{equation}
  \hat\varepsilon_{\rm val}\triangleq\Bigl[\tfrac{1}{M_{\rm val}}
  \sum_{\ell}\bigl\|\bPhi(x_{\ell+1})-\bL\bPhi(x_\ell)
  -\bU(x_\ell)\ba^{\rm ref}\bigr\|^2\Bigr]^{1/2},
  \label{eq:val_residual}
\end{equation}
with $\bPhi(x)\triangleq\mathrm{col}\{\bphi(x_i)\}$. This is an
empirical mean-square estimate of $\Lnorm{\bzeta_k}$ along the
calibration distribution, \emph{not} a uniform bound, and it mixes
model mismatch with process noise. On the state coordinates the two can
be separated: by Proposition~\ref{prop:lifted_noise}, the state block
of the summand is $\Pip\bzeta_\ell=\Pip\bdelta_\ell+w_\ell$, where
$\Pip\bdelta_\ell$ is a function of $x_\ell$ and $w_\ell$ is zero-mean and
independent of $x_\ell$, so that the cross term has zero mean and
\begin{equation}
  \expect\norm{\Pip\bzeta_\ell}^2=\expect\norm{\Pip\bdelta_\ell}^2
  +\expect\,\tr(\bQ_\ell)
  \label{eq:val_split}
\end{equation}
whenever the left-hand side is finite. Hence the state-block validation
residual $\hat\varepsilon_{{\rm val},x}$ (with $\Pip$ applied inside the
norm in~\eqref{eq:val_residual}) measures the combined prediction error
due to model mismatch and process noise, and, if $\tr(\bQ_\ell)$ is known,
$(\hat\varepsilon_{{\rm val},x}^2-\tr\bar\bQ)_+^{1/2}$ is a de-noised
estimate of the mean-square state-block residual \emph{on the
calibration distribution}. Neither quantity bounds the deployment
constant $\varepsilon_x$ of Theorem~\ref{thm:topology_error}, which is
a supremum over the operational trajectory; both are used for
calibration, for setting the filter covariance $\bar\bQ$ below and as
empirical diagnostics. No decomposition of the
type~\eqref{eq:val_split} holds on the non-state coordinates, where the
lifted noise has a conditional bias (Proposition~\ref{prop:lifted_noise}).
If the calibration records are themselves noisy, their measurement noise
adds a further term.

\subsection{Re-Lifted State Estimation}
\label{subsec:kf_update}

The linear lifted dynamics~\eqref{eq:global_lifted_state} suggest running
a Kalman filter directly on $\bPsi_k$. The next result shows that this
is not a certifiable design whenever the sensors do not see a
non-decaying lifted mode.

\begin{proposition}[Undetectable lifted modes cannot be stabilized]
\label{prop:undetectable}
Let $\bL\in\mathbb{R}^{Np\times Np}$ and $\bH\in\mathbb{R}^{m\times Np}$,
and suppose that $\bL\bv=\lambda\bv$ and $\bH\bv=\bzero$ for some
$\bv\in\mathbb{C}^{Np}\setminus\{\bzero\}$ and $\lambda\in\mathbb{C}$ with
$|\lambda|\ge1$. Then:
\begin{enumerate}[label=(\roman*),leftmargin=*]
\item for every $\bK\in\mathbb{R}^{Np\times m}$, $\lambda$ is an
eigenvalue of $(\bI-\bK\bH)\bL$; hence $\rho((\bI-\bK\bH)\bL)\ge1$ and
$\norm{(\bI-\bK\bH)\bL}\ge1$ in every induced matrix norm;
\item for every gain sequence $\{\bK_k\}$, the homogeneous error
recursion $\be_k=(\bI-\bK_k\bH)\bL\be_{k-1}$ has a real solution that does
not converge to zero.
\end{enumerate}
For the Kuramoto lifting of Example~\ref{ex:constructions} with
$\bH_i=[1,0,0,0]$, the vector $\bv_i=[0,0,1,-\imath]^\top$
($\imath^2=-1$) satisfies $\bF_i\bv_i=e^{\imath\omega_i\Delta t}\bv_i$ and
$\bH_i\bv_i=0$, so (i)--(ii) apply with $|\lambda|=1$.
\end{proposition}
\begin{proof}
See Appendix~\ref{app_filter}.
\end{proof}

In a lifted Kalman filter the non-state coordinates along such modes
are never corrected by the data; they drift away from
$\bphi(\hat x_k)$ and, if used for topology inference, contaminate it.
Koopman-GKFA therefore keeps the lifted coordinates \emph{consistent}
with the state estimate: after each update the posterior is mapped back
to the lifted manifold $\{\bPhi(x)\}$. Given $\hat x_{k-1}$, its
covariance $\bP_{k-1}\in\mathbb{R}^{np\times np}$ and $\hat\ba_{k-1}$,
\begin{align}
  \hat{\bPsi}_{k|k-1}
  &= \bL\bPhi(\hat x_{k-1})+\bU(\hat x_{k-1})\,\hat\ba_{k-1},
  \label{eq:kf_pred_state}\\
  \bP_{k|k-1}
  &= \hat\bM_{k-1}\bP_{k-1}\hat\bM_{k-1}^\top + \bar{\bQ}_{k-1},
  \label{eq:kf_pred_cov}\\
  \bK_k
  &= \bP_{k|k-1}\bH_k^\top
     \bigl(\bH_k\bP_{k|k-1}\bH_k^\top+\bR_k\bigr)^{-1},
  \label{eq:kalman_gain}\\
  \hat x_k
  &= \Pip\bigl[\hat{\bPsi}_{k|k-1}
    + \bK_k\bigl(z_k - \bH_k\hat{\bPsi}_{k|k-1}\bigr)\bigr],
  \label{eq:kf_update_state}\\
  \bP_k
  &= \Pip\bigl[(\bI - \bK_k\bH_k)\bP_{k|k-1}(\bI - \bK_k\bH_k)^\top\nonumber\\
  &\quad+\bK_k\bR_k\bK_k^\top\bigr]\Pip^\top,
  \label{eq:kf_update_cov}\\
  \hat\bPsi_k&\triangleq\bPhi(\hat x_k),
  \label{eq:relift}
\end{align}
where
$\hat\bM_{k-1}\triangleq\bL\bJPhi(\hat x_{k-1})+\partial_x\bigl[\bU(x)\hat\ba_{k-1}\bigr]_{x=\hat x_{k-1}}\in\mathbb{R}^{Np\times np}$
is the Jacobian of the lifted one-step map $x\mapsto\bL\bPhi(x)+\bU(x)\hat\ba_{k-1}$,
and $\bar\bQ_{k-1}$ is a design covariance calibrated
from~\eqref{eq:val_residual} and Proposition~\ref{prop:lifted_noise}. The
prior $\bP_{k|k-1}$ contains the cross-covariances between the state
and the nonlinear features, which the update exploits when the sensors
measure features, $h_i=\bH_i\bphi$. If $\bchi$ is not differentiable, any
bounded surrogate of $\hat\bM_{k-1}$ may be used: the analysis of
Section~\ref{sec:theory} holds for every gain sequence satisfying its
assumptions, irrespective of how the gain is computed. Partial
observation is handled by the time-varying $\bH_k$. Because the lifted
model is not linear-Gaussian, the
recursion is a recursive estimator, not an MMSE estimator.

The next result makes precise what the lifting does, and does not,
contribute to state estimation.

\begin{proposition}[Reduction to an EKF for linear sensors]
\label{prop:ekf}
Suppose $\bH_k=\bC_k\Pip$ with $\bC_k$ block diagonal, and define the
surrogate physical map
\begin{equation}
  \mathfrak{f}(x;\ba)\triangleq\Pip\bigl[\bL\bPhi(x)+\bU(x)\ba\bigr],
  \label{eq:surrogate}
\end{equation}
i.e., $\mathfrak{f}_i(x;\ba)=\bPi\bF_i\bphi(x_i)+\sum_ja_{ij}\bPi\bB_i\bchi(x_i,x_j)$.
Then~\eqref{eq:kf_pred_state}--\eqref{eq:relift} produce exactly the
estimates $\hat x_k$ and covariances $\bP_k$ of the EKF (in Joseph form)
for the model $x_k=\mathfrak{f}(x_{k-1};\hat\ba_{k-1})+\bar w_{k-1}$,
$z_k=\bC_kx_k+v_k$, with $\mathrm{Cov}(\bar w_{k-1})=\Pip\bar\bQ_{k-1}\Pip^\top$,
linearized at $\hat x_{k-1}$. If, in addition, $\Pip\bdelta\equiv\bzero$
(Theorem~\ref{thm:lifting}(ii)), then
$\mathfrak{f}_i(x;\ba)=f_i(x_i)+\sum_ja_{ij}g(x_i,x_j)$.
\end{proposition}
\begin{proof}
See Appendix~\ref{app_filter}.
\end{proof}

\begin{remark}[What the lifting contributes]
\label{rem:lifting_role}
(a) Linearity of the coupling in $\bA$ is inherited from~\eqref{eq:state};
the lifting preserves it but does not create
it. Consequently, a filter that operates only in the physical space can
also be combined with a convex topology regression, and the
\emph{convexity} of the topology step is not an advantage specific to
lifting. (b) With linear sensors and known affine $f_i$, as in the
three benchmark families of Example~\ref{ex:constructions},
Proposition~\ref{prop:ekf} shows that the re-lifted filter is an EKF of
the physical model; any difference with respect to an EKF-based joint
estimator must then come from the topology step or from tuning. A meaningful assessment of the lifting therefore keeps the
calibration data, the candidate edge set and the topology estimator
fixed and varies only the state-space model. Such an assessment is reported in Section~\ref{subsec:controlled}. (c) The lifting
matters, first, as a structured data-driven model class: when $f_i$ is
unknown, $\bPi\bF_i\bphi$ and the interaction matrices are learned by a
convex EDMD-type regression in which the coupling is constrained to
remain affine in the physical weights, a property that a generic
black-box surrogate need not have, and no analytic $f_i$ or Jacobian is
needed. Second, it matters when the sensors measure nonlinear features:
the predicted measurement $\bH_k\hat\bPsi_{k|k-1}$ is then propagated
through the lifted model rather than obtained by linearizing $h_i$ at
the prior mean, which is a different first-order approximation; neither
dominates the other in general. Whether the learned lifted surrogate yields smaller errors than another
surrogate learned from the same data is an empirical question that the
theory does not settle; for the system of
Section~\ref{subsec:controlled}, it does not.
(d) A lifted filter without~\eqref{eq:relift} would propagate
$\expect[\bphi(x)]$ rather than $\bphi(\expect[x])$, which can be
advantageous for nonlinear $f_i$, but it is not certifiable when
Proposition~\ref{prop:undetectable} applies. Re-lifting trades this
potential advantage for a closed-loop guarantee.
\end{remark}

\subsection{Prediction-Error Topology Inference in the Physical Parameter Space}
\label{subsec:admm_update}

\emph{Free coordinates.} Physically justified constraints are encoded
by a box $\calA=\{\ba:\,l_{ij}\le a_{ij}\le u_{ij}\}$ with
$l_{ii}=u_{ii}=0$, which covers the absence of self-loops, nonnegative
weights, known absent edges ($l_{ij}=u_{ij}=0$) and magnitude bounds.
Let $\calI\triangleq\{(i,j):l_{ij}<u_{ij}\}$ be the set of \emph{free}
coordinates, $\calI_i\triangleq\{j:(i,j)\in\calI\}$, $p_\calI\triangleq|\calI|$
and $\bar p\triangleq\max_i|\calI_i|\le p-1$. Entries outside $\calI$ are
known; without loss of generality they are zero (otherwise their
known contribution is subtracted from the data below). Henceforth
$\ba_k\in\mathbb{R}^{p_\calI}$ denotes the vector of free weights,
$\bU(x)\in\mathbb{R}^{Np\times p_\calI}$ the corresponding columns
of~\eqref{eq:regressor}, and $\calA$ the box restricted to $\calI$. In
particular, the diagonal, whose regressor column vanishes for
difference couplings (Example~\ref{ex:constructions}), is never
estimated.

\emph{Why not regress filtered lifted increments.} A natural regression
uses $\hat\br_t\triangleq\hat\bPsi_t-\bL\hat\bPsi_{t-1}$ as regressand. For a
lifted Kalman filter, $\hat\br_t=\bU(\hat x_{t-1})\hat\ba_{t-1}+\bK_t\bm\iota_t$
with the innovation $\bm\iota_t$: the regressand contains the previous
estimate $\hat\ba_{t-1}$ itself, so that new information about $\ba$
enters only through $\bK_t\bH_t$, nothing is learned as
$\bK_t\to\bO$, and the error analysis must control the self-reference.

\emph{Prediction-error regression.} By~\eqref{eq:global_lifted_state}
and~\eqref{eq:global_lifted_meas},
$z_t=\bH_t\bL\bPhi(x_{t-1})+\bH_t\bU(x_{t-1})\ba_{t-1}+\bH_t\bzeta_{t-1}+v_t$,
which is linear in the physical weights. Replacing $x_{t-1}$ by its
filtered estimate gives the regression data
\begin{equation}
  \hat\by_t\triangleq z_t-\bH_t\bL\bPhi(\hat x_{t-1}),\
  \hD_t\triangleq\bH_t\hU_{t-1},\,\hU_{t-1}\triangleq\bU(\hat x_{t-1}),
  \label{eq:regression_data}
\end{equation}
so that $\hat\by_t-\hD_t\hat\ba_{t-1}$ is exactly the filter innovation,
whereas $(\hat\by_t,\hD_t)$ themselves contain no topology
estimate explicitly; they depend on past estimates only through
$\hat x_{t-1}$, an errors-in-variables effect that is quantified in
Theorem~\ref{thm:topology_error}. Let $\calT_k\triangleq\{\max(1,k-W+1),\ldots,k\}$ be a window of
length $W\le\infty$ and $\gamma\in(0,1]$ a forgetting factor, with
normalized weights and effective sample size
\begin{equation}
  \omega_{k,t}\triangleq\frac{\gamma^{k-t}}{\sum_{s\in\calT_k}\gamma^{k-s}},
  \qquad
  n_{{\rm eff},k}\triangleq\Bigl(\sum_{t\in\calT_k}\omega_{k,t}^2\Bigr)^{-1}.
  \label{eq:weights}
\end{equation}
For a full window, $n_{\rm eff}=\frac{(1-\gamma^W)(1+\gamma)}{(1-\gamma)(1+\gamma^W)}$,
which equals $W$ for $\gamma=1$ and tends to $\frac{1+\gamma}{1-\gamma}$
as $W\to\infty$ (e.g., $19$, $39$, $99$, $199$ for
$\gamma=0.90,0.95,0.98,0.99$). The topology estimate is the elastic-net
solution
\begin{align}
  \hat\ba_k\in\argmin_{\ba\in\calA}\;J_k(\ba)\triangleq\;
  &\frac12\sum_{t\in\calT_k}\omega_{k,t}\bigl\|\hat\by_t-\hD_t\ba\bigr\|^2\nonumber\\
  &+\frac{\mu_0}{2}\norm{\ba}^2+\lambda_k\norm{\ba}_1,
  \label{eq:topology_subprob}
\end{align}
with $\lambda_k\ge0$ and $\mu_0\ge0$, $\lambda_k+\mu_0>0$. Writing
\begin{align}
  \hG_k&\triangleq\sum_{t\in\calT_k}\omega_{k,t}\hD_t^\top\hD_t
  =\blkdiag(\hG_{1,k},\ldots,\hG_{p,k}),\nonumber\\
  \hG_{\mu,k}&\triangleq\hG_k+\mu_0\bI,\qquad
  \hat\bb_k\triangleq\sum_{t\in\calT_k}\omega_{k,t}\hD_t^\top\hat\by_t,
  \label{eq:Gram}
\end{align}
the smooth part of~\eqref{eq:topology_subprob} is
$\frac12\ba^\top\hG_{\mu,k}\ba-\hat\bb_k^\top\ba+{\rm const}$. Since $\bH_t$
and $\bU$ are block diagonal over nodes, $\hG_k$ is block diagonal over
rows, and~\eqref{eq:topology_subprob} decouples into $p$ independent
row problems of dimension $|\calI_i|$. The strong-convexity modulus of
$J_k$ is
\begin{equation}
  \mu_k\triangleq\lambda_{\min}(\hG_{\mu,k})\ge\mu_0 .
  \label{eq:muk}
\end{equation}

\begin{remark}[Design choices]
\label{rem:design}
(a) \emph{Weighting.} Replacing $(\bH_t,v_t)$ by $(\bLambda_t\bH_t,\bLambda_tv_t)$
with an $\calF_{t-1}$-measurable block-diagonal $\bLambda_t$, e.g., the
inverse square root of the innovation covariance, gives a weighted
regression; all results below hold with $\bar h$ and the noise
constants computed for the weighted quantities. (b) \emph{Rank.} Each
time step adds at most $q$ to the rank of $\hG_{i,k}$, and only if node
$i$ is measured; hence $\rank(\hG_{i,k})\le r_{i,k}$, the number of
measurement rows of node $i$ in $\calT_k$ ($r_{i,k}\le qm_{i,k}$ if node
$i$ is measured $m_{i,k}$ times), and $\lambda_{\min}(\hG_{i,k})>0$
requires $r_{i,k}\ge|\calI_i|$ together with varying regressors. When this fails,
$\mu_0>0$ keeps~\eqref{eq:topology_subprob} strongly convex and
Proposition~\ref{prop:sparse} describes the role of the $\ell_1$
penalty. In-edges of a node that is never measured are not identifiable
from one-step prediction errors. (c) \emph{Constraints.} The constraint
$\bA\bone=\bone$ is not imposed: it is not satisfied by general
weighted graphs, and for $\bA\ge\bO$ it makes $\normone{\bA}=p$ constant
on the feasible set, which neutralizes the $\ell_1$ penalty. Physically
justified affine constraints (e.g., symmetry) can be eliminated by an
affine reparameterization whenever the induced constraint set is a box.
(d) \emph{Tuning.} The analysis suggests $\lambda_k\propto n_{{\rm eff},k}^{-1/2}$ and a small $\mu_0$, whose bias is
explicit in Theorem~\ref{thm:topology_error}.
\end{remark}

\begin{remark}[Sensors of higher relative degree]
\label{rem:reldeg}
The data~\eqref{eq:regression_data} carry information about $\ba$ only
through $\hD_t=\bH_t\hU_{t-1}$. If $\bH_t\bU(x)=\bO$ for all $x$, as when
only positions are measured while the coupling acts on velocities, then
$\hD_t=\bO$, $\hG_k=\bO$ and $\hat\bb_k=\bzero$, so
that~\eqref{eq:topology_subprob} reduces to
$\min_{\ba\in\calA}\frac{\mu_0}{2}\norm{\ba}^2+\lambda_k\norm{\ba}_1$,
whose unique minimizer is $\bzero$ whenever $\bzero\in\calA$ and
$\mu_0>0$. The one-step regression is then uninformative, however
accurately the filter estimates the unmeasured coordinates: state
observability does not create a nonzero regressor. Suppose instead that
$\bH_t\bL^{l}\bU(x)=\bO$ for all $x$ and $0\le l\le r-2$, for some
$r\ge2$, i.e., the weights act on the measurements with a delay of $r$
steps, the discrete-time analogue of a relative
degree~\cite{isidori1995nonlinear}. Iterating~\eqref{eq:global_lifted_state} $r$ times and using
$\bPsi_s=\bPhi(x_s)$ gives the exact identity
\begin{align}
  z_t=\;&\bH_t\bL^r\bPhi(x_{t-r})+\bH_t\bL^{r-1}\bU(x_{t-r})\ba_{t-r}\nonumber\\
  &+\sum_{l=0}^{r-1}\bH_t\bL^{l}\bzeta_{t-1-l}+v_t,
  \label{eq:rstep}
\end{align}
because the terms $\bH_t\bL^{l}\bU(x_{t-1-l})\ba_{t-1-l}$, $l\le r-2$,
vanish. The right-hand side is again affine in the physical weights,
and replacing $x_{t-r}$ by $\hat x_{t-r}$ gives the $r$-step data
\begin{equation}
  \hat\by^{(r)}_t\triangleq z_t-\bH_t\bL^r\bPhi(\hat x_{t-r}),\,
  \hD^{(r)}_t\triangleq\bH_t\bL^{r-1}\bU(\hat x_{t-r}),
  \label{eq:regression_data_r}
\end{equation}
which replace~\eqref{eq:regression_data} in~\eqref{eq:topology_subprob}
and coincide with it for $r=1$. Since $\bH_t$, $\bL$ and $\bU$ are block
diagonal over nodes, the row-wise decoupling is preserved. The analysis
of Section~\ref{sec:theory} is carried out for $r=1$. For $r\ge2$ the
regression residual has the same structure, but its noise part
$\sum_{l}\bH_t\bL^{l}\tilde\bxi_{t-1-l}+v_t$ is correlated over $r$
consecutive steps, so that~\eqref{eq:martingale_bound} does not apply as
stated; results obtained with $r\ge2$ are therefore empirical.
\end{remark}

\subsection{ADMM Solver}
\label{subsubsec:block_soft}

Splitting $\ba=\bc$, the scaled ADMM with penalty $\beta>0$ reads
\begin{align}
  \ba^{(r+1)}&=(\hG_{\mu,k}+\beta\bI)^{-1}\bigl(\hat\bb_k+\beta(\bc^{(r)}-\bu^{(r)})\bigr),
  \label{eq:T_update}\\
  \bc^{(r+1)}&=\prox_{\beta^{-1}\mathcal{R}}\bigl(\ba^{(r+1)}+\bu^{(r)}\bigr),
  \label{eq:C_update}\\
  \bu^{(r+1)}&=\bu^{(r)}+\ba^{(r+1)}-\bc^{(r+1)},
  \label{eq:D_update}
\end{align}
where $\mathcal{R}(\ba)\triangleq\lambda_k\norm{\ba}_1+\iota_\calA(\ba)$
and $\iota_\calA$ is the indicator of $\calA$. The matrix
$\hG_{\mu,k}+\beta\bI$ is factorized once per time step and row by row.

\begin{lemma}[Proximal operator]
\label{lem:prox}
\begin{oldrevblock}
For the box $\calA$, $\prox_{\beta^{-1}\mathcal{R}}$ acts entrywise: for
every free coordinate $m=(i,j)\in\calI$,
\begin{equation}
  [\prox_{\beta^{-1}\mathcal{R}}(\bv)]_{m}
  =\clip_{[l_{m},u_{m}]}\Bigl(\sign(v_{m})\max\bigl(|v_{m}|-\tfrac{\lambda_k}{\beta},0\bigr)\Bigr).
  \label{eq:prox}
\end{equation}
\end{oldrevblock}
\end{lemma}
\begin{proof}
See Appendix~\ref{app_thm2}.
\end{proof}

\begin{remark}
\begin{oldrevblock}
(a) Positive definiteness of $\hG_{\mu,k}+\beta\bI$ ensures only that the
$\ba$-update is unique; strong convexity of $J_k$ requires $\mu_k>0$,
which the ridge term guarantees by design when $\mu_0>0$
(Theorem~\ref{thm:admm_convergence}). Without strong convexity,
only the global convergence of Theorem~\ref{thm:admm_convergence}(i)
is available; the linear rate and the accuracy certificate of parts
(ii)--(iii) do not apply. (b) For multi-channel edges,
where each edge carries a vector $\boldsymbol{\theta}_{ij}\in\mathbb{R}^m$ of weights
of $m$ interaction functions, the penalty becomes
$\lambda_k\sum_{(i,j)\in\calI}\norm{\boldsymbol{\theta}_{ij}}$ and~\eqref{eq:prox} is replaced
by the block soft-thresholding
$\prox_{\tau\norm{\cdot}}(\bv)=(1-\tau/\norm{\bv})_+\bv$ (zero for
$\bv=\bzero$) followed by the projection onto the constraint set when it
is a product of balls.
\end{oldrevblock}
\end{remark}

\subsection{Complete Algorithm}
\label{subsec:complete_alg}

Algorithm~\ref{alg:koopman_gkfa} summarizes Koopman-GKFA. The ADMM
is warm-started at the previous estimate and stopped by the a
posteriori certificate of Theorem~\ref{thm:admm_convergence}(iii), which
guarantees a prescribed optimization accuracy $\varepsilon_{\rm opt}$
whenever $\mu_k>0$. \rev{The algorithm therefore uses $\mu_0>0$, which
makes $\mu_k\ge\mu_0>0$ and the stopping rule computable. If $\mu_0=0$
is chosen and $\hG_k$ is singular, the iteration can only be stopped
on the primal and dual residuals $\norm{\ba^{(r+1)}-\bc^{(r+1)}}$ and
$\beta\norm{\bc^{(r+1)}-\bc^{(r)}}$; it still converges, but $\varepsilon_{\rm opt}$ in
(C5) below is then an unverified tolerance.} An edge set is reported by
thresholding $\hat\ba_k$ at level $\tau$. The data-computable quantities of
Remark~\ref{rem:verifiable}(a) are evaluated along the run and serve as
diagnostics.

\begin{algorithm}[ht]
\caption{Koopman-GKFA (Koopman Graph-structured KF--ADMM)}
\label{alg:koopman_gkfa}
\begin{algorithmic}[1]
\REQUIRE
  Dictionaries $\bphi,\bchi$; matrices $\{\bF_i,\bB_i,\bH_i\}$ from
  Section~\ref{subsec:pretrain}; box $\calA$ and free set $\calI$; $\hat x_0$, $\bP_0$,
  $\hat\ba_0\in\calA$; $\{\lambda_k\}$, \rev{$\mu_0>0$}, $\beta,\gamma,W,\tau,\varepsilon_{\rm opt}$;
  measurements $\{z_k\}$.
\ENSURE $\hat x_k$, $\hat\bA_k$, $\hat\calE_k$ for each $k$.
\FOR{$k = 1, 2, \ldots$}
  \STATE \textit{\# Re-lifted filter with interaction input}
  \STATE Compute $\hat{\bPsi}_{k|k-1}$, $\bP_{k|k-1}$
         by~\eqref{eq:kf_pred_state}--\eqref{eq:kf_pred_cov}.
  \STATE Compute $\bK_k$, $\hat x_k$, $\bP_k$
         by~\eqref{eq:kalman_gain}--\eqref{eq:kf_update_cov};
         re-lift $\hat{\bPsi}_{k} \leftarrow \bPhi(\hat x_k)$.
  \STATE \textit{\# Prediction-error data and Gram update}
  \STATE Form $\hat\by_k$, $\hD_k$ by~\eqref{eq:regression_data}
         (or by~\eqref{eq:regression_data_r}, Remark~\ref{rem:reldeg});
         update $\hG_{\mu,k}$, $\hat\bb_k$ by~\eqref{eq:Gram}.
  \STATE \textit{\# ADMM on the free physical weights}
  \STATE $\bc^{(0)}\leftarrow\hat\ba_{k-1}$, $\bu^{(0)}\leftarrow\bzero$;
         factorize $\hG_{\mu,i,k}+\beta\bI$, $i\in\calV$.
  \REPEAT
    \STATE Update $\ba^{(r+1)},\bc^{(r+1)},\bu^{(r+1)}$ by~\eqref{eq:T_update}--\eqref{eq:D_update} with~\eqref{eq:prox};
           $\bt^{(r)}\leftarrow\ba^{(r+1)}+\bu^{(r)}$.
  \UNTIL{$\norm{\bt^{(r)}-\bt^{(r-1)}}\le(1-\rho(\beta))\,\varepsilon_{\rm opt}$}
  \STATE $\hat\ba_k\leftarrow\bc^{(r+1)}$;
         $\hat\calE_k\leftarrow\{(i,j)\in\calI:|\hat a_{ij,k}|>\tau\}$.
\ENDFOR
\end{algorithmic}
\end{algorithm}

\begin{remark}[On the joint objective]
\label{rem:convergence_stability}
The online scheme does not minimize a fixed joint objective over
$(\bPsi_k,\bA)$; instead, each time step solves a well-posed convex topology problem
(Theorem~\ref{thm:admm_convergence}), and the interaction between the
filter and the topology estimator is analyzed as a feedback loop
(Theorems~\ref{thm:topology_error} and~\ref{thm:error_bound}).
\end{remark}

\subsection{Computational Complexity Analysis}
\label{sec:complexity}

Let $R$ be the number of ADMM iterations per step and $m_k\le qp$ the
number of measurements at time $k$. \emph{Filter:} with dense
matrices, $\hat\bM_{k-1}\bP_{k-1}\hat\bM_{k-1}^\top$ costs
$\mathcal{O}(N^2np^3)$, the innovation covariance and gain cost
$\mathcal{O}(m_kN^2p^2+m_k^3)$, and the Jacobian, which has at most
$p(d+1)$ nonzero $N\times n$ blocks when $\hat\bA_{k-1}$ has in-degree at
most $d$, costs $\mathcal{O}(p\,d\,NN_e n)$; the memory is
$\mathcal{O}(N^2p^2)$. \emph{Regression data:} $\hU_{k-1}$ requires $p^2$
evaluations of $\bchi$ and $\mathcal{O}(p^2NN_e)$ operations, $\hD_k$
costs $\mathcal{O}(qNp^2)$, and updating $\hG_{i,k}$ costs
$\mathcal{O}(q\bar p^2)$ per node (recursively for $W=\infty$, with one
down-date for finite $W$). \emph{ADMM:} the $p$ Cholesky factorizations
cost $\mathcal{O}(p\bar p^3)$ and each iteration $\mathcal{O}(p\bar p^2)$.
The overall per-step cost is therefore
\begin{equation}
  \mathcal{O}\bigl(N^2np^3+m_k^3+p^2NN_e+qNp^2+p\bar p^3+Rp\bar p^2\bigr).
  \label{eq:complexity}
\end{equation}
By Theorem~\ref{thm:admm_convergence}, $R=\mathcal{O}(\log(1/\varepsilon_{\rm opt})/\log(1/\rho(\beta)))$
for a fixed $\beta$, and
$R=\mathcal{O}(\sqrt{\kappa_k}\log(1/\varepsilon_{\rm opt}))$ for the step-dependent choice $\beta=\sqrt{\mu_k\bar\mu_k}$, with
$\kappa_k\le\lambda_{\max}(\hG_{\mu,k})/\mu_0$ when $\mu_0>0$. The topology
stage does not scale with $N$, but with a dense candidate set,
$\bar p=p-1$, its factorization cost $p\bar p^3=\mathcal{O}(p^4)$ exceeds
the filter cost $\mathcal{O}(N^2np^3)$ once $p\gtrsim N^2n$; the filter
dominates when the candidate sets are small, $\bar p\ll p$. The filter
cost is polynomial in $Np$ and can be reduced by exploiting the block
sparsity of $\hat\bM_{k-1}$. Particle-based joint estimators require a
number of particles that typically grows exponentially with the state
dimension~\cite{Bengtsson2008}.

%% ===============================================================
\section{Theoretical Analysis}
\label{sec:theory}
%% ===============================================================

The analysis follows the logic of the method: we establish convergence
of the solver that is actually run; we bound the topology error when the regression is fed with
\emph{estimated} states; we close the
loop between the filter and the topology estimator by a small-gain
argument; we analyze the sparse regime
in which the Gram matrix may be singular; and we derive consequences for time-varying topologies, edge-support
recovery and dictionary design. We also
show that the one-step contraction underlying the closed-loop theorem
fails under partial sensing and replace it by a multi-step condition. All statements are non-asymptotic. Their constants fall into three classes, quantities computed from
data, a priori bounds implied by the design, and bounds on unobservable
quantities that must be supplied as prior information; the statements are certificates only
relative to the latter.

\subsection{Standing Assumptions}
\label{subsec:assumptions}

Let $\bDelta_k\triangleq\hat\ba_k-\ba_k$ denote the error of the exact
minimizer of~\eqref{eq:topology_subprob}, $\hat\ba_k^{\rm alg}$ the ADMM
output used by the filter, and $\be_k\triangleq x_k-\hat x_k$ the
\emph{physical} state error. For a nonnegative sequence $\{y_t\}$,
define its window average
$[y]_k\triangleq(\sum_{t\in\calT_k}\omega_{k,t}y_t^2)^{1/2}$ and its
lagged version $[y_{\cdot-1}]_k\triangleq(\sum_{t\in\calT_k}\omega_{k,t}y_{t-1}^2)^{1/2}$.
Let $\bar h\triangleq\sup_t\norm{\bH_t}$ and define the
\emph{errors-in-variables constant}
\begin{align}
  c_e&\triangleq\ell_{HL}+\bar h\,\ell_U,\nonumber\\
  \ell_{HL}&\triangleq\sup_t\sup_{\bm\xi\in\mathbb{R}^{np}}\norm{\bH_t\bL\bJPhi(\bm\xi)}
  \le\bar h\norm{\bL}L_\phi,
  \label{eq:ce}
\end{align}
with $\ell_U$ from Lemma~\ref{lem:regressor} and (C2).

\begin{lemma}[Regressor bounds]
\label{lem:regressor}
Under Assumption~\ref{ass:dictionary}(A3), for all $x,x'\in\mathbb{R}^{np}$
and every $\bA$ supported on $\calI$ with free vector $\ba$,
\begin{align}
  \norm{\bU(x)}&\le b_U\triangleq\sqrt{\bar p}\,\max_i\norm{\bB_i}\,b_\chi,
  \label{eq:bU}\\
  \norm{(\bU(x)-\bU(x'))\ba}&\le\ell_U(\bA)\norm{x-x'},
  \label{eq:lU}
\end{align}
where $\ell_U(\bA)\triangleq\max_i\norm{\bB_i}\,L_\chi\bigl(\norm{\bA}_\infty
+\sqrt{\norm{\bA}_\infty\norm{\bA}_{1\to1}}\bigr)$. Moreover,
$\norm{\hD_t}\le b_D\triangleq\bar hb_U$, and every column of $\hD_t$ has
norm at most $b_c\triangleq\bar h\max_i\norm{\bB_i}b_\chi$.
\end{lemma}
\begin{proof}
See Appendix~\ref{app_thm3}.
\end{proof}

\begin{assumption}[Closed-loop conditions]
\label{ass:stability}
\begin{enumerate}
\item[(C1)] \textit{Feasibility, representation, initialization:}
$\ba_k\in\calA$ for all $k$; the residual~\eqref{eq:delta_def} of the
matrices used online satisfies $\sup_k\Lnorm{\bdelta_k}\le\varepsilon_{\rm rep}$;
$\bar h<\infty$; and $\expect\norm{\be_t}^2<\infty$ for $t\le k_0$.
\item[(C2)] \textit{Topology variation and sparsity:}
$\norm{\ba_{k+1}-\ba_k}\le\nu$ ($\nu=0$ for a static topology),
$|\supp(\ba_k)|\le s$, $\bar a\triangleq\sup_k\norm{\ba_k}<\infty$, and
$\ell_U\triangleq\sup_k\ell_U(\bA_k)<\infty$.
\item[(C3)] \textit{Excitation:} $\mu_k\ge\mu>0$ almost surely for all
$k\ge k_0$, with $\mu_k$ from~\eqref{eq:muk}.
\item[(C4)] \textit{Filter gains:} $\bK_t$ is $\calF_{t-1}$-measurable and,
almost surely for all $t\ge1$,
\begin{align}
  &\sup_{\bm\xi\in\mathbb{R}^{np}}\norm{\Pip(\bI-\bK_t\bH_t)\bL\bJPhi(\bm\xi)}\le\rho_L,\nonumber\\
  &\norm{\Pip(\bI-\bK_t\bH_t)}\le\kappa_K,\qquad\norm{\Pip\bK_t}\le\kappa_v,
  \label{eq:C4}
\end{align}
and, for $t\ge k_0+1$,
$\norm{\Pip(\bI-\bK_t\bH_t)\hU_{t-1}\hG_{\mu,t-1}^{-1/2}}\le\kappa_U$.
\item[(C5)] \textit{Solver accuracy:}
$\norm{\hat\ba_k^{\rm alg}-\hat\ba_k}\le\varepsilon_{\rm opt}$ almost surely
for $k\ge k_0$.
\end{enumerate}
\end{assumption}

\begin{remark}[Verifiability and nature of the constants]
\label{rem:verifiable}
(a) \emph{Three classes of constants.} (i) \emph{Computed from data
along the run:} $\mu_k$; the norms in (C4) other than $\rho_L$, i.e.,
$\kappa_K$, $\kappa_v$ and $\kappa_U$, evaluated at the realized
$\bK_t$, $\bH_t$, $\hU_{t-1}$ and $\hG_{\mu,t-1}$; $\bar h$ for a known
schedule; and the solver accuracy (C5), certified by the stopping rule
when $\mu_k>0$ (Theorem~\ref{thm:admm_convergence}(iii)). (ii) \emph{A
priori bounds implied by the design:} $L_\phi$, $L_2$, $b_\chi$, $L_\chi$,
$b_U$, $b_c$ and $\ell_{HL}$ follow from the dictionaries and $\bL$;
$\rho_L$ is a supremum over $\bm\xi$, which is explicit when
$\Pip\bL\bJPhi$ is constant (Example~\ref{ex:constructions}) and is
otherwise replaced by a conservative bound such as
$\sup_t\norm{\Pip(\bI-\bK_t\bH_t)}\norm{\bL}L_\phi$; $\ell_U$ and $\bar a$
follow from the box $\calA$ when it is bounded, or from prior bounds
on $\bA_k$. (iii) \emph{Unobservable quantities, supplied as prior
bounds:} the deployment residuals $\varepsilon_{\rm rep}$ and
$\varepsilon_x$, the drift $\nu$, the noise levels if the covariances
are unknown, $a_{\min}$, and the support $S_k$ in~\eqref{eq:RE}. The
validation residual~\eqref{eq:val_residual} and the analytic bounds of
Theorem~\ref{thm:lifting} and Lemma~\ref{lem:localization} inform, but
do not certify, the bounds in (iii) (Section~\ref{subsec:pretrain}).
(b) \emph{Certificates versus diagnostics.} Assumption~\ref{ass:stability}
requires almost-sure bounds valid for every realization. If $\mu_0>0$,
(C3) holds deterministically with $\mu=\mu_0$ and $k_0=1$; this is a
property of the optimization problem, not of the data, whose
excitation is measured by $\lambda_{\min}(\hG_k)$; if, in
addition, a fixed gain $\bK$ is used, every constant in (C4) is
deterministic, with $\kappa_U\le\kappa_Kb_U/\sqrt{\mu_0}$, and
Theorems~\ref{thm:topology_error}--\ref{thm:error_bound} are a priori
guarantees given the bounds in (a)(iii). When the constants are instead
evaluated along one realized trajectory (time-varying gains, $\mu$
from excitation), finite-horizon maxima are \emph{empirical
diagnostics}: they indicate whether a run is consistent with the
certified regime but do not establish the almost-sure bounds of
Assumption~\ref{ass:stability}. Under phase locking, (C3) holds only with $\mu=\mu_0$ and
the bounds degrade accordingly instead of becoming vacuous.
(c) \emph{Effect of re-lifting, and its limit.} $\rho_L$ involves only
$\Pip(\bI-\bK_t\bH_t)\bL\bJPhi$, i.e., the action of the lifted dynamics
on tangent directions of the lifted manifold, read out on the state
coordinates. Undetectable modes of $(\bL,\bH_t)$ that are not excited
along these directions, such as the rotation modes of
Proposition~\ref{prop:undetectable}, do not enter. For the Kuramoto
lifting with $\bH_t=\Pip$ and $\Pip\bK_t=k\bI_p$, one has
$\Pip\bL\bJPhi\equiv\bI_p$, hence $\rho_L=\kappa_K=1-k$, $\ell_{HL}=1$, and
$\rho_L<1$ for every $k\in(0,1]$, in contrast with
Proposition~\ref{prop:undetectable}. Re-lifting does not, however,
remove \emph{physical} directions that are not measured: one-step
contraction fails whenever an unmeasured physical direction is a
fixed direction of the linearized self-dynamics, e.g., for every
unmeasured Kuramoto node. Uniform observability and
controllability of the linearized pair yield stability of the EKF error
in a time-varying (Lyapunov) metric and over several
steps~\cite{anderson1979optimal,reif1999stochastic}; they do not imply the
one-step Euclidean contraction $\rho_L<1$. The multi-step condition of
Section~\ref{subsec:multistep} is the corresponding relaxation in the
present setting.
\end{remark}

\subsection{Convergence of the ADMM Solver}
\label{subsec:admm_theory}

Fix $k$ and write $\hat\bG=\hG_{\mu,k}$, $\hat\bb=\hat\bb_k$, $\mu_k=\lambda_{\min}(\hat\bG)$,
$\bar\mu_k=\lambda_{\max}(\hat\bG)$, $\kappa_k=\bar\mu_k/\mu_k$, and
$\bt^{(r)}\triangleq\ba^{(r+1)}+\bu^{(r)}$.

\begin{theorem}[Convergence of the topology ADMM]
\label{thm:admm_convergence}
Let $\calA$ be a nonempty box, $\lambda_k\ge0$, $\mu_0\ge0$ with
$\lambda_k+\mu_0>0$, and $\beta>0$.
\begin{enumerate}[label=(\roman*),leftmargin=*]
\item \emph{(Global convergence)} Problem~\eqref{eq:topology_subprob} has
a minimizer, and for every initialization the iterates
of~\eqref{eq:T_update}--\eqref{eq:D_update} satisfy
$\bc^{(r)}\to\hat\ba_k$ and $\ba^{(r)}-\bc^{(r)}\to\bzero$, where
$\hat\ba_k$ is a minimizer.
\item \emph{(Linear rate)} If $\mu_k>0$, which holds whenever $\mu_0>0$,
the minimizer $\hat\ba_k$ is unique and, for all $r\ge0$,
\begin{align}
  &\max\bigl\{\norm{\bc^{(r+1)}-\hat\ba_k},\,\norm{\ba^{(r+2)}-\hat\ba_k}\bigr\}\nonumber\\
  &\qquad\le\rho(\beta)^r\norm{\bt^{(0)}-\bt^\star},
  \label{eq:linear_rate}
\end{align}
where $\bt^\star$ is the unique fixed point of the iteration map and
\begin{align}
  \rho(\beta)&=\frac{1+\delta(\beta)}{2},\nonumber\\
  \delta(\beta)&=\max\Bigl\{\frac{|\beta-\mu_k|}{\beta+\mu_k},
  \frac{|\bar\mu_k-\beta|}{\bar\mu_k+\beta}\Bigr\}<1.
  \label{eq:rho}
\end{align}
The rate is minimized by $\beta^\star=\sqrt{\mu_k\bar\mu_k}$, which gives
$\rho(\beta^\star)=\sqrt{\kappa_k}/(\sqrt{\kappa_k}+1)$, so that an
accuracy $\varepsilon$ is reached after
$\mathcal{O}(\sqrt{\kappa_k}\log(1/\varepsilon))$ iterations.
\item \emph{(A posteriori certificate)} If $\mu_k>0$, then for all $r\ge1$,
\begin{equation}
  \norm{\bc^{(r+1)}-\hat\ba_k}\le\frac{\rho(\beta)}{1-\rho(\beta)}
  \norm{\bt^{(r)}-\bt^{(r-1)}}.
  \label{eq:certificate}
\end{equation}
In particular, the stopping rule of Algorithm~\ref{alg:koopman_gkfa}
guarantees $\norm{\hat\ba_k^{\rm alg}-\hat\ba_k}\le\rho(\beta)\varepsilon_{\rm opt}$.
\end{enumerate}
\end{theorem}
\begin{proof}
See Appendix~\ref{app_thm2}.
\end{proof}

\begin{remark}[What strong convexity requires]
The proof uses the equivalence of ADMM with Douglas--Rachford
splitting~\cite{lions1979splitting,eckstein1992douglas}; under strong
convexity and smoothness, rates of this type are classical for general
problems~\cite{giselsson2017linear,hong2017linear}. The rate~\eqref{eq:rho} is governed by
$\hG_{\mu,k}$, not by the ADMM system matrix $\hG_{\mu,k}+\beta\bI$, which
is positive definite for any $\beta>0$ and therefore says nothing about
strong convexity. Because~\eqref{eq:topology_subprob} is strictly convex
when $\mu_k>0$, every convergent solver returns the same estimate;
the role of the theorem is to certify the accuracy $\varepsilon_{\rm opt}$
that enters the closed-loop analysis.
\end{remark}

\subsection{Topology Error Under Estimated States}
\label{subsec:topology_error}

\begin{theorem}[Topology error under estimated states]
\label{thm:topology_error}
Let Assumptions~\ref{ass:dictionary} and~\ref{ass:stability} hold and
$k\ge k_0$. Then, almost surely,
\begin{align}
  \norm{\bDelta_k}_{\hG_{\mu,k}}\le\;&
  \frac{\lambda_k\sqrt s+\mu_0\bar a}{\sqrt\mu}
  +c_e\,[\norm{\be_{\cdot-1}}]_k
  +\bigl[\eta^{\rm b}\bigr]_k+\frac{\norm{\bvarsigma_k}}{\sqrt\mu},
  \label{eq:topology_bound}
\end{align}
and $\norm{\bDelta_k}\le\norm{\bDelta_k}_{\hG_{\mu,k}}/\sqrt\mu$. Here
$\bvarsigma_k\triangleq\sum_{t\in\calT_k}\omega_{k,t}\hD_t^\top\bm\varepsilon_t$
with $\bm\varepsilon_t\triangleq\bH_t\tilde\bxi_{t-1}+v_t$ and
$\tilde\bxi_{t}\triangleq\bxi_{t}-\expect[\bxi_{t}\mid\calF_{t}]$ is a
martingale term, and
$\eta^{\rm b}_{k,t}\triangleq\bigl\|\bH_t\bigl(\hU_{t-1}(\ba_{t-1}-\ba_k)
+\bdelta_{t-1}+\expect[\bxi_{t-1}\mid\calF_{t-1}]\bigr)\bigr\|$
collects drift and bias. Moreover,
\begin{align}
  \Lnorm{\norm{\bvarsigma_k}}&\le\frac{b_D\sigma_\varepsilon}{\sqrt{n_{{\rm eff},k}}},
  \sigma_\varepsilon^2\triangleq\bar h^2L_\phi^2\sigma_w^2+\sigma_v^2,
  \label{eq:martingale_bound}\\
  \Lnorm{[\eta^{\rm b}]_k}&\le\bar hb_U\nu\,\ell^{(2)}_k+\beta_{\rm rep},
  \beta_{\rm rep}\le\bar h\Bigl(\varepsilon_{\rm rep}+\frac{L_2}{2}\sigma_w^2\Bigr),
  \label{eq:bias_bound}
\end{align}
where $\ell^{(2)}_k\triangleq(\sum_{t\in\calT_k}\omega_{k,t}(k-t+1)^2)^{1/2}$
and $\beta_{\rm rep}\triangleq\sup_t\Lnorm{\bH_t(\bdelta_{t-1}+\expect[\bxi_{t-1}\mid\calF_{t-1}])}$.
If the sensors are linear, $\bH_t=\bC_t\Pip$ with $\norm{\bC_t}\le\bar c$,
then
\begin{equation}
  \beta_{\rm rep}\le\bar c\varepsilon_x,\,
  \varepsilon_x\triangleq\sup_t\Lnorm{\Pip\bdelta_t},
  \sigma_\varepsilon^2\le\bar c^2\sigma_w^2+\sigma_v^2,
  \label{eq:linear_sensor}
\end{equation}
and $\varepsilon_x=0$ whenever the matrices used online satisfy
$\bPi\bF_i\bphi\equiv f_i$ and $\bPi\bB_i\bchi\equiv g$.
\end{theorem}
\begin{proof}
See Appendix~\ref{app_thm3}.
\end{proof}

Theorem~\ref{thm:topology_error} identifies five mechanisms: the
$\ell_1$ and ridge shrinkage biases $(\lambda_k\sqrt s+\mu_0\bar a)/\sqrt\mu$;
the errors-in-variables term, proportional to the window-averaged
physical state error with gain $c_e$; the drift and representation bias
$[\eta^{\rm b}]_k$; and the statistical term, which decays as
$n_{{\rm eff},k}^{-1/2}$. Three structural features are essential. First,
the regression data~\eqref{eq:regression_data} do not contain any explicit topology estimate; past estimates enter only through
$\hat x_{t-1}$, i.e., through the errors-in-variables term, so no
self-referential term appears and no
contraction of the filter correction is needed for the topology step
in isolation. Second, the bound is derived in the $\hG_{\mu,k}$-norm
through the self-normalized inequality of Lemma~\ref{lem:selfnormalized},
so that the gain from residuals to topology errors is at most one and
does not depend on the condition number of $\hG_{\mu,k}$. Third, with
linear sensors the lifted-noise bias of Proposition~\ref{prop:lifted_noise}
and the second-order coupling remainder of Theorem~\ref{thm:lifting}
do not enter~\eqref{eq:linear_sensor}: only the state-block residual
$\varepsilon_x$ does. For the constructions of
Example~\ref{ex:constructions}, matrices with $\varepsilon_x=0$ exist, so
that $\varepsilon_x$ reflects only the calibration error of the learned
state rows along the operational trajectory.

\subsection{Closed-Loop Joint Error Bound}
\label{subsec:mse_bound}

Define the state error and the topology error in the data metric,
\begin{equation}
  X_k\triangleq\Lnorm{\be_k},\qquad
  Y_k\triangleq\Lnorm{\norm{\bDelta_k}_{\hG_{\mu,k}}},
  \label{eq:XY}
\end{equation}
the filter contraction $\rho_x\triangleq\rho_L+\kappa_K\ell_U$, and the
nonnegative loop matrix
\begin{equation}
  \bM\triangleq\begin{bmatrix}\rho_x&\kappa_U\\ c_e&0\end{bmatrix}.
  \label{eq:loop_matrix}
\end{equation}
The entry $\kappa_U$ is the gain from topology errors to state errors
and $c_e$ the gain from state errors to topology errors. Let
$\bar\lambda\triangleq\limsup_k\lambda_k$, $\bar n_{\rm eff}\triangleq\liminf_kn_{{\rm eff},k}$
and $\bar\ell^{(2)}\triangleq\limsup_k\ell^{(2)}_k$ (with $\nu\bar\ell^{(2)}=0$
if $\nu=0$).

\begin{theorem}[Closed-loop joint estimation error]
\label{thm:error_bound}
Let Assumptions~\ref{ass:dictionary} and~\ref{ass:stability} hold, let
$\sup_k\lambda_k<\infty$, let the weights~\eqref{eq:weights} satisfy
$W<\infty$ or $\gamma<1$ or $\nu=0$, and suppose the small-gain condition
\begin{equation}
  \rho(\bM)<1\iff\rho_{\rm cl}\triangleq\rho_x+\kappa_Uc_e<1
  \label{eq:small_gain}
\end{equation}
holds. Then $\sup_kX_k<\infty$, $\sup_{k\ge k_0}Y_k<\infty$, and
\begin{align}
  \limsup_{k\to\infty}\begin{bmatrix}X_k\\Y_k\end{bmatrix}
  &\le(\bI-\bM)^{-1}\begin{bmatrix}f_x\\f_y\end{bmatrix}\nonumber\\
  &=\frac{1}{1-\rho_{\rm cl}}
  \begin{bmatrix}f_x+\kappa_Uf_y\\c_ef_x+(1-\rho_x)f_y\end{bmatrix},
  \label{eq:error_bound}
\end{align}
with the forcing terms
\begin{align}
  f_x&\triangleq\kappa_K\bigl(\varepsilon_{\rm rep}+L_\phi\sigma_w
  +b_U\varepsilon_{\rm opt}\bigr)+\kappa_v\sigma_v,
  \label{eq:fx}\\
  f_y&\triangleq\frac{\bar\lambda\sqrt s+\mu_0\bar a}{\sqrt\mu}
  +\bar hb_U\nu\bar\ell^{(2)}+\beta_{\rm rep}
  +\frac{b_D\sigma_\varepsilon}{\sqrt{\mu\,\bar n_{\rm eff}}}.
  \label{eq:fy}
\end{align}
Consequently,
$\limsup_k\Lnorm{\hat\ba_k^{\rm alg}-\ba_k}\le\mu^{-1/2}\limsup_kY_k+\varepsilon_{\rm opt}$.
\end{theorem}
\begin{proof}
See Appendix~\ref{app_thm4}.
\end{proof}

\begin{remark}[Interpretation and regime of validity]
\label{rem:interpretation}
(a) \emph{Dependence on the forcing terms.} All forcing terms are
explicit. For a fixed problem whose constants satisfy
Assumption~\ref{ass:stability} and~\eqref{eq:small_gain}, the ultimate
bound vanishes if $\sigma_w=\sigma_v=0$, $\varepsilon_{\rm rep}=\beta_{\rm rep}=0$,
$\nu=0$, $\bar\lambda=0$, $\varepsilon_{\rm opt}=0$ and $\mu_0=0$, with (C3)
provided by excitation. The loop constants $\mu$, $\rho_x$, $\kappa_U$ and
$c_e$, however, are themselves properties of the trajectory and of the
regularization. A consistency statement for a family of problems in
which the forcing terms tend to zero therefore requires, in addition,
uniform margins along the family: $\inf\mu>0$,
$\sup\rho_{\rm cl}<1$ and uniform bounds on $b_D$, $\bar h$, $\kappa_K$,
$\kappa_v$, $\kappa_U$ and $c_e$; when $\mu_0\to0$, (C3) must be provided
uniformly by excitation, which is a property of the data and not of
the ridge term. Under fixed noise, the bound does not vanish as
$n_{{\rm eff},k}\to\infty$: the errors-in-variables term
$c_e[\norm{\be_{\cdot-1}}]_k$ persists, as is well known for regression
with noisy regressors~\cite{soderstrom2007errors}, and topology
consistency under fixed noise is not claimed. For fixed loop constants
and linear sensors, the ultimate bound is affine in $(\sigma_w,\sigma_v)$;
for nonlinear sensors the bound~\eqref{eq:bias_bound} on $\beta_{\rm rep}$
also contains the term $\bar hL_2\sigma_w^2/2$. For $\gamma=1$ and $W=\infty$ the statistical term
vanishes.
(b) \emph{Role of the penalties.} For fixed $\mu$ in (C3), the bound is
smallest for $\bar\lambda=0$ and small $\mu_0$: when the Gram matrix is
well conditioned, the $\ell_1$ penalty is not needed for accuracy and
serves only to produce sparse estimates. Its statistical role is in the
underdetermined regime, where (C3) can hold only through $\mu_0$; this is
quantified by Proposition~\ref{prop:sparse}, which suggests
$\lambda_k\propto n_{{\rm eff},k}^{-1/2}$, so that $\bar\lambda=0$ for
$\gamma=1$, $W=\infty$. The ridge weight also acts on the \emph{loop
gain}: $\kappa_U\le\kappa_Kb_U/\sqrt{\mu_0}$, so a sufficiently large
$\mu_0$ always enforces~\eqref{eq:small_gain} when $\rho_x<1$, namely for
$\mu_0>(\kappa_Kb_Uc_e/(1-\rho_x))^2$. The price is the shrinkage bias
$\mu_0\bar a/\sqrt\mu$ in $f_y$: for large $\mu_0$ the certified Euclidean
topology error is of the order of $\bar a$, i.e., uninformative. An
informative certificate requires a small topology-to-state gain at a
moderate ridge weight, i.e., informative measurements.
(c) \emph{Regime of validity.} For the Kuramoto example of
Remark~\ref{rem:verifiable}(c) with full observation and scalar weight
$k$, one has $\rho_x=(1-k)(1+\ell_U)$, $c_e=1+\ell_U$ and
$\kappa_U=(1-k)\bar\chi$ with the normalized leverage
$\bar\chi\triangleq\sup_t\norm{\Pip\hU_{t-1}\hG_{\mu,t-1}^{-1/2}}$, so
that~\eqref{eq:small_gain} reads $(1-k)(1+\ell_U)(1+\bar\chi)<1$. Since
$\hG_{\mu,t-1}$ is a normalized Gram matrix, $\bar\chi$ is typically of the
order of $\sqrt{\bar p}$; the certificate therefore applies when the
measurements are sufficiently informative relative to the prediction,
and it is not available at low measurement signal-to-noise ratios. With
partial sensing, the one-step certificate is not available at all
(Proposition~\ref{prop:onestep}); Theorem~\ref{thm:multistep} addresses
this case. These are properties of the certificate rather than
necessarily of the algorithm, because the bound treats the state error
adversarially. Evaluating $\rho_{\rm cl}$
along a run provides a diagnostic of whether the run lies in the
certified regime (Remark~\ref{rem:verifiable}(b)).
(d) \emph{Relation to classical results.} The structure of
Theorem~\ref{thm:error_bound} parallels the analysis of joint state and
parameter estimation~\cite{ljung1979asymptotic} and of forgetting-factor
tracking algorithms~\cite{guo1995exponential}; the new elements are the
errors-in-variables channel through the lifted surrogate, its feedback
through the filter, and the explicit loop gain.
\end{remark}

\subsection{Intermittent Sensing: A Multi-Step Certificate}
\label{subsec:multistep}

The next result shows that the one-step contraction in (C4) is
incompatible with the most common form of partial observation.

\begin{proposition}[Limits of one-step contraction]
\label{prop:onestep}
Let $\bH_t=\bC_t\Pip$, and suppose that for some $\bm\xi\in\mathbb{R}^{np}$
there is $\bu\neq\bzero$ with $\Pip\bL\bJPhi(\bm\xi)\bu=\bu$ and
$\bC_t\bu=\bzero$. Then, for every gain $\bK_t$,
$\norm{\Pip(\bI-\bK_t\bH_t)\bL\bJPhi(\bm\xi)}\ge1$; hence $\rho_L\ge1$,
$\rho_x\ge1$ and $\rho_{\rm cl}\ge1$, and
Theorem~\ref{thm:error_bound} does not apply. For the Kuramoto lifting
of Example~\ref{ex:constructions}, $\Pip\bL\bJPhi\equiv\bI_p$, so this
occurs at every time at which some node is not measured, with $\bu$ the
corresponding unit vector.
\end{proposition}
\begin{proof}
See Appendix~\ref{app_calib}.
\end{proof}

Re-lifting removes the undetectable lifted modes of
Proposition~\ref{prop:undetectable}, but an unmeasured node is still
propagated open loop during one step, and with marginally stable
self-dynamics, such as phase rotation, no gain can contract its error
in that step. Contraction can nevertheless occur over several steps if
every node is measured sufficiently often. To capture this, let
$\bchi$ be continuously differentiable in addition to (A3), so that
$x\mapsto\bU(x)\ba$ is $C^1$, and define, for $\bm\xi\in\mathbb{R}^{np}$ and
$\ba\in\mathbb{R}^{p_\calI}$,
\begin{equation}
  \bN(\bm\xi;\ba)\triangleq\bL\bJPhi(\bm\xi)+\partial_{\bm\xi}\bigl[\bU(\bm\xi)\ba\bigr].
  \label{eq:Nmap}
\end{equation}
Let $\calA_\star\triangleq\{\ba\in\calA:\ell_U(\bA)\le\ell_U\}$, which
contains every $\ba_k$ by (C1)--(C2), and define the set of one-step
error-transition matrices
\begin{equation}
  \calS_t\triangleq\bigl\{\Pip(\bI-\bK_t\bH_t)\bN(\bm\xi;\ba):\,
  \bm\xi\in\mathbb{R}^{np},\ \ba\in\calA_\star\bigr\}.
  \label{eq:St}
\end{equation}
By (A2) and~\eqref{eq:lU}, $\norm{\bN(\bm\xi;\ba)}\le\norm{\bL}L_\phi+\ell_U$
on $\mathbb{R}^{np}\times\calA_\star$, so $\calS_t$ is bounded.

\begin{assumption}[Multi-step contraction]
\label{ass:multistep}
(C4$_m$) There exist $m\in\mathbb{N}$, $\rho_m\in[0,1)$ and
$\varpi\ge1$ such that, almost surely,
\begin{enumerate}[label=(\alph*),leftmargin=*]
\item $\norm{\bS_k\bS_{k-1}\cdots\bS_{k-m+1}}\le\rho_m$ for all
$k\ge k_0+m$ and all $\bS_t\in\calS_t$;
\item $\norm{\bS_k\cdots\bS_{k-r+1}}\le\varpi$ for all $k\ge k_0+1$, all
$1\le r\le m-1$ with $k-r\ge k_0$, and all $\bS_t\in\calS_t$;
\end{enumerate}
and the bounds $\kappa_K$, $\kappa_v$ and $\kappa_U$ of (C4) hold.
\end{assumption}

Condition (C4$_m$)(a) is a uniform bound on products of $m$ consecutive
transition sets, in the spirit of the joint spectral radius of a family
of matrices~\cite{jungers2009joint}; it does not require any single
factor to be contractive. Being uniform over $\bm\xi$ and
$\ba\in\calA_\star$, it is more demanding than a bound along the
realized trajectory, which can only serve as a diagnostic. Define the multi-step loop matrix and gain
\begin{equation}
  \bM_m\triangleq\begin{bmatrix}\rho_m&m\varpi\kappa_U\\ c_e&0\end{bmatrix},
  \,
  \rho_{\rm cl}^{(m)}\triangleq\rho_m+m\varpi\kappa_Uc_e .
  \label{eq:loop_matrix_m}
\end{equation}

\begin{theorem}[Multi-step closed-loop error bound]
\label{thm:multistep}
Let Assumption~\ref{ass:dictionary} hold with $\bchi\in C^1$, let
Assumption~\ref{ass:stability} hold with the first inequality
of~\eqref{eq:C4} replaced by Assumption~\ref{ass:multistep}, let
$\sup_k\lambda_k<\infty$, let the weights~\eqref{eq:weights} satisfy
$W<\infty$ or $\gamma<1$ or $\nu=0$, and suppose that
\begin{equation}
  \rho(\bM_m)<1\iff\rho_{\rm cl}^{(m)}<1 .
  \label{eq:small_gain_m}
\end{equation}
Then $\sup_kX_k<\infty$, $\sup_{k\ge k_0}Y_k<\infty$, and, with $f_x$,
$f_y$ from~\eqref{eq:fx}--\eqref{eq:fy},
\begin{align}
  \limsup_{k\to\infty}\begin{bmatrix}X_k\\Y_k\end{bmatrix}
  &\le(\bI-\bM_m)^{-1}\begin{bmatrix}m\varpi f_x\\f_y\end{bmatrix}\nonumber\\
  &=\frac{1}{1-\rho^{(m)}_{\rm cl}}
  \begin{bmatrix}m\varpi(f_x+\kappa_Uf_y)\\m\varpi c_ef_x+(1-\rho_m)f_y\end{bmatrix}.
  \label{eq:error_bound_m}
\end{align}
Moreover, if $\rho_x<1$, then (C4$_m$) holds with $m=1$, $\varpi=1$
and $\rho_1\le\rho_x$, so Theorem~\ref{thm:error_bound} is recovered
with $\rho_x$ replaced by the possibly smaller $\rho_1$.
\end{theorem}
\begin{proof}
See Appendix~\ref{app_multistep}.
\end{proof}

\begin{corollary}[Round-robin sensing]
\label{cor:roundrobin}
Consider the Kuramoto lifting of Example~\ref{ex:constructions} ($n=1$)
with $\bchi\in C^1$, sensors $\bH_t=\bC_t\Pip$, where $\bC_t$ selects the
nodes in $\calV^{\rm obs}_t$, and a filter with $\Pip\bK_t=k\,\bC_t^\top$,
$k\in(0,1]$. Let
$\varepsilon_c\triangleq\sup_{\bm\xi,\,\ba\in\calA_\star}\norm{\Pip\partial_{\bm\xi}[\bU(\bm\xi)\ba]}$.
If every node belongs to $\calV^{\rm obs}_t$ for at least one $t$ in
every $m$ consecutive steps, then (C4$_m$) holds with
\begin{equation}
  \rho_m\le(1+\varepsilon_c)^m-k,\qquad \varpi\le(1+\varepsilon_c)^{m-1},
  \label{eq:roundrobin}
\end{equation}
whenever the right-hand side of the first bound is below one, and
$\kappa_K\le1$. In contrast, $\rho_{\rm cl}\ge1$ at every step at which
$\calV^{\rm obs}_t\neq\calV$.
\end{corollary}
\begin{proof}
See Appendix~\ref{app_multistep}.
\end{proof}

Theorem~\ref{thm:multistep} separates the two obstructions that partial
sensing creates. The filter part of the loop, which fails to contract
in one step, may contract over a sensing cycle, as quantified by
Corollary~\ref{cor:roundrobin}: for weak coupling, $\rho_m\approx1-k$.
The topology-to-state channel remains: an unmeasured node receives the
full topology error during each step, which is reflected by the factor
$m\varpi$ in~\eqref{eq:small_gain_m}; with $\varepsilon_c\to0$,
condition~\eqref{eq:small_gain_m} reads approximately $m\kappa_Uc_e<k$.
As in Remark~\ref{rem:interpretation}(b), the ridge weight trades loop
gain for shrinkage bias. In-edges of a node that is never measured are
not identifiable from one-step prediction errors, and no $m$ satisfies (C4$_m$) in that
case. The gain $\Pip\bK_t=k\bC_t^\top$ of
Corollary~\ref{cor:roundrobin} is that of a fixed-gain filter; the
covariance-based gain of Algorithm~\ref{alg:koopman_gkfa} does not have
this structure in general, so the corollary shows that (C4$_m$) is
attainable but does not certify Algorithm~\ref{alg:koopman_gkfa}.

\subsection{The Sparse Regime}
\label{subsec:sparse}

Theorems~\ref{thm:topology_error}--\ref{thm:error_bound} rely on the
strong convexity (C3). The unregularized Gram matrix $\hG_{i,k}$ is
singular whenever the stacked weighted regression matrix of row $i$
lacks full column rank, in particular whenever $r_{i,k}<|\calI_i|$, with
$r_{i,k}$ the number of measurement rows of node $i$ in the window
(Remark~\ref{rem:design}(b)). The effective sample size
$n_{{\rm eff},k}$ governs statistical concentration
(Lemma~\ref{lem:subgauss}), not algebraic rank. In the singular case,
(C3) can hold only through $\mu_0$. The following results show that
sparsity then restores accuracy.

\begin{lemma}[Concentration of the martingale term]
\label{lem:subgauss}
Let Assumption~\ref{ass:dictionary}(A1)--(A2) hold, let
$w_k\sim\mathcal{N}(\bzero,\bQ_k)$ and $v_k\sim\mathcal{N}(\bzero,\bR_k)$, and
set $\bar q\triangleq\sup_k\norm{\bQ_k}$, $\bar r\triangleq\sup_k\norm{\bR_k}$,
$\sigma_{\rm sg}^2\triangleq\bar h^2L_\phi^2\bar q+\bar r$. Then, for every $k$
and $\delta\in(0,1)$, with probability at least $1-\delta$,
\begin{equation}
  \norm{\bvarsigma_k}_\infty\le\tau_k(\delta)\triangleq
  b_c\,\sigma_{\rm sg}\sqrt{\frac{2\log(2p_\calI/\delta)}{n_{{\rm eff},k}}}.
  \label{eq:tau}
\end{equation}
\end{lemma}
\begin{proof}
See Appendix~\ref{app_sparse}.
\end{proof}

The constant $\sigma_{\rm sg}$ involves operator norms of the noise
covariances, not traces, and $b_c$ is a column norm; neither grows with
$p$.

\begin{proposition}[Restricted-eigenvalue bound]
\label{prop:sparse}
Fix $k$, let $S_k\triangleq\supp(\ba_k)$ with $|S_k|\le s$, and suppose
that $\hG_{\mu,k}$ satisfies the restricted-eigenvalue condition
\begin{equation}
  \norm{\bv}_{\hG_{\mu,k}}^2\ge\kappa_{\rm RE}\norm{\bv}^2\quad
  \forall\,\bv\in\calC(S_k)\triangleq\{\bv:\norm{\bv_{S_k^c}}_1\le3\norm{\bv_{S_k}}_1\}.
  \label{eq:RE}
\end{equation}
Let $\bvartheta_k\triangleq\sum_{t\in\calT_k}\omega_{k,t}\hD_t^\top(\hat\by_t-\hD_t\ba_k)-\mu_0\ba_k$.
If $\lambda_k\ge2\norm{\bvartheta_k}_\infty$, then $\bDelta_k\in\calC(S_k)$ and
\begin{equation}
  \norm{\bDelta_k}\le\frac{3\lambda_k\sqrt s}{2\kappa_{\rm RE}},\qquad
  \norm{\bDelta_k}_1\le\frac{6\lambda_ks}{\kappa_{\rm RE}}.
  \label{eq:RE_bound}
\end{equation}
Moreover,
$\norm{\bvartheta_k}_\infty\le b_c\bigl(c_e[\norm{\be_{\cdot-1}}]_k+[\eta^{\rm b}]_k\bigr)
+\mu_0\norm{\ba_k}_\infty+\norm{\bvarsigma_k}_\infty$; hence, under the
conditions of Lemma~\ref{lem:subgauss}, the condition on $\lambda_k$
holds with probability at least $1-\delta$ whenever
\begin{equation}
  \lambda_k\ge2\Bigl(b_c\bigl(c_e[\norm{\be_{\cdot-1}}]_k+[\eta^{\rm b}]_k\bigr)
  +\mu_0\norm{\ba_k}_\infty+\tau_k(\delta)\Bigr).
  \label{eq:lambda_choice}
\end{equation}
\end{proposition}
\begin{proof}
See Appendix~\ref{app_sparse}.
\end{proof}

\begin{remark}
When the errors-in-variables and bias terms are small compared with
$\tau_k(\delta)$, i.e., at high measurement informativeness, with an exact
state block and $\mu_0\ll\tau_k(\delta)/\bar a$, the
choice $\lambda_k\asymp b_c\sigma_{\rm sg}\sqrt{\log(p_\calI/\delta)/n_{{\rm eff},k}}$
gives $\norm{\bDelta_k}\lesssim(b_c\sigma_{\rm sg}/\kappa_{\rm RE})
\sqrt{s\log(p_\calI/\delta)/n_{{\rm eff},k}}$, the familiar Lasso
rate~\cite{bickel2009simultaneous,negahban2012unified}, which remains
valid when $\hG_k$ is singular. Condition~\eqref{eq:RE} only involves
the cone $\calC(S_k)$ and can hold for \rev{$r_{i,k}<|\calI_i|$}. Because
\eqref{eq:lambda_choice} involves the realized state error,
Proposition~\ref{prop:sparse} is a per-step certificate; a closed-loop
high-probability analysis is left for future work. \rev{Under fixed
noise, the Lasso rate is recovered only when the errors-in-variables
and bias terms in~\eqref{eq:lambda_choice} are of smaller order than
$\tau_k(\delta)$; otherwise they set a floor that does not decrease with
$n_{{\rm eff},k}$.}
\end{remark}

\subsection{Time-Varying Topology and the Forgetting Factor}
\label{subsec:forgetting}

\begin{corollary}[Tracking bound]
\label{cor:tracking}
Under the conditions of Theorem~\ref{thm:error_bound} with $W=\infty$,
$\gamma\in(0,1)$ and $\lambda_k=\lambda_0\,n_{{\rm eff},k}^{-1/2}$, one has
$\bar n_{\rm eff}=\frac{1+\gamma}{1-\gamma}$ and
$\bar\ell^{(2)}=\frac{\sqrt{1+\gamma}}{1-\gamma}$. Hence, for loop
constants that do not depend on $\gamma$, the $\gamma$-dependent part of
each component of~\eqref{eq:error_bound} is
\begin{equation}
  h(\gamma)=c_s\sqrt{\frac{1-\gamma}{1+\gamma}}
  +c_\nu\,\nu\,\frac{\sqrt{1+\gamma}}{1-\gamma},
  \label{eq:tradeoff}
\end{equation}
with, e.g., for $X$:
$c_s=\kappa_U(\lambda_0\sqrt s+b_D\sigma_\varepsilon)/(\sqrt\mu(1-\rho_{\rm cl}))$ and
$c_\nu=\kappa_U\bar hb_U/(1-\rho_{\rm cl})$, and for $Y$ the same expressions
with $\kappa_U$ replaced by $1-\rho_x$.
(i) If $\nu=0$, $h$ is decreasing in $\gamma$, and the choice $\gamma=1$,
$W=\infty$ removes the statistical term. (ii) If $\nu>0$, $c_s>0$ and $c_\nu>0$, the upper bound
$\tilde h(u)\triangleq c_s\sqrt{u}+\sqrt2c_\nu\nu/u$ of $h$, with
$u=1-\gamma$, is decreasing for $u<u^\star$ and increasing for
$u>u^\star$, where $u^\star=(2\sqrt2c_\nu\nu/c_s)^{2/3}$. Hence, if
$u^\star<1$, it is minimized over $\gamma\in(0,1)$ at
$\gamma^\star=1-u^\star$, with minimal value
$\frac{3}{\sqrt2}c_s^{2/3}(c_\nu\nu)^{1/3}$; if $u^\star\ge1$, it is
decreasing on $(0,1)$ and its infimum is approached as $\gamma\downarrow0$.
\end{corollary}
\begin{proof}
See Appendix~\ref{app_cor1}.
\end{proof}

Corollary~\ref{cor:tracking} is a design heuristic rather than a sharp
rate: $\mu$, $\kappa_U$ and $\rho_{\rm cl}$ depend on $\gamma$ through the
window, the minimized quantity is an upper bound for fixed loop
constants, and the actual optimal forgetting factor need not coincide
with $\gamma^\star$. A rewiring of size $J$ at isolated instants
satisfies (C2) with $\nu\ge J$, so the theorems remain valid, but the
bound is then pessimistic between switches. The theorems provide no
transient decay rate after a switch: the topology recursion involves
the whole window, and $\rho_{\rm cl}<1$ is a
sufficient condition for the comparison system used in the proof, not
a decay rate of the windowed estimator. Beyond boundedness, we therefore
make no claim on the transient.

\subsection{Edge-Support Recovery}

\begin{corollary}[Support recovery by thresholding]
\label{cor:sparsity}
Let $a_{\min}\triangleq\inf_k\min_{(i,j)\in\supp(\ba_k)}|a_{ij,k}|>0$, $\tau\in(0,a_{\min})$, $\tau_\star\triangleq\min\{\tau,a_{\min}-\tau\}$, and
$\hat\calE_k=\{(i,j)\in\calI:|\hat a^{\rm alg}_{ij,k}|>\tau\}$. If
$\norm{\hat\ba_k^{\rm alg}-\ba_k}_\infty<\nrev{\tau_\star}$, then
$\hat\calE_k=\supp(\ba_k)$. Consequently:
(i) under the conditions of Theorem~\ref{thm:error_bound},
\begin{equation}
  \limsup_{k\to\infty}\prob\bigl(\hat\calE_k\neq\supp(\ba_k)\bigr)
  \le\frac{1}{\tau_\star^2}\Bigl(\frac{\bar Y}{\sqrt\mu}+\varepsilon_{\rm opt}\Bigr)^2,
  \label{eq:support_recovery}
\end{equation}
where $\bar Y$ is the second component of the bound~\eqref{eq:error_bound};
(ii) under the conditions of Proposition~\ref{prop:sparse} with
$\lambda_k$ satisfying~\eqref{eq:lambda_choice}, if
$3\lambda_k\sqrt s/(2\kappa_{\rm RE})+\varepsilon_{\rm opt}<\nrev{\tau_\star}$, then
$\hat\calE_k=\supp(\ba_k)$ with probability at least $1-\delta$. The threshold $\tau=a_{\min}/2$ maximizes $\tau_\star$.
\end{corollary}
\begin{proof}
See Appendix~\ref{app_cor2}.
\end{proof}

Part (i) is a direct consequence of Markov's inequality and inherits
the scaling of Theorem~\ref{thm:error_bound}; part (ii) inherits the
$\sqrt{s\log p_\calI/n_{\rm eff}}$ scaling of Proposition~\ref{prop:sparse}.
Neither requires an irrepresentability
condition~\cite{zhao2006model,bach2008consistency}, which would have to
hold for regressors built from estimated states; the thresholded
estimator requires only a beta-min condition. The edge-level analogue
for an unstructured lifted estimate is Lemma~\ref{lem:homomorphism}(ii).

\subsection{Dictionary Refinement}

\begin{corollary}[Dictionary refinement]
\label{cor:consistency}
Let $\bphi^{(1)}\subset\bphi^{(2)}\subset\cdots$ be nested state-inclusive
node dictionaries and $\bchi^{(1)}\subset\bchi^{(2)}\subset\cdots$ nested
interaction dictionaries, let $\nu_{\rm cal}$ be the calibration
distribution of $x$, and define the best state-coordinate residual
\begin{equation}
  \varepsilon_{x}(m)\triangleq\min_{\{\bF_i,\bB_i\}}
  \Bigl(\expect_{\nu_{\rm cal}}\sum_i\norm{\bPi\bdelta_i(x)}^2\Bigr)^{1/2}
  \label{eq:eps_x}
\end{equation}
for the $m$-th pair. (i) $\varepsilon_x(m)$ is nonincreasing in $m$, and
$\varepsilon_x(m)=0$ as soon as $f_i\in\mathrm{span}\{\bphi^{(m)}\}$ and
$g\in\mathrm{span}\{\bchi^{(m)}\}$ componentwise. (ii) The full residual
$\varepsilon_{\rm rep}$ is not monotone in general: each added non-affine
observable incurs the second-order coupling remainder of
Theorem~\ref{thm:lifting}(i), and the constants $L_\phi,L_2,b_U,\mu,\rho_{\rm cl},c_e,\kappa_U$ of~\eqref{eq:error_bound} also depend on
the dictionaries. In particular, for linear $f_i,g$ and a node with
two in-neighbours, $\varepsilon_{\rm rep}=0$ for $\bphi^{(1)}(x)=x$,
whereas $\varepsilon_{\rm rep}>0$ for $\bphi^{(2)}(x)=[x,x^2]^\top$ with
\emph{any} pairwise interaction dictionary. Hence the
bound~\eqref{eq:error_bound} is in general not monotone in the
dictionary size.
\end{corollary}
\begin{proof}
See Appendix~\ref{app_cor1}.
\end{proof}

Corollary~\ref{cor:consistency} shows that local density of the
dictionary implies neither a vanishing network residual
(Proposition~\ref{prop:separable_fails}) nor a monotone decrease of the
estimation error. With linear sensors, the topology bound
depends on the representation only through the deployment residual
$\varepsilon_x$ (Theorem~\ref{thm:topology_error}), and part (i)
shows that its best achievable value $\varepsilon_x(m)$ on the
calibration distribution is nonincreasing. The deployment residual
itself is not monotone in general, because it also contains the
estimation error of the learned matrices and the shift from the
calibration to the operational distribution; the non-monotone part of the full residual acts on the state error through $f_x$. In practice,
the dictionaries are selected by minimizing the validation
residual~\eqref{eq:val_residual} subject to empirical diagnostics
(Remark~\ref{rem:verifiable}(b)): a lower bound on the excitation
$\mu_k$ and an upper bound on the loop gain $\rho_{\rm cl}$ evaluated on
calibration runs.

\section{Experimental Results}
\label{sec:simulation}
%% ===============================================================
%% ===============================================================

\subsection{Synthetic Kuramoto Networks}
\label{subsec:synthetic}

Synthetic networks are used to evaluate the method against known ground
truth and to vary, one at a time, the factors that the analysis of
Section~\ref{sec:theory} identifies as decisive: measurement noise,
coupling strength, sensing pattern and topology variation.

\subsubsection{Data Generation}
\label{subsubsec:data}

We simulate the discretized Kuramoto network of
Example~\ref{ex:constructions},
\begin{align}
  \theta_{i,k+1}&=\theta_{i,k}+\omega_i\Delta t+w_{i,k}\nonumber\\
  &\quad+\frac{K\Delta t}{p}\sum_{j=1}^{p}a_{ij,k}\sin(\theta_{j,k}-\theta_{i,k}),\nonumber\\
  z_{i,k}&=\theta_{i,k}+v_{i,k},
  \label{eq:sim_kuramoto}
\end{align}
with $p=20$ nodes, $\Delta t=0.05$\,s, natural frequencies
$\omega_i\sim\mathcal{U}[-1.5,1.5]$\,rad/s and initial phases
$\theta_{i,0}\sim\mathcal{U}[-\pi,\pi]$. The noises are Gaussian with
per-node standard deviations $\bar\sigma_w=0.01$\,rad and
$\bar\sigma_v=0.02$\,rad, i.e., $\sigma_w^2=p\bar\sigma_w^2$ and
$\sigma_v^2=p\bar\sigma_v^2$ in the notation of Section~\ref{subsec:model}.
Three directed graph families are considered: Erd\H{o}s--R\'enyi (ER)
graphs with edge probability $0.2$~\cite{erdos1959random};
Watts--Strogatz (WS) graphs in which each node receives from its four
nearest neighbours on a ring, rewired with probability
$0.1$~\cite{watts1998collective}; and Barab\'asi--Albert (BA) graphs with
attachment parameter $2$~\cite{barabasi1999emergence}, each undirected
edge being replaced by two directed edges. All families have a mean
in-degree close to four, the BA family with a heavy-tailed in-degree
distribution. Edge weights are drawn from $\mathcal{U}[0.5,1.5]$ and
graphs that are not weakly connected are discarded. Unless stated
otherwise, $K=4$; the network is then partially synchronized but not
phase locked, with time-averaged order parameter
$\bar r\triangleq\frac1T\sum_k|\frac1p\sum_i e^{\imath\theta_{i,k}}|\approx0.35$;
this indicates partial synchronization, but the order parameter alone
does not certify that the regressors of
Remark~\ref{rem:identifiability}(c) remain excited.

Following the information pattern (I1)--(I5), every Monte Carlo run
draws two \emph{independent} graphs from the same family: the reference
topology $\bA^{\rm ref}$, under which $10$ calibration trajectories of
$1000$ steps with full-state records are generated, and the operational
topology $\bA$, under which a test record of $T=3000$ steps is
generated. Hence the calibration data carry no information about the
operational edge set. Two time-varying scenarios are added:
(i) an abrupt switch at $k=1500$, at which $30\%$ of the edges are removed
and the same number of new edges with fresh weights are inserted; and
(ii) a continuous drift on a fixed edge set, in which the weights follow
a reflected random walk on $[0.5,1.5]$ with relative increment
$\delta=\norm{\bA_{k+1}-\bA_k}_{\rm F}/\norm{\bA_k}_{\rm F}\in
\{10^{-4},5\times10^{-4},2\times10^{-3}\}$.
All results are averaged over $50$ independent runs, each with new
graphs, frequencies, initial conditions and noise realizations.

\subsubsection{Preprocessing and Calibration}
\label{subsubsec:preproc}

Phases are kept unwrapped, as required by the affine self-dynamics of
Example~\ref{ex:constructions}, and estimation errors are wrapped to
$(-\pi,\pi]$ before evaluation; the wrapped error coincides with
the Euclidean error of Section~\ref{sec:theory} as long as
$|\hat\theta_{i,k}-\theta_{i,k}|<\pi$, i.e., in the absence of cycle
slips. The calibration trajectories are split
into training and validation parts in the ratio $80{:}20$. We use the
node dictionary $\bphi(\theta)=[\theta,1,\cos\theta,\sin\theta]^\top$
($N=4$) and the tensor-product interaction dictionary of
Example~\ref{ex:constructions} ($N_e=10$), for which an exact state
block exists and the only remaining approximation error is the
second-order coupling remainder. The matrices $(\bF_i,\bB_i)$ are
obtained from the (P1)
problem~\eqref{eq:calibration} with the known $\bA^{\rm ref}$. Since the simulation provides the noise-free successors
$\tilde x_{\ell+1}$, we evaluate the representation residual itself,
i.e.,~\eqref{eq:val_residual} with $\bPhi(x_{\ell+1})$ replaced by
$\bPhi(\tilde x_{\ell+1})$ and normalized by $\sqrt p$. Its value,
$(1.6\pm0.2)\times10^{-3}$, lies below the uniform bound
$\frac12(d\,a_{\max}K\Delta t/p)^2$ of Example~\ref{ex:constructions},
in which $d$ is the \emph{maximum} in-degree of $\bA^{\rm ref}$ (the bound
equals $1.8\times10^{-3}$ at $d=4$ and grows as $d^2$). The validation
residual~\eqref{eq:val_residual} also contains the process noise:
by~\eqref{eq:val_split}, the mean square of its state block exceeds
$p\bar\sigma_w^2$, so its RMS is at least
$\sqrt p\,\bar\sigma_w\approx4.5\times10^{-2}$\,rad; it is used to set
$\bar\bQ$. All hyperparameters
of all methods are selected on the validation part of the calibration
data; the operational records are used only for testing.

\subsubsection{Implementation and Evaluation Metrics}
\label{subsubsec:metrics}

Koopman-GKFA (Algorithm~\ref{alg:koopman_gkfa}) uses $W=\infty$,
$\gamma=0.999$ in static and $\gamma=0.995$ in time-varying scenarios,
$\lambda_k=\lambda_0 n_{{\rm eff},k}^{-1/2}$ with $\lambda_0$ selected on
the validation data, $\mu_0=10^{-3}$, $\beta=1$,
$\varepsilon_{\rm opt}=10^{-6}$, the box $\calA=[0,2]^{p_\calI}$ with zero
diagonal, the edge threshold $\tau=0.2$ (so that $\tau_\star=0.2$ in
Corollary~\ref{cor:sparsity}, since $a_{\min}=0.5$), $\hat\ba_0=\bzero$ and
$\bP_0=\frac{\pi^2}{3}\bI$. State accuracy is measured by
\begin{equation}
  \mathrm{RMSE}=\Bigl(\frac{1}{Mp|\calK_s|}\sum_{r=1}^{M}\sum_{k\in\calK_s}
  \norm{\hat x_k^{(r)}-x_k^{(r)}}^2\Bigr)^{1/2},
  \label{eq:rmse}
\end{equation}
where $M=50$ and $\calK_s=\{1001,\ldots,3000\}$ excludes the initial
transient. Topology accuracy is measured by the normalized error
$\mathrm{NMSE}_k=10\log_{10}(\normF{\hat\bA_k-\bA_k}^2/\normF{\bA_k}^2)$,
averaged over $\calK_s$, and by the $F_1$ score of the thresholded edge
set $\hat\calE_k$ against the true off-diagonal support. The Posterior
Cram\'er--Rao Lower Bound (PCRLB) is computed by the recursion
of~\cite{tichavsky1998posterior} with the true model and the true
topology; it bounds the state error of any estimator that \emph{knows}
$\bA$ and is therefore a conservative reference for joint estimation.
Run times are measured per time step on an Intel Core i7-12650H, 16GB
RAM thread with a common
NumPy implementation for all non-neural methods.

\subsubsection{Baselines}
\label{subsubsec:baselines}

The baselines cover the three routes discussed in
Section~\ref{sec:intro}, and each is configured so that its assumptions
are favoured wherever possible.
\emph{(i) Known-topology filters}: GEKF~\cite{sagi2023extended},
GUKF~\cite{li2023unscented} and GSP-KalmanNet~\cite{buchnik2024gspkalmannet},
all run with $\bA^{\rm ref}$, which is the best topology available
offline; GEKF and GUKF also receive the true $f_i$ and $g$, and
GSP-KalmanNet is trained with state labels on the calibration data.
\emph{(ii) Joint estimators for linear models}: KF-ADMM~\cite{fang2025joint},
the reference online method for linear graphical systems, and the batch
method GraphEM~\cite{elvira2022graphical}; both use linear models fitted to
the calibration data. GraphEM processes the entire operational record
offline ($50$ EM iterations), and its state estimates are produced by a
Kalman filter with the final model; its state recursion is causal,
but its model and topology use the entire record, so that GraphEM is an
offline reference with access to future data.
\emph{(iii) Joint estimators for nonlinear models}: the joint EKF and the
joint UKF~\cite{julier2004unscented,wan1997dual}, which append the
$p(p-1)=380$ free weights to the state and use the \emph{true} $f_i$ and
$g$, and the model-based deep learning method
GKNet~\cite{sabbaqi2026gknet}. GKNet is trained on the calibration data
and its graph module is then fine-tuned, by its unsupervised prediction
loss, on a $500$-step segment of operational measurements that precedes
the test window. The joint EKF and UKF thus enjoy exact model knowledge
and GKNet uses operational data that Koopman-GKFA does not use
(cf.~(I5)). For reference, we also report an oracle EKF that knows the
true $f_i$, $g$ and $\bA$. All topology estimators start from
$\hat\bA_0=\bO$ and are evaluated with the same threshold $\tau$. The baselines differ from Koopman-GKFA in more than one component
(state-space model, topology update and tuning), so the comparison
evaluates complete methods. By Proposition~\ref{prop:ekf}, isolating the contribution of the
lifting requires a physical-space surrogate learned from the same data
and combined with the same topology estimator; this comparison is
the subject of Section~\ref{subsec:controlled}.

\subsubsection{Overall Comparison}
\label{subsubsec:main_results}

\begin{table*}[!t]
\centering
\caption{Steady-state performance on synthetic Kuramoto networks
($p=20$; mean\,$\pm$\,std over $50$ Monte Carlo runs). RMSE in
$10^{-2}$\,rad, NMSE in dB. Best results in bold and second best
underlined, reference rows excluded.}
\label{tab:main}
\footnotesize
\setlength{\tabcolsep}{3.1pt}
\renewcommand{\arraystretch}{1.08}
\resizebox{\textwidth}{!}{%
\begin{tabular}{l ccc ccc ccc c}
\toprule
\multirow{2}{*}{Method} & \multicolumn{3}{c}{Erd\H{o}s--R\'enyi} & \multicolumn{3}{c}{Watts--Strogatz} & \multicolumn{3}{c}{Barab\'asi--Albert} & \multirow{2}{*}{\shortstack{Time\\(ms/step)}}\\
\cmidrule(lr){2-4}\cmidrule(lr){5-7}\cmidrule(lr){8-10}
 & RMSE & NMSE & $F_1$ & RMSE & NMSE & $F_1$ & RMSE & NMSE & $F_1$ & \\
\midrule
\multicolumn{11}{l}{\textit{Reference (true topology and model)}}\\
PCRLB~\cite{tichavsky1998posterior} & 1.27 & -- & -- & 1.27 & -- & -- & 1.28 & -- & -- & --\\
Oracle EKF & 1.29$\pm$0.02 & -- & -- & 1.29$\pm$0.01 & -- & -- & 1.30$\pm$0.03 & -- & -- & 0.4\\
\midrule
\multicolumn{11}{l}{\textit{Known-topology filters (run with $\bA^{\rm ref}$)}}\\
GEKF~\cite{sagi2023extended} & 1.84$\pm$0.11 & -- & -- & 1.77$\pm$0.10 & -- & -- & 1.93$\pm$0.15 & -- & -- & \textbf{0.4}\\
GUKF~\cite{li2023unscented} & 1.81$\pm$0.10 & -- & -- & 1.75$\pm$0.10 & -- & -- & 1.90$\pm$0.14 & -- & -- & 1.1\\
GSP-KalmanNet~\cite{buchnik2024gspkalmannet} & 1.66$\pm$0.09 & -- & -- & 1.62$\pm$0.08 & -- & -- & 1.72$\pm$0.12 & -- & -- & 1.2\\
\midrule
\multicolumn{11}{l}{\textit{Joint state estimation and topology inference}}\\
KF-ADMM~\cite{fang2025joint} & 1.76$\pm$0.10 & $-$4.6$\pm$1.3 & 0.607$\pm$0.055 & 1.71$\pm$0.09 & $-$4.3$\pm$1.2 & 0.596$\pm$0.052 & 1.83$\pm$0.13 & $-$3.9$\pm$1.5 & 0.571$\pm$0.061 & \underline{0.9}\\
GraphEM~\cite{elvira2022graphical} & 1.63$\pm$0.08 & $-$5.3$\pm$1.1 & 0.642$\pm$0.051 & 1.59$\pm$0.08 & $-$5.0$\pm$1.1 & 0.633$\pm$0.049 & 1.70$\pm$0.11 & $-$4.7$\pm$1.3 & 0.618$\pm$0.057 & 31.5$^{\dagger}$\\
Joint EKF & 1.43$\pm$0.08 & $-$10.6$\pm$2.1 & 0.828$\pm$0.047 & 1.41$\pm$0.07 & $-$10.2$\pm$1.9 & 0.817$\pm$0.044 & 1.49$\pm$0.10 & $-$9.1$\pm$2.4 & 0.789$\pm$0.056 & 9.8\\
Joint UKF & \underline{1.39$\pm$0.06} & \underline{$-$11.8$\pm$1.7} & \underline{0.861$\pm$0.038} & \underline{1.38$\pm$0.06} & \underline{$-$11.3$\pm$1.6} & \underline{0.852$\pm$0.037} & \underline{1.42$\pm$0.08} & \underline{$-$10.6$\pm$1.9} & \underline{0.837$\pm$0.043} & 84.3\\
GKNet~\cite{sabbaqi2026gknet} & 1.45$\pm$0.07 & $-$10.2$\pm$1.5 & 0.813$\pm$0.036 & 1.44$\pm$0.06 & $-$9.8$\pm$1.4 & 0.806$\pm$0.035 & 1.48$\pm$0.08 & $-$9.3$\pm$1.7 & 0.792$\pm$0.041 & 1.6\\
\midrule
Koopman-GKFA & \textbf{1.34$\pm$0.04} & \textbf{$-$15.1$\pm$1.2} & \textbf{0.942$\pm$0.021} & \textbf{1.33$\pm$0.04} & \textbf{$-$14.4$\pm$1.1} & \textbf{0.931$\pm$0.022} & \textbf{1.38$\pm$0.06} & \textbf{$-$12.9$\pm$1.6} & \textbf{0.904$\pm$0.031} & 2.7\\
\bottomrule
\multicolumn{11}{l}{\scriptsize $^{\dagger}$Offline batch processing of the whole record, amortized per time step. --: not applicable.}
\end{tabular}}
\end{table*}

Table~\ref{tab:main} summarizes the steady-state results. On all three
graph families, Koopman-GKFA attains the lowest state RMSE and the best
topology scores. On ER graphs its RMSE is $5.5\%$ above the PCRLB and
$3.9\%$ above the oracle EKF, which knows $\bA$; the remaining gap combines the cost of estimating $380$ weights online with the
calibration error of the learned surrogate and with filter tuning,
which this comparison does not separate. Its topology NMSE improves on
that of the strongest baseline, the joint UKF, by $3.3$\,dB, and its
$F_1$ score by $0.08$. Three further observations deserve comment.
First, the known-topology filters show that a wrong topology is costly
even when the node and interaction models are exact: with
$\bA^{\rm ref}$, GEKF and GUKF exceed the PCRLB by $43$--$45\%$, more than
any joint estimator for nonlinear models. Second, KF-ADMM and GraphEM
remain far from the other joint estimators because a linear coupling
cannot represent $\sin(\theta_j-\theta_i)$ on an incoherent trajectory. Third, the joint UKF is the most accurate baseline but is about $31$ times
slower than Koopman-GKFA, because its $801$ sigma points must be
propagated through the $400$-dimensional augmented state; GEKF, GUKF,
GSP-KalmanNet, KF-ADMM and GKNet are faster than Koopman-GKFA, but none
of them provides a comparably accurate topology. Performance degrades
for all methods on BA graphs, whose hubs have more in-neighbours,
so that their rows contain more unknowns. Since the
ER, WS and BA results lead to the same conclusions, the remaining
experiments use ER graphs.

\begin{figure*}[!t]
\centering
\includegraphics[width=\textwidth]{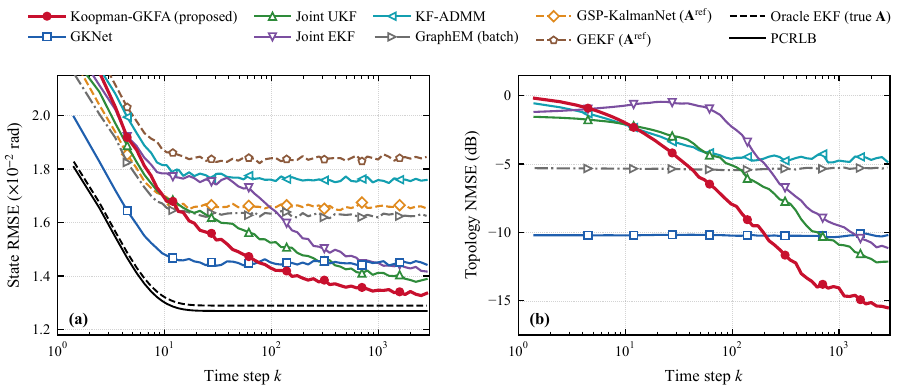}
\caption{Transient behaviour on ER graphs with a static topology
(averages over $50$ runs, log-spaced bins). (a) State RMSE with the
oracle EKF and the PCRLB as references. (b) Topology NMSE; GKNet and
GraphEM are constant because their graphs are fixed offline.}
\label{fig:convergence}
\end{figure*}

\subsubsection{Convergence}
\label{subsubsec:convergence}

Fig.~\ref{fig:convergence} shows how the estimates evolve. Starting from
$\hat\bA_0=\bO$, all online joint estimators initially behave as filters
without coupling, and their state errors decrease as the topology is
learned. Koopman-GKFA reaches $-10$\,dB after $175$ steps, compared with
$670$ and $1253$ steps for the joint UKF and EKF, and its NMSE continues
to decrease until the forgetting window saturates. This is
consistent with the decay of the statistical term of
Theorem~\ref{thm:topology_error}, whose square scales as
$n_{\rm eff}^{-1}$; the theorem does not predict the rate of the total
error, which also contains the errors-in-variables, bias and
regularization terms.
The joint EKF stalls during the first $\sim\!60$ steps, a manifestation
of the linearization bias of parameter-augmented filters
\cite{ljung1979asymptotic}. GKNet is more accurate than all online
methods during the first $\sim\!150$ steps because its graph has been
fitted offline, but this graph is not refined afterwards. The state RMSE
of Koopman-GKFA settles close to the oracle EKF, which reflects
Proposition~\ref{prop:ekf}: once $\hat\bA_k$ is accurate, the re-lifted
filter behaves as the EKF of the learned surrogate, whose state rows
are close to the true model when the calibration is accurate.

\begin{figure*}[!t]
\centering
\includegraphics[width=\textwidth]{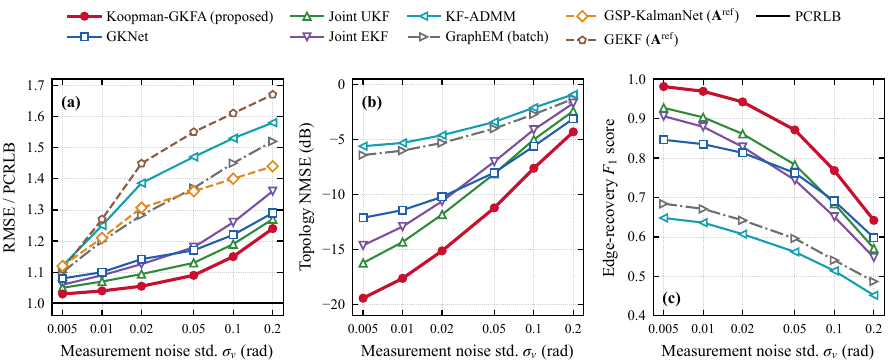}
\caption{Effect of the measurement noise level $\bar\sigma_v$ (ER graphs,
$K=4$). (a) State RMSE normalized by the PCRLB. (b) Topology NMSE.
(c) $F_1$ score of the recovered edge set.}
\label{fig:noise}
\end{figure*}

\subsubsection{Robustness to Measurement Noise}
\label{subsubsec:noise}

Fig.~\ref{fig:noise} varies $\bar\sigma_v$ from $0.005$ to $0.2$\,rad,
i.e., from one half to twenty times the per-node process noise. At low noise,
Koopman-GKFA is within $3\%$ of the PCRLB and achieves $-19.4$\,dB NMSE
and $F_1=0.981$. As $\bar\sigma_v$ increases, all methods deteriorate and
their differences shrink: at $\bar\sigma_v=0.2$\,rad the advantage of
Koopman-GKFA over GKNet reduces to $1.2$\,dB in NMSE and $0.04$ in $F_1$,
and no method recovers the topology reliably. This is consistent
with Remark~\ref{rem:interpretation}(c): the prediction errors become
uninformative relative to the measurement noise, so the topology-to-state
gain increases and the small-gain certificate ceases to apply. The
known-topology filters degrade fastest in relative terms, since the
filter relies increasingly on a model with the wrong graph.

\begin{figure*}[!t]
\centering
\includegraphics[width=\textwidth]{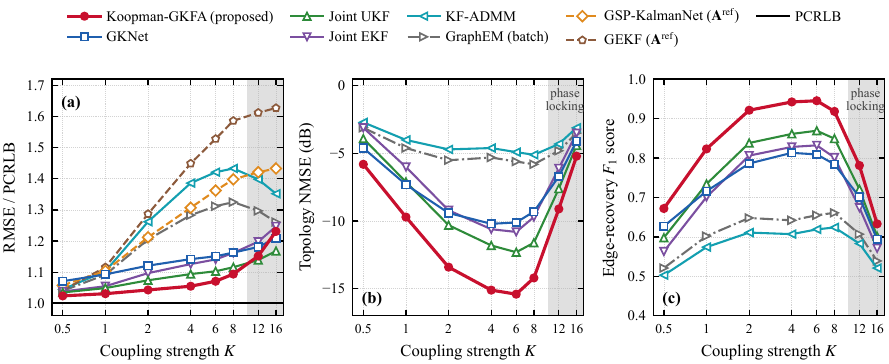}
\caption{Effect of the coupling strength $K$ (ER graphs,
$\bar\sigma_v=0.02$\,rad). (a) State RMSE normalized by the PCRLB.
(b) Topology NMSE. (c) $F_1$ score. The shaded region marks partial phase
locking ($\bar r\ge0.86$).}
\label{fig:coupling}
\end{figure*}

\subsubsection{Coupling Strength and Nonlinearity}
\label{subsubsec:coupling}

Fig.~\ref{fig:coupling} varies $K$ from $0.5$ to $16$, which moves the
network from incoherence ($\bar r=0.08$) to partial phase locking
($\bar r=0.93$). Topology accuracy is non-monotone for every method. For
weak coupling the edge signal $K\Delta t/p$ is small relative to the
noise; for strong coupling the phase differences become nearly constant,
the regressors lose excitation, and only weighted row sums remain
identifiable (Remark~\ref{rem:identifiability}(c)). Koopman-GKFA performs
best over the whole range, with a broad optimum at $K\in[4,6]$, where its
NMSE is $3.1$--$3.3$\,dB lower than that of the joint UKF.

Panel~(a) shows a limitation of the method.
The RMSE of Koopman-GKFA relative to the PCRLB rises from $1.02$ to $1.23$
as $K$ increases and exceeds that of the joint UKF for $K\ge12$ and that
of GKNet at $K=16$. This degradation is not caused by the
second-order remainder of Remark~\ref{rem:first_order}: with the linear
sensor $z_{i,k}=\theta_{i,k}+v_{i,k}$, the physical estimate depends only
on the state rows of the learned model (Proposition~\ref{prop:ekf}),
which contain no such remainder (Theorem~\ref{thm:lifting}(ii)). Two
mechanisms are consistent with the analysis: near phase locking the
regressors lose excitation, which decreases $\mu$ and increases the
topology-to-state gain $\kappa_U$ (Remark~\ref{rem:interpretation}), and
the first-order linearization of the coupling in the EKF becomes less
accurate, whereas the joint UKF propagates sigma points through the
exact model. The experiment does not separate the two. Conversely, the linear
model of KF-ADMM becomes more accurate near phase locking, where
$\sin(\theta_j-\theta_i)\approx\theta_j-\theta_i$, and its relative RMSE
decreases for $K>8$. The interaction-aware lifting is thus most
effective in the moderately coupled regime that makes the topology
identifiable in the first place.

\begin{figure*}[!t]
\centering
\includegraphics[width=\textwidth]{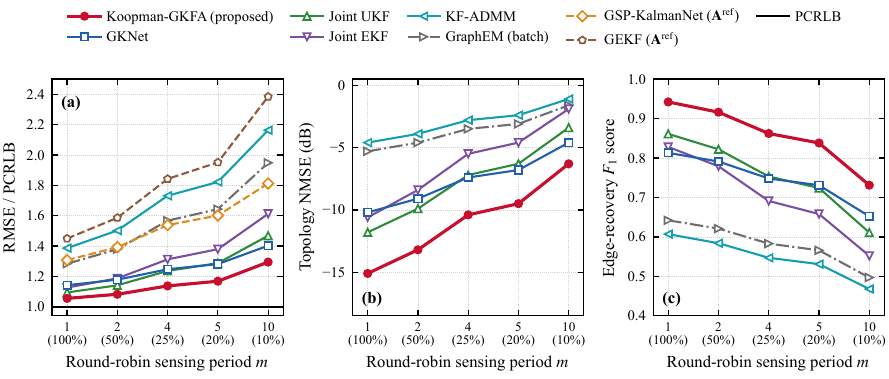}
\caption{Round-robin sensing with period $m$, i.e., $p/m$ nodes measured
per step (ER graphs, $K=4$). (a) State RMSE normalized by the PCRLB of the
same sensing schedule. (b) Topology NMSE. (c) $F_1$ score.}
\label{fig:sensing}
\end{figure*}

\subsubsection{Intermittent Sensing}
\label{subsubsec:sensing}

Fig.~\ref{fig:sensing} considers round-robin schedules in which each node
is measured once every $m\in\{1,2,4,5,10\}$ steps, the sensing
pattern of Corollary~\ref{cor:roundrobin}; the filter uses the
covariance-based gain of Algorithm~\ref{alg:koopman_gkfa}, not the
fixed gain of the corollary. By Proposition~\ref{prop:onestep}, the
one-step certificate fails for every $m>1$, whereas the multi-step
condition of Theorem~\ref{thm:multistep} may still hold; since
$\rho_m$ is not evaluated along the runs, the following observations
are qualitative. Empirically,
Koopman-GKFA degrades gracefully: its relative RMSE grows from $1.06$ to
$1.29$ and its NMSE from $-15.1$ to $-6.3$\,dB as the sensed fraction
drops from $100\%$ to $10\%$. The joint EKF is the most sensitive
estimator ($-10.6$ to $-1.9$\,dB), because its parameter covariance is
updated through a linearization around poorly observed states. The
advantage of Koopman-GKFA narrows with $m$, to $1.7$\,dB over GKNet at
$m=10$; this is consistent with the factor $m\varpi$ in~\eqref{eq:small_gain_m},
according to which an unmeasured node accumulates the topology error
over a whole sensing cycle. Every in-edge of a node is observed only
once per cycle, so the effective sample size of each row of $\hat\bA_k$
shrinks by the factor $m$ (Remark~\ref{rem:design}(b)).

\begin{figure*}[!t]
\centering
\includegraphics[width=\textwidth]{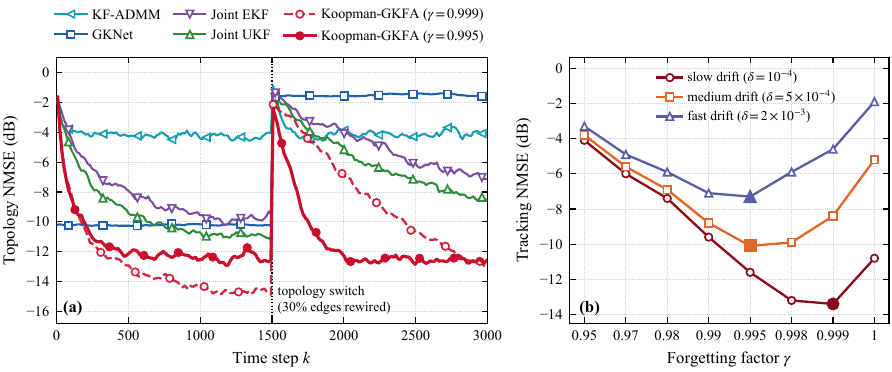}
\caption{Time-varying topologies (ER graphs). (a) Abrupt switch at
$k=1500$ in which $30\%$ of the edges are rewired. (b) Tracking NMSE of
Koopman-GKFA under continuous weight drift versus the forgetting factor
$\gamma$; filled markers indicate the best $\gamma$ for each drift rate
$\delta$.}
\label{fig:tracking}
\end{figure*}

\subsubsection{Time-Varying Topology}
\label{subsubsec:tracking}

Fig.~\ref{fig:tracking}(a) shows the response to an abrupt topology
switch. With $\gamma=0.995$, Koopman-GKFA returns to its pre-switch level
of about $-12.5$\,dB within roughly $300$ steps and averages $-11.2$\,dB
over $k\in[1501,3000]$. With $\gamma=0.999$, the pre-switch error is
lower ($-14.7$\,dB at $k=1500$), but recovery takes more than $1000$
steps. The joint UKF and EKF, whose random-walk parameter models are
tuned for the static phase, recover only to $-8.5$ and $-7.1$\,dB by
$k=3000$. GKNet does not adapt, and its error remains at
$-1.6$\,dB after the switch. KF-ADMM recovers quickly owing to its own
forgetting mechanism, but to its pre-switch level of about $-4.2$\,dB. These recovery times are empirical: the theory provides ultimate
bounds but no transient rate (Section~\ref{subsec:forgetting}).

Fig.~\ref{fig:tracking}(b) quantifies the trade-off of
Corollary~\ref{cor:tracking}. For each drift rate, the tracking NMSE is
a U-shaped function of $\gamma$: small $\gamma$ yields a small effective
sample size and a large variance, whereas $\gamma$ close to one averages
over outdated topologies. The optimum moves from $\gamma=0.999$ for slow
drift to $\gamma=0.995$ for medium and fast drift, and for fast drift the
curve becomes flat between $0.99$ and $0.995$, where variance and lag
contributions balance, consistently with~\eqref{eq:tradeoff}.
Because $\nu$ is not observable online, $\gamma$ is a design choice
that reflects the expected rate of topology variation.

\subsubsection{Ablation Study}
\label{subsubsec:ablation}

\begin{table*}[!t]
\centering
\caption{Ablation study on ER graphs ($p=20$; mean\,$\pm$\,std over $50$
runs). RMSE in $10^{-2}$\,rad, NMSE in dB. Time-varying: switch at
$k=1500$, NMSE averaged over $k\in[1501,3000]$. $k_{-10\,\mathrm{dB}}$:
first step at which the mean NMSE reaches $-10$\,dB (static, full
sensing). Best results in bold and second best underlined.}
\label{tab:ablation}
\footnotesize
\setlength{\tabcolsep}{3.3pt}
\renewcommand{\arraystretch}{1.1}
\resizebox{\textwidth}{!}{%
\begin{tabular}{l l ccc ccc c c c}
\toprule
 & \multirow{2}{*}{Variant} & \multicolumn{3}{c}{Static, full sensing} & \multicolumn{3}{c}{Static, round-robin ($m=4$)} & Time-varying & \multirow{2}{*}{$k_{-10\,\mathrm{dB}}$} & \multirow{2}{*}{\shortstack{Time\\(ms/step)}}\\
\cmidrule(lr){3-5}\cmidrule(lr){6-8}\cmidrule(lr){9-9}
 & & RMSE & NMSE & $F_1$ & RMSE & NMSE & $F_1$ & NMSE & & \\
\midrule
(A0) & Koopman-GKFA (full) & \underline{1.34$\pm$0.04} & $-$15.1$\pm$1.2 & 0.942$\pm$0.021 & \underline{2.30$\pm$0.09} & \underline{$-$10.4$\pm$1.5} & \underline{0.862$\pm$0.034} & \textbf{$-$11.2$\pm$1.4} & 175 & 2.7\\
\midrule
(A1) & Node-separable lifting & 1.58$\pm$0.07 & $-$6.9$\pm$1.6 & 0.741$\pm$0.049 & 2.71$\pm$0.13 & $-$4.8$\pm$1.8 & 0.652$\pm$0.058 & $-$4.6$\pm$1.7 & $>$3000 & \underline{2.1}\\
(A2) & w/o re-lifting (lifted KF) & 1.47$\pm$0.06 & $-$11.3$\pm$1.8 & 0.887$\pm$0.033 & 2.96$\pm$0.17 & $-$6.1$\pm$2.2 & 0.734$\pm$0.051 & $-$7.4$\pm$1.9 & 1186 & 2.6\\
(A3) & Filtered-increment regression & 1.37$\pm$0.05 & $-$12.2$\pm$1.6 & 0.903$\pm$0.029 & 2.47$\pm$0.11 & $-$7.9$\pm$1.9 & 0.801$\pm$0.042 & $-$5.8$\pm$1.8 & 694 & 2.6\\
(A4) & $\ell_1$ only ($\mu_0=0$) & \underline{1.34$\pm$0.04} & \underline{$-$15.3$\pm$1.2} & \underline{0.945$\pm$0.021} & 2.41$\pm$0.14 & $-$8.2$\pm$2.6 & 0.817$\pm$0.047 & \underline{$-$11.0$\pm$1.5} & \textbf{168} & 3.6\\
(A5) & Ridge only ($\lambda_k=0$) & 1.35$\pm$0.04 & $-$13.4$\pm$1.1 & 0.873$\pm$0.030 & 2.33$\pm$0.10 & $-$9.6$\pm$1.4 & 0.811$\pm$0.039 & $-$10.3$\pm$1.3 & 214 & 2.2\\
(A6) & No forgetting ($\gamma=1$) & \textbf{1.33$\pm$0.04} & \textbf{$-$15.8$\pm$1.1} & \textbf{0.951$\pm$0.019} & \textbf{2.28$\pm$0.09} & \textbf{$-$11.0$\pm$1.4} & \textbf{0.874$\pm$0.032} & $-$5.2$\pm$1.2 & \underline{174} & 2.7\\
(A7) & Single ADMM iteration ($R=1$) & 1.35$\pm$0.04 & $-$14.6$\pm$1.3 & 0.934$\pm$0.023 & 2.32$\pm$0.09 & $-$9.8$\pm$1.6 & 0.849$\pm$0.036 & $-$10.6$\pm$1.5 & 196 & \textbf{1.9}\\
\bottomrule
\end{tabular}}
\end{table*}

Table~\ref{tab:ablation} isolates the design choices of
Section~\ref{sec:algorithm}; each variant changes one component of the
full method (A0).

\emph{Representation.} Replacing the interaction-aware lifting by a
node-separable lifting with sender-only coupling $(\bA\otimes\bGamma)\bPsi$
(A1) causes the largest loss, $8.2$\,dB in NMSE, and the NMSE never reaches
$-10$\,dB. Because the sender-only structure cannot represent the
receiver dependence of $\sin(\theta_j-\theta_i)$, (A1) changes the
physical interaction model as well as the lifting structure, and the
loss reflects their combined effect. With the linear phase sensor, the
estimates depend only on the state rows (Proposition~\ref{prop:ekf}), so
the relevant error is the misspecified coupling in the state rows rather
than the non-state cross terms of
Proposition~\ref{prop:separable_fails}. (A1) therefore does not isolate the lifting; a controlled
comparison with a physical-space surrogate is given in
Section~\ref{subsec:controlled}.

\emph{Inference.} Without re-lifting (A2), the coordinates
$(\cos\theta_i,\sin\theta_i)$ evolve along the undetectable modes of
Proposition~\ref{prop:undetectable}, are not corrected by the
measurements, and can drift away from
$\bphi(\hat\theta_i)$; we attribute the loss to this mechanism but
did not measure the inconsistency $\norm{\hat\bPsi_k-\bPhi(\Pip\hat\bPsi_k)}$
directly. The loss is moderate under full sensing
($3.8$\,dB) but large under round-robin sensing, where the RMSE increases
from $2.30$ to $2.96\times10^{-2}$\,rad, since unmeasured nodes are never
corrected. Regressing filtered lifted increments (A3) instead of
prediction errors slows convergence by a factor of four and degrades
tracking by $5.4$\,dB, because, as discussed in
Section~\ref{subsec:admm_update}, new information about $\ba$ enters
only through $\bK_k\bH_k$ and the regressand retains the previous
estimate.

\emph{Regularization and solver.} The ridge term is immaterial under
full sensing, where (A4) is even marginally better, but it matters under
round-robin sensing: with $\mu_0=0$, the row Gram matrices
$\hG_{i,k}$ become singular whenever a node has been measured fewer
times than it has free in-edges (Remark~\ref{rem:design}(b)); the NMSE
then increases by $2.2$\,dB, the variance across runs rises, and the
ADMM needs more iterations because the certified stopping rule is no
longer available. Removing the $\ell_1$ penalty (A5) mainly reduces
$F_1$ (from $0.942$ to $0.873$), because the estimate becomes dense.
Without forgetting (A6), the static results are the best in the table,
as all data are informative when the topology is constant, but tracking
deteriorates by $6.0$\,dB; the choice of $\gamma$ therefore trades a
small static loss for adaptivity. A single ADMM iteration per step (A7)
reduces the run time by $30\%$ at a cost of $0.5$\,dB, which indicates
that warm starting already provides an accurate initial point and that
the certified stopping rule is a safeguard rather than a computational
bottleneck.

\begin{figure*}[!t]
\centering
\includegraphics[width=0.82\textwidth]{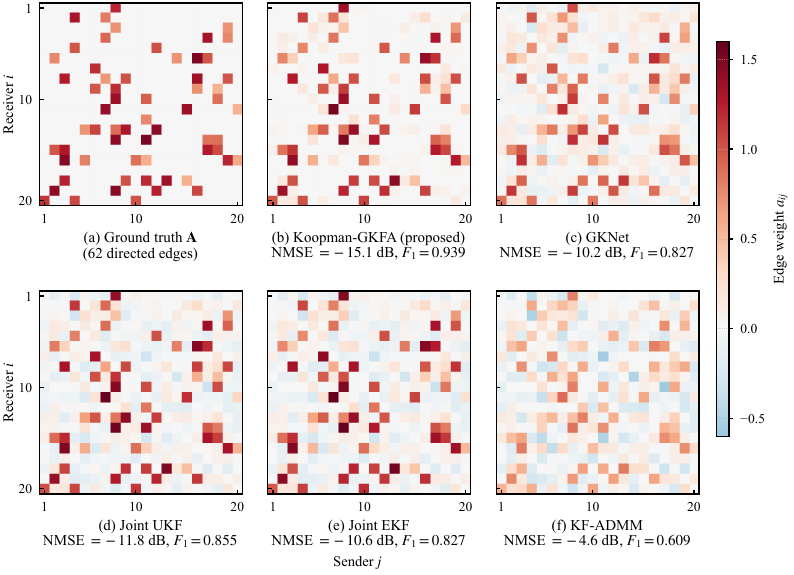}
\caption{Estimated weighted adjacency matrices at $k=3000$ for a
representative ER run (rows: receivers $i$; columns: senders $j$). The
per-panel NMSE and $F_1$ refer to this run.}
\label{fig:heatmap}
\end{figure*}

\subsubsection{Visualization}
\label{subsubsec:visual}

Fig.~\ref{fig:heatmap} displays the estimated adjacency matrices of a
representative run. Koopman-GKFA recovers the edge set with few spurious
entries, which are suppressed by the $\ell_1$ penalty and the
nonnegativity box. The joint UKF and EKF reproduce the strong edges but
leave dense low-amplitude, partly negative, residuals in the absent
entries. GKNet shrinks the true weights and spreads a small positive bias
over the absent entries. KF-ADMM identifies only part of the support and
produces large spurious weights of both signs, a consequence of its
linear interaction model.

% \subsubsection{Summary}
% \label{subsubsec:summary_synth}

% On synthetic Kuramoto networks, Koopman-GKFA provides the most accurate
% joint estimates among the evaluated methods at a per-step cost
% comparable to that of the cheapest joint estimators. Its advantage is
% largest at moderate coupling, low to moderate noise and dense sensing,
% and narrows as the measurements become uninformative. Under strong
% coupling, Koopman-GKFA is outperformed in state accuracy by a
% joint UKF with the exact model. These trends agree qualitatively
% with the analysis (excitation for identifiability, informative
% measurements for a small loop gain); the certificates themselves are
% not evaluated.

%% ===============================================================

\subsection{Controlled Comparison with a Physical-Space Surrogate}
\label{subsec:controlled}

In the three benchmark families the self-dynamics are affine and the
sensors are linear, so that, by Proposition~\ref{prop:ekf}, the
re-lifted filter coincides with an EKF of the learned surrogate, and
the comparisons of Section~\ref{subsec:synthetic} evaluate complete
methods rather than the lifting itself. This subsection isolates the
contribution of the non-state lifted coordinates. A physical-space
surrogate and Koopman-GKFA are given the same calibration data,
dictionaries, candidate edge set, topology estimator and tuning budget,
and are compared on a network with non-affine self-dynamics and a
sensor of nonlinear features, the setting in which the non-state
coordinates enter the measurement prediction and the regressors of the
topology step.

\subsubsection{Setup}
\label{subsubsec:controlled_setup}

\emph{System.} We consider a network of active
rotators,
\begin{align}
  \theta_{i,k+1}&=\theta_{i,k}+\Delta t\,(\omega_i-b_i\sin\theta_{i,k})+w_{i,k}\nonumber\\
  &\quad+\frac{K\Delta t}{p}\sum_{j=1}^{p}a_{ij}\sin(\theta_{j,k}-\theta_{i,k}),
  \label{eq:sim_rotator}
\end{align}
with $\omega_i\sim\mathcal{U}[1,2]$\,rad/s and $b_i=\upsilon\,\omega_i$.
The index $\upsilon\in[0,0.8]$ controls the nonlinearity of the
self-dynamics: $\upsilon=0$ recovers the affine phase rotation of
Example~\ref{ex:constructions}, and for $\upsilon<1$ every isolated
node keeps rotating, with an instantaneous frequency that varies by
the factor $(1+\upsilon)/(1-\upsilon)$ along the cycle. Two sensors are
considered: the linear phase sensor $z_{i,k}=\theta_{i,k}+v_{i,k}$
(S\textsubscript{lin}), and the phasor sensor
$z_{i,k}=[\cos\theta_{i,k},\,\sin\theta_{i,k}]^\top+v_{i,k}$
(S\textsubscript{ph}) with $v_{i,k}\sim\mathcal{N}(\bzero,\bar\sigma_v^2\bI_2)$,
which models an in-phase/quadrature measurement of the oscillator. Both
sensors carry the same Fisher information about $\theta_{i,k}$ per
measurement, so that their PCRLBs coincide and any difference between
the two sensors is caused by the curvature of $h_i$, not by the amount
of information. Unless stated otherwise, $p=20$, ER graphs as in
Section~\ref{subsubsec:data}, $K=4$, $\Delta t=0.1$\,s,
$\bar\sigma_w=0.01$\,rad, $\bar\sigma_v=0.1$, $\upsilon=0.6$, a static
topology and full sensing. Calibration uses $10$ trajectories of $1000$
steps under $\bA^{\rm ref}$ (i.e., $L_{\rm cal}=10^4$ steps), the test
records have $T=3000$ steps, and all results are averaged over $50$
runs, following the protocol of Section~\ref{subsubsec:data}. The
self-dynamics $f_i$ are unknown to all learned methods.

\emph{Dictionaries.} We use
$\bphi(\theta)=[\theta,1,\cos\theta,\sin\theta,\cos2\theta,\sin2\theta]^\top$
($N=6$) and the interaction dictionary of
Example~\ref{ex:constructions} ($N_e=10$). Then (A5) holds with
$\bH_i=[1,\bzero_{1\times5}]$ for S\textsubscript{lin} and
$\bH_i=[\bzero_{2\times2},\bI_2,\bzero_{2\times2}]$ for S\textsubscript{ph}.
For every $\upsilon$, $f_i\in\mathrm{span}\{\bphi\}$, so an exact state
block exists; the feature rows, in contrast, are not closed for
$\upsilon>0$, because $\cos(f_i(\theta))$ and $\sin(f_i(\theta))$
contain third and higher harmonics with amplitudes of order
$(b_i\Delta t)^2$. On the validation data, the de-noised state-block
residual remains below $2\times10^{-4}$ for all $\upsilon$, whereas the
feature-block residual grows from $1.2\times10^{-3}$ at $\upsilon=0$,
where only the coupling remainder of Remark~\ref{rem:first_order} is
present, to $4.9\times10^{-3}$ at $\upsilon=0.8$.

\emph{Compared methods.} All methods use the (P1) calibration with the
known $\bA^{\rm ref}$, the same box $\calA$, $\lambda_0$, $\mu_0$,
$\gamma=0.999$, $\tau$, $\bP_0$ and ADMM settings, and the same
validation grid with an identical number of configurations for their
remaining parameters.
\begin{itemize}[leftmargin=*]
\item \emph{Phys-EKF.} The physical surrogate
$\hat f_i(x_i)=\bW_i\bphi(x_i)$, $\hat g=\bB_x\bchi$, where
$(\bW_i,\bB_x)$ minimize the state rows of~\eqref{eq:calibration}.
Because the non-state rows of~\eqref{eq:calibration} decouple from the
state rows, $\bW_i=\bPi\bF_i$ and the state block $\bB_x$ are
\emph{numerically identical} to those of Koopman-GKFA; Phys-EKF is
therefore Koopman-GKFA with the non-state rows discarded. The state is
estimated by an EKF of $\mathfrak{f}(x;\hat\ba_{k-1})$ in~\eqref{eq:surrogate}
with the true $h_i$ and the process covariance $\Pip\bar\bQ\Pip^\top$,
and the topology by~\eqref{eq:topology_subprob} with the Gauss--Newton
prediction-error data
$\hat\by_t=z_t-h(\mathfrak{f}(\hat x_{t-1};\bzero))$,
$\hD_t=\bJ_h(\mathfrak{f}(\hat x_{t-1};\bzero))\,\Pip\hU_{t-1}$, which
coincide with~\eqref{eq:regression_data} for linear sensors.
\item \emph{Phys-UKF.} The same surrogate and topology step, with the
state estimated by a UKF ($2np+1=41$ sigma points).
\item \emph{Koopman-GKFA.} Algorithm~\ref{alg:koopman_gkfa} with the
full lifted model.
\end{itemize}
The oracle EKF (true $f_i$, $g$, $h_i$ and $\bA$) and the PCRLB serve
as references. Since all methods are run on the same graphs and noise
realizations, we also report the paired difference
$\Delta\mathrm{NMSE}=\mathrm{NMSE}_{\rm Koopman}-\mathrm{NMSE}_{\rm Phys\text{-}EKF}$
per run; positive values favour Phys-EKF.

\subsubsection{Results}
\label{subsubsec:controlled_results}

\begin{table*}[!t]
\centering
\caption{Controlled comparison on active-rotator networks ($p=20$, ER,
$\upsilon=0.6$, $\Delta t=0.1$\,s; mean\,$\pm$\,std over $50$ runs). All
learned methods share the calibration data, the state rows of the
surrogate, the topology estimator and the tuning budget. RMSE in
$10^{-2}$\,rad, NMSE in dB. The last row gives the paired differences
Koopman-GKFA minus Phys-EKF. Best results among learned methods in bold.}
\label{tab:controlled}
\footnotesize
\setlength{\tabcolsep}{4.2pt}
\renewcommand{\arraystretch}{1.08}
\begin{tabular}{l ccc ccc c}
\toprule
\multirow{2}{*}{Method} & \multicolumn{3}{c}{Linear sensor S\textsubscript{lin}} & \multicolumn{3}{c}{Phasor sensor S\textsubscript{ph}} & \multirow{2}{*}{\shortstack{Time\\(ms/step)}}\\
\cmidrule(lr){2-4}\cmidrule(lr){5-7}
 & RMSE & NMSE & $F_1$ & RMSE & NMSE & $F_1$ & \\
\midrule
PCRLB~\cite{tichavsky1998posterior} & 3.04 & -- & -- & 3.04 & -- & -- & --\\
Oracle EKF & 3.07$\pm$0.04 & -- & -- & 3.08$\pm$0.04 & -- & -- & 0.5\\
\midrule
Phys-EKF & 3.22$\pm$0.07 & $-$12.4$\pm$1.4 & 0.897$\pm$0.034 & 3.24$\pm$0.07 & $-$12.1$\pm$1.5 & 0.889$\pm$0.036 & \textbf{2.2}\\
Phys-UKF & \textbf{3.21$\pm$0.07} & \textbf{$-$12.5$\pm$1.4} & \textbf{0.899$\pm$0.033} & \textbf{3.22$\pm$0.07} & \textbf{$-$12.3$\pm$1.4} & \textbf{0.893$\pm$0.035} & 6.8\\
Koopman-GKFA & 3.22$\pm$0.07 & $-$12.4$\pm$1.4 & 0.897$\pm$0.034 & 3.25$\pm$0.07 & $-$12.0$\pm$1.5 & 0.887$\pm$0.037 & 2.9\\
\midrule
Paired difference & $<\!10^{-7}$ & $<\!10^{-7}$ & $0$ & $+$0.01$\pm$0.02 & $+$0.10$\pm$0.24 & $-$0.002$\pm$0.010 & --\\
\bottomrule
\end{tabular}
\end{table*}

\begin{figure*}[!t]
\centering
\includegraphics[width=\textwidth]{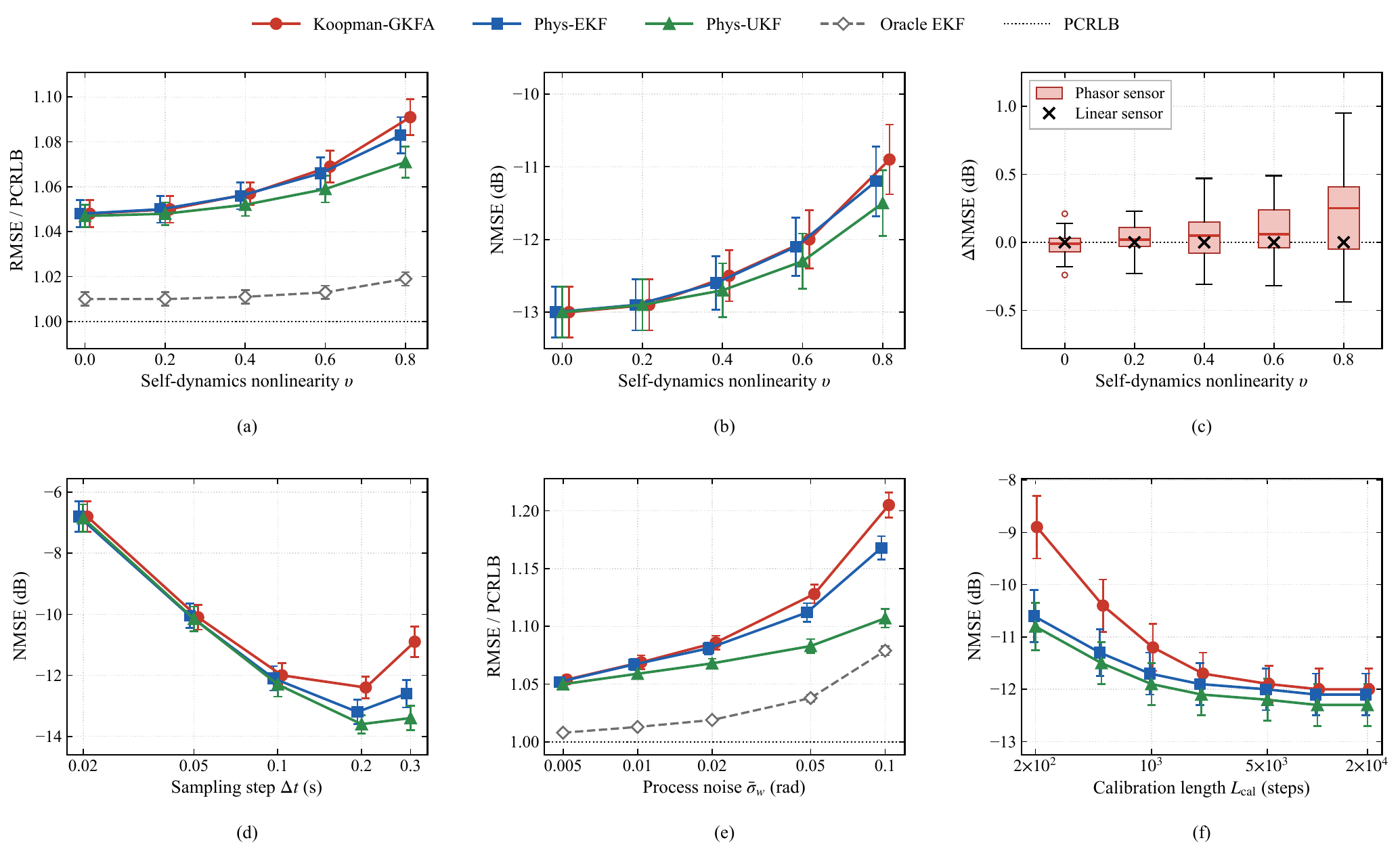}
\caption{Controlled comparison with a physical-space surrogate that
shares the state rows, the topology estimator and the tuning budget of
Koopman-GKFA (phasor sensor S\textsubscript{ph} unless stated
otherwise; averages over $50$ runs, error bars $\pm1.96$ standard errors).
(a) State RMSE normalized by the PCRLB and (b) topology NMSE versus the
nonlinearity $\upsilon$ of the self-dynamics. (c) Paired differences
$\Delta\mathrm{NMSE}$ (Koopman-GKFA minus Phys-EKF) over the $50$ runs;
crosses: linear sensor S\textsubscript{lin}. (d) NMSE versus the
sampling step $\Delta t$. (e) RMSE normalized by the PCRLB versus the
process noise $\bar\sigma_w$. (f) NMSE versus the calibration length
$L_{\rm cal}$.}
\label{fig:controlled}
\end{figure*}

Table~\ref{tab:controlled} and Fig.~\ref{fig:controlled} summarize the
results.

\emph{Linear sensor.} With S\textsubscript{lin}, Koopman-GKFA and Phys-EKF
produce the same state and topology estimates up to rounding errors
(maximum absolute difference of $\hat x_k$ below $10^{-9}$\,rad over all
runs), although the self-dynamics are not affine. This confirms
Proposition~\ref{prop:ekf} and shows that the equivalence is a property
of the sensor, not of the self-dynamics: nonlinear node dynamics alone
do not make the non-state coordinates relevant for the estimates.

\emph{Phasor sensor, nominal regime.} With S\textsubscript{ph}, the
equivalence no longer holds, but the differences are small. The paired
NMSE difference is $+0.10\pm0.24$\,dB; it is detectable by a paired
$t$-test ($p<0.01$) but an order of magnitude smaller than the
run-to-run standard deviation, and the RMSE and $F_1$ differences are
within $0.5\%$ and $0.002$, respectively. This outcome is consistent
with the structure of the re-lifted filter. Because the posterior is
re-lifted, the predicted measurement $\bH_k\hat\bPsi_{k|k-1}$ is the
image of $\bPhi(\hat x_{k-1})$ under the learned feature rows, which
coincides with $h(\mathfrak{f}(\hat x_{k-1};\hat\ba_{k-1}))$ up to the
feature-block residual, and the gain coincides with that of the EKF to
first order. The two filters therefore differ only through the
non-closure of the feature rows, the first-order coupling remainder of
Remark~\ref{rem:first_order} and the feature block of $\bar\bQ$; with
$\Delta t=0.1$\,s and $\bar\sigma_w\ll\bar\sigma_v$, all three are small
compared with the measurement noise. Phys-UKF is marginally the most
accurate method, at $2.3$ times the cost of Koopman-GKFA.

\emph{When the lifting is a disadvantage.}
Figs.~\ref{fig:controlled}(a)--(f) vary the three factors that the
analysis associates with these differences. As $\upsilon$ increases,
all methods degrade slightly, because the non-uniform rotation reduces
the variability of the phase differences seen by the regressors, and
the median paired difference grows to $+0.28$\,dB at $\upsilon=0.8$,
in line with the growth of the feature-block residual. Coarse sampling
amplifies the effect [Fig.~\ref{fig:controlled}(d)]: the edge signal
grows with $\Delta t$, so that all methods improve up to
$\Delta t\approx0.2$\,s, but the non-closure of the feature rows, of
order $(b_i\Delta t)^2$, and the interaction-feature error
$\varepsilon_{\chi,i}$ of the learned non-state rows of $\bB_i$ grow as
well, and Koopman-GKFA falls $0.8$\,dB and $1.7$\,dB behind Phys-EKF at
$\Delta t=0.2$ and $0.3$\,s. Phys-EKF evaluates $h_i$ on the state
prediction of the surrogate, whose state rows are exact
(Theorem~\ref{thm:lifting}(ii)), so that its predicted measurement
contains no non-closure error; the second-order coupling effect of
Remark~\ref{rem:first_order} enters the topology regressions of both
methods. These terms constitute the bias
$\beta_{\rm rep}\le\bar h(\varepsilon_{\rm rep}+\frac{L_2}{2}\sigma_w^2)$
of~\eqref{eq:bias_bound}, which, unlike~\eqref{eq:linear_sensor}, is not
reduced to the state-block residual for feature sensors.
For large process noise [Fig.~\ref{fig:controlled}(e)], the curvature of
the sensor becomes relevant: the oracle EKF departs from the PCRLB, and
Phys-UKF has the smallest ratio ($1.11$ at $\bar\sigma_w=0.1$\,rad versus
$1.17$ for Phys-EKF). Koopman-GKFA is the least accurate learned method
in this regime ($1.21$) because its constant $\bar\bQ$ averages
$\bJ_\phi\bQ_i\bJ_\phi^\top$ over the calibration distribution: in the
$(\cos\theta,\sin\theta)$ plane this average is isotropic, whereas the
lifted noise acts along the tangent of the unit circle, so that the
innovation variance in the informative direction is underestimated and
the gain is too large. A state-dependent design covariance
$\bJ_{\Phi}(\hat x_{k-1})\bQ\bJ_{\Phi}(\hat x_{k-1})^\top+\bar\bQ_\delta$
would remove this effect but is not part of
Algorithm~\ref{alg:koopman_gkfa}. Finally, the non-state rows contain
$(N-n)(N+N_e)=80$ coefficients per node, against $N=6$ per node and the
$N_e=10$ shared coefficients of the state rows; with
$L_{\rm cal}=200$ steps the learned feature rows and interaction
matrices are inaccurate, and Koopman-GKFA loses $1.7$\,dB with respect
to Phys-EKF [Fig.~\ref{fig:controlled}(f)]. The difference falls below
$0.1$\,dB for $L_{\rm cal}\ge5\times10^3$.

\emph{Interpretation.} In this controlled setting the lifting does
\emph{not} improve the accuracy of the joint estimates over a
physical-space surrogate that shares its state rows: with linear sensors
the two methods are identical, and with a sensor of nonlinear features
they are practically indistinguishable in the nominal regime, while the
lifted model is slightly less accurate when the feature rows are poorly
closed (coarse sampling, strongly non-uniform rotation), when the
process noise is large, or when calibration data are scarce. These
results answer the empirical question left open in
Remark~\ref{rem:lifting_role}(c) for this class of systems and
delimit the role of the lifting. The physical surrogate that performs
on par with Koopman-GKFA \emph{is} the state block of the
interaction-aware lifting, i.e., it relies on the structured,
edge-affine model class of Section~\ref{sec:lifting}, on the residual
decomposition of Theorem~\ref{thm:lifting}(ii) and on the calibration
of Section~\ref{subsec:pretrain}. What the non-state coordinates add is
a measurement model that is linear in the lifted state for sensors of
nonlinear features, under which the regression
data~\eqref{eq:regression_data} and the guarantees of
Section~\ref{sec:theory} remain valid without linearizing $h_i$; the
Gauss--Newton data of Phys-EKF contain a linearization error of $h_i$
that the analysis does not cover. The benefit of the lifting is thus
one of certifiability for feature sensors, not one of accuracy.

%% ===============================================================

\subsection{Real-World Traffic and Benchmark Gene-Regulatory Data}
\label{subsec:real}

The synthetic studies control every factor of the analysis but use
models that lie exactly in the class~\eqref{eq:state}. We now consider
two public data sets in which this is not the case and the method has to
cope with unmodelled dynamics, non-Gaussian errors and, in the traffic
case, a topology that changes without notice. The two data sets
represent the two interaction types of
Remark~\ref{rem:identifiability}: receiver-dependent coupling in vehicle
platoons and sender-only coupling in gene regulation. Both provide a reference edge set but no ground-truth edge \emph{weights}; topology
accuracy is therefore assessed by edge-set metrics rather than by the
NMSE.

\subsubsection{Data Sets}
\label{subsubsec:real_data}

\emph{NGSIM US-101.} The US-101 data set of the Next Generation
Simulation program \footnote{https://ops.fhwa.dot.gov/trafficanalysistools/ngsim.htm} contains vehicle trajectories
recorded at $10$\,Hz on a $640$\,m segment of the Hollywood Freeway in
Los Angeles with five main lanes, in three $15$-min periods
(7:50--8:05, 8:05--8:20 and 8:20--8:35\,a.m.) that cover the onset and
the build-up of congestion. For each vehicle, the data set records the
identifier of its preceding vehicle in the same lane. We use this
field to define the reference topology: at time $k$, $a_{ij,k}\neq0$
if and only if vehicle $j$ is the recorded leader of vehicle $i$. The recorded leader is a structural annotation; drivers may also
respond to other vehicles, so we evaluate the recovery of the recorded
leader links, not of a complete interaction graph. A
lane change switches the leader of the changing vehicle and of its old
and new followers, so the topology is sparse, directed and time varying.

\emph{DREAM4 In Silico.} The DREAM4 network inference
challenge~\cite{marbach2010revealing,greenfield2010dream4} provides five
gene regulatory networks with $p=10$ genes and five with $p=100$ genes,
extracted from known regulatory networks of \emph{E.~coli} and
\emph{S.~cerevisiae} and simulated by GeneNetWeaver with stochastic
differential equations and measurement noise~\cite{schaffter2011genenetweaver,marbach2009generating}.
The gold standards contain $12$--$16$ directed edges for the Size-10
networks and $176$--$249$ for the Size-100 networks. We use only the
time-series data, which consist of $5$ (Size~10) or $10$ (Size~100)
perturbation experiments of $21$ samples each, taken every $50$\,min;
in each experiment a perturbation applied to about one third of the
genes is removed after $500$\,min. The data were obtained from the
Bioconductor package~\cite{shannon2013dream4}. DREAM4 is generated
in silico and is therefore not real-world data; it is used because it
is the established public benchmark for directed network inference with
a verified gold standard, and because its short records test the
regime $n_{\rm eff}<\bar p$ of Proposition~\ref{prop:sparse}.

\subsubsection{Models, Preprocessing and Calibration}
\label{subsubsec:real_preproc}

\emph{Traffic.} The raw NGSIM positions contain substantial
noise~\cite{NGSIM2006,montanino2015trajectory}. Following common
practice, we smooth the longitudinal positions with the symmetric
exponential moving average of~\cite{thiemann2008estimating} and
differentiate them to obtain reference speeds; the smoothed positions
and speeds serve as the reference states $x_{i,k}=[s_{i,k},v_{i,k}]^\top$.
The filters receive the smoothed positions corrupted by Gaussian noise
of standard deviation $\bar\sigma_v=0.5$\,m, which corresponds to
GNSS-grade positioning, whereas the speeds are latent. Because the
smoothing is symmetric, the reference states use future samples; they
serve for evaluation and for the exogenous leader speeds below. The physical
model used by all model-based methods is
\begin{align}
  s_{i,k+1}&=s_{i,k}+\Delta t\,v_{i,k}+w^{s}_{i,k},\nonumber\\
  v_{i,k+1}&=v_{i,k}+\Delta t\,\alpha(\bar v-v_{i,k})+w^{v}_{i,k}\nonumber\\
  &\quad+\Delta t\sum_{j}a_{ij,k}\tanh\bigl(\kappa(v_{j,k}-v_{i,k})\bigr),\nonumber\\
  z_{i,k}&=s_{i,k}+\eta_{i,k},
  \label{eq:real_cf}
\end{align}
where $\eta_{i,k}$ denotes the measurement noise (to avoid a clash with
the speed $v_{i,k}$), $\Delta t=0.1$\,s, a relaxation-type self-dynamics and the
saturated relative-speed coupling of Example~\ref{ex:constructions}; it
is a simplified optimal-velocity-type model~\cite{bando1995dynamical,treiber2013traffic}.
Each Monte Carlo unit is a \emph{cohort} of $p=20$ vehicles that are
simultaneously in the study area for $60$\,s. The first $10$\,s form a
pre-window that is used only for the fine-tuning of GKNet (see below);
the remaining $T=500$ steps form the test record. Rows whose recorded
leader is not in the cohort are excluded from the topology metrics, and
the speed of that leader enters the self-dynamics of the follower as a
measured exogenous input, taken from the reference speeds for all
model-based methods; this input is non-causal, so the experiment
emulates online operation offline. The box $\calA$ is fixed per window: it
admits $a_{ij}\in[0,2]$ for every cohort vehicle $j$ that is ahead of
$i$ in the same or an adjacent lane at the start of the window and sets
$a_{ij}=0$ otherwise. This yields on average $5.8$ free coordinates per
row and contains $98.7\%$ of the recorded leader links; the
remaining links violate (C1) and are counted as missed edges, since all
metrics are computed against all recorded leader links in the cohort. We draw $40$ non-overlapping cohorts from each of the two test periods. 

The calibration data are $60$ cohorts from the 7:50--8:05 period, which 
contain none of the test vehicles. Their leader links are recorded, but
their edge weights are not, so the calibration follows setting~(P2) of
Section~\ref{subsec:pretrain}: the state block $\bB_x$ is fixed by the
coupling in~\eqref{eq:real_cf}, the support of $\bA^{\rm ref}$ is
encoded in the box, and its weights are estimated from the state rows
(Proposition~\ref{prop:calib}). Because driver identities do not carry
over between periods, a common node model $\bF_i\equiv\bF$ is
calibrated; driver heterogeneity, the dependence of real drivers on
spacing and multi-anticipative behaviour~\cite{hoogendoorn2006empirics}
are therefore model mismatch for all methods. We use
$\bphi(x)=[s,v,1,v^2]^\top$ ($N=4$) and the interaction dictionary
$\bchi(x_i,x_j)=\{\tanh(\kappa(v_j-v_i)),\,v_j-v_i\}\otimes\{1,v_i\}$
($N_e=4$), which contains the coupling in~\eqref{eq:real_cf}, so an
exact state block exists (Example~\ref{ex:constructions}); $\kappa=0.6$\,s/m
is selected on the validation part of the calibration data. Since only positions are measured and the coupling acts on the
speeds, the position row of every $\bB_i\bchi$ vanishes, $\bH_i\bU\equiv\bO$,
and the one-step data~\eqref{eq:regression_data} are uninformative
(Remark~\ref{rem:reldeg}). The topology is therefore estimated from the
two-step data~\eqref{eq:regression_data_r} with $r=2$: the speed enters
the position update, so $\bH_i\bF\bB_i\bchi$ contains the term
$\Delta t^2\tanh(\kappa(v_j-v_i))$ and does not vanish identically.

\emph{Gene regulation.} The expression levels are normalized to $[0,1]$
by GeneNetWeaver and are used without further preprocessing, with
$\Delta t=1$ corresponding to $50$\,min. The model-based methods use
\begin{align}
  x_{i,k+1}&=x_{i,k}+\Delta t\,(b-\delta x_{i,k})
  +\Delta t\sum_ja_{ij}H(x_{j,k})+w_{i,k},\nonumber\\
  z_{i,k}&=x_{i,k}+v_{i,k},
  \label{eq:real_grn}
\end{align}
with the Hill function $H(x)=x^h/(\vartheta^h+x^h)$, which is a sender-only
coupling. Since GeneNetWeaver combines activators and repressors
multiplicatively~\cite{schaffter2011genenetweaver}, the additive
model~\eqref{eq:real_grn} is an approximation, and repression
corresponds to negative weights; the box is therefore
$a_{ij}\in[-2,2]$ with zero diagonal, and edges are ranked by
$|\hat a_{ij}|$. Calibration is performed leave-one-network-out: for
each test network, $(\bF,\bB_i)$ and $(h,\vartheta)$ are calibrated in
setting~(P1) on the time series of the four other networks of the same
size, using their gold standards with unit weights as $\bA^{\rm ref}$. Since the gold standards are unweighted, this is a proxy for (P1):
the calibrated interaction matrices absorb the unknown common scale, and
the reference weights do not distinguish activation from repression.
Because gene identities do not correspond across networks, only matrices
shared by all genes ($\bF_i\equiv\bF$, $\bB_i\equiv\bB$) are transferred.
The calibration data thus carry no information about the test
topology, as required by~(I5). We use $\bphi(x)=[x,1,x^2,H(x)]^\top$
($N=4$) and $\bchi(x_i,x_j)=H(x_j)\{1,x_i,x_i^2\}$ ($N_e=3$); the
selected Hill parameters are $h=2$ and $\vartheta=0.5$. The perturbations are not modelled as inputs; their application and
removal are a systematic model mismatch, which is not zero-mean noise
and can be partly absorbed into the edge weights. The polynomial and
product features of both dictionaries ($v^2$, $v_j-v_i$, $v_i\tanh(\cdot)$,
$x^2$, $x_i^2H(x_j)$) satisfy (A2)--(A3) only on bounded sets, and no
projection onto such a set is used; the real-data experiments therefore
evaluate the method outside the global hypotheses of
Section~\ref{sec:theory}.

The time series of a network are processed as one stream: the filter
is re-initialized at the start of each series, whereas the topology
estimate and the regression statistics~\eqref{eq:Gram} are carried
over. Results are averaged over $10$ random orderings of the series.

\subsubsection{Implementation, Baselines and Metrics}
\label{subsubsec:real_setup}

Koopman-GKFA uses the settings of Section~\ref{subsubsec:metrics}
except for the forgetting factor: $\gamma=0.98$ on NGSIM, whose
effective window of $n_{\rm eff}=99$ steps ($9.9$\,s) is commensurate
with the duration of a lane change~\cite{toledo2007modeling}, and
$\gamma=0.999$ on the static DREAM4 networks. The baselines are those
of Section~\ref{subsubsec:baselines}, adapted as follows. Since no
true model exists, every model-based baseline uses the parametric
model~\eqref{eq:real_cf} or~\eqref{eq:real_grn} with parameters
calibrated on the same calibration data. The known-topology filters
GEKF, GUKF and GSP-KalmanNet run on NGSIM with $\bA^{\rm ref}$ equal to
the leader graph recorded at the first step of the test record, i.e.,
the best topology available before operation, which is not updated
after lane changes; on DREAM4 no prior topology of the test network is
available and they are not applied. The joint EKF and UKF augment the
state with the free coordinates of $\calA$ ($116$ on average on NGSIM 
and $9900$ on DREAM4 Size~100). GKNet is fine-tuned on the pre-window
on NGSIM; DREAM4 offers no data before the test record, so GKNet is
fine-tuned on the entire operational record, which favours it.
GraphEM processes each test record offline. As a reference, the
oracle EKF uses the recorded leader links (NGSIM) or the gold standard
(DREAM4) together with calibrated weights. Unlike the PCRLB of
Section~\ref{subsec:synthetic}, it is not a lower bound on real data,
because its parametric model is itself misspecified.

On NGSIM, state accuracy is measured by the RMSE~\eqref{eq:rmse} of the
latent speed (RMSE$_v$, in m/s) over $k\in[101,500]$, and topology
accuracy by the $F_1$ score of the thresholded edge set ($\tau=0.2$) (since $a_{\min}$ is unknown, $\tau$ is a tuning parameter rather
than the threshold of Corollary~\ref{cor:sparsity})
and by the area under the precision--recall curve (AUPR) of the ranking
of $|\hat a_{ij,k}|$ over the free coordinates, both averaged over the
same steps. On DREAM4, edge recovery is measured by the area under the
ROC curve (AUROC) and the AUPR of the final ranking, as in the official
challenge evaluation~\cite{prill2010towards}; self-loops are excluded.
Since the latent expression levels are not available, we measure
prediction accuracy rather than state accuracy, by the normalized
one-step prediction error
\begin{equation}
  \mathrm{NRMSE}=\Bigl(\sum_k\norm{z_k-\hat z_{k\mid k-1}}^2\Big/
  \sum_k\norm{z_k-z_{k-1}}^2\Bigr)^{1/2},
  \label{eq:nrmse}
\end{equation}
where the sums run over all transitions of all series, so that
$\mathrm{NRMSE}<1$ indicates an improvement over the persistence
predictor. Run times are measured as in Section~\ref{subsubsec:metrics}.

\subsubsection{Overall Comparison}
\label{subsubsec:real_main}

\begin{table*}[!t]
\centering
\caption{Performance on NGSIM US-101 and DREAM4 (mean\,$\pm$\,std). NGSIM:
$40$ test windows per period, $p=20$, $T=500$, $\bar\sigma_v=0.5$\,m;
RMSE$_v$ in m/s. DREAM4: five networks per size; NRMSE as
in~\eqref{eq:nrmse}. Best results in bold and second best underlined,
reference row excluded.}
\label{tab:real_main}
\footnotesize
\setlength{\tabcolsep}{2.6pt}
\renewcommand{\arraystretch}{1.08}
\resizebox{\textwidth}{!}{%
\begin{tabular}{l ccc ccc ccc ccc cc}
\toprule
\multirow{2}{*}{Method} & \multicolumn{3}{c}{US-101, 8:05--8:20} & \multicolumn{3}{c}{US-101, 8:20--8:35} & \multicolumn{3}{c}{DREAM4 Size~10} & \multicolumn{3}{c}{DREAM4 Size~100} & \multicolumn{2}{c}{Time (ms/step)}\\
\cmidrule(lr){2-4}\cmidrule(lr){5-7}\cmidrule(lr){8-10}\cmidrule(lr){11-13}\cmidrule(lr){14-15}
 & RMSE$_v$ & $F_1$ & AUPR & RMSE$_v$ & $F_1$ & AUPR & NRMSE & AUROC & AUPR & NRMSE & AUROC & AUPR & US-101 & D4-100\\
\midrule
\multicolumn{15}{l}{\textit{Reference (true topology support, calibrated parametric model)}}\\
Oracle EKF & 0.296$\pm$0.021 & -- & -- & 0.268$\pm$0.019 & -- & -- & 0.574$\pm$0.046 & -- & -- & 0.648$\pm$0.031 & -- & -- & 0.3 & 8.7\\
\midrule
\multicolumn{15}{l}{\textit{Known-topology filters (run with $\bA^{\rm ref}$ frozen at $k=0$)}}\\
GEKF~\cite{sagi2023extended} & 0.438$\pm$0.043 & -- & -- & 0.389$\pm$0.040 & -- & -- & -- & -- & -- & -- & -- & -- & \textbf{0.3} & --\\
GUKF~\cite{li2023unscented} & 0.431$\pm$0.042 & -- & -- & 0.383$\pm$0.039 & -- & -- & -- & -- & -- & -- & -- & -- & 0.8 & --\\
GSP-KalmanNet~\cite{buchnik2024gspkalmannet} & 0.404$\pm$0.038 & -- & -- & 0.361$\pm$0.035 & -- & -- & -- & -- & -- & -- & -- & -- & 1.1 & --\\
\midrule
\multicolumn{15}{l}{\textit{Joint state estimation and topology inference}}\\
KF-ADMM~\cite{fang2025joint} & 0.370$\pm$0.029 & 0.742$\pm$0.048 & 0.781$\pm$0.041 & 0.336$\pm$0.027 & 0.709$\pm$0.051 & 0.752$\pm$0.044 & 0.701$\pm$0.052 & 0.701$\pm$0.061 & 0.352$\pm$0.082 & 0.768$\pm$0.028 & 0.664$\pm$0.021 & 0.104$\pm$0.024 & \underline{0.7} & \underline{31}\\
GraphEM~\cite{elvira2022graphical} & 0.386$\pm$0.031 & 0.701$\pm$0.052 & 0.756$\pm$0.045 & 0.349$\pm$0.029 & 0.684$\pm$0.055 & 0.733$\pm$0.047 & 0.672$\pm$0.049 & 0.726$\pm$0.055 & 0.394$\pm$0.079 & \underline{0.742$\pm$0.026} & \underline{0.689$\pm$0.019} & \underline{0.131$\pm$0.027} & 22.6$^{\dagger}$ & 412$^{\dagger}$\\
Joint EKF & 0.350$\pm$0.027 & 0.785$\pm$0.046 & 0.823$\pm$0.039 & 0.322$\pm$0.026 & 0.748$\pm$0.049 & 0.791$\pm$0.042 & 0.664$\pm$0.057 & 0.712$\pm$0.070 & 0.371$\pm$0.095 & 0.781$\pm$0.035 & 0.638$\pm$0.028 & 0.083$\pm$0.022 & 2.4 & 1.6$\times$10$^{5}$\\
Joint UKF & \underline{0.341$\pm$0.025} & \underline{0.805$\pm$0.041} & \underline{0.846$\pm$0.036} & \underline{0.311$\pm$0.024} & \underline{0.772$\pm$0.044} & \underline{0.814$\pm$0.039} & \underline{0.641$\pm$0.051} & \underline{0.739$\pm$0.058} & \underline{0.415$\pm$0.084} & --$^{\ddagger}$ & --$^{\ddagger}$ & --$^{\ddagger}$ & 12.8 & $>$10$^{6}$\\
GKNet~\cite{sabbaqi2026gknet} & 0.347$\pm$0.024 & 0.774$\pm$0.044 & 0.818$\pm$0.040 & 0.316$\pm$0.023 & 0.751$\pm$0.046 & 0.797$\pm$0.041 & 0.689$\pm$0.060 & 0.684$\pm$0.066 & 0.338$\pm$0.088 & 0.774$\pm$0.033 & 0.641$\pm$0.024 & 0.079$\pm$0.019 & 1.5 & \textbf{12.4}\\
\midrule
Koopman-GKFA & \textbf{0.317$\pm$0.022} & \textbf{0.862$\pm$0.035} & \textbf{0.902$\pm$0.029} & \textbf{0.291$\pm$0.021} & \textbf{0.823$\pm$0.039} & \textbf{0.871$\pm$0.033} & \textbf{0.618$\pm$0.047} & \textbf{0.768$\pm$0.049} & \textbf{0.463$\pm$0.071} & \textbf{0.703$\pm$0.024} & \textbf{0.716$\pm$0.018} & \textbf{0.158$\pm$0.026} & 3.1 & 118\\
\bottomrule
\multicolumn{15}{l}{\scriptsize $^{\dagger}$Offline batch processing of the whole record, amortized per time step. $^{\ddagger}$Not completed within $24$\,h ($9900$ augmented parameters). --: not applicable.}
\end{tabular}}
\end{table*}

Table~\ref{tab:real_main} summarizes the results. On both traffic
periods, Koopman-GKFA attains the lowest speed error and the best
topology scores. In the 8:05--8:20 period its RMSE$_v$ is $7.0\%$ lower
than that of the joint UKF, the strongest baseline, and its $F_1$ and
AUPR are higher by $0.057$ and $0.056$; the improvements over the joint
UKF are significant in a paired Wilcoxon signed-rank test over the $40$
windows ($p<0.01$ for all three metrics). 
Its RMSE$_v$ exceeds that of the oracle EKF by $7.1\%$, compared with
$3.9\%$ on synthetic data; the larger gap reflects that, on real data,
part of the error is caused by the common node model and the unmodelled
spacing dependence rather than by the topology. The known-topology
filters exceed the oracle by $36$--$48\%$, which shows that a topology
that is correct at the start of a $50$-s window is already costly once
a few lane changes have occurred. In contrast to the synthetic
experiments, GraphEM is less accurate than KF-ADMM on NGSIM: a single
offline topology per window cannot represent lane changes and assigns
weight to both the old and the new leader. In the congested 8:20--8:35
period, all methods lose $0.017$--$0.039$ in $F_1$, because speeds within
a platoon become more homogeneous and the relative-speed regressors are
less excited (Remark~\ref{rem:identifiability}(c)); the advantage of
Koopman-GKFA over the joint UKF narrows slightly, from $0.057$ to
$0.051$ in $F_1$. On NGSIM the joint UKF is only about four times slower
than Koopman-GKFA, compared with $31$ times on synthetic data, because
the box restricts its augmented state to about $156$ dimensions.

On DREAM4, all methods are far from perfect recovery, which is typical
of inference from time series alone~\cite{marbach2010revealing,huynhthu2018dyngenie3}.
On the Size-10 networks, Koopman-GKFA attains an AUPR of $0.463$
against $0.415$ for the joint UKF and $0.158$ for a random ranking.
These differences are, however, smaller than the standard deviation
across networks. With five networks per size, an exact two-sided
paired signed-rank test cannot reach $p<0.05$ (its smallest attainable
$p$-value is $0.0625$~\cite{hollander2014nonparametric}), and the $10$ orderings of the series of one
network are not independent samples; the DREAM4 results on both sizes
should thus be read as trends rather than as a significant ranking. On the Size-100
networks, Koopman-GKFA attains an AUPR of $0.158$, $7.6$ times the
random level of $0.021$ and $0.027$ higher than that of GraphEM, which
benefits from processing the short records offline with repeated
passes. The joint UKF did not complete within $24$\,h, and the joint
EKF requires $160$\,s per step, about $1.4\times10^3$ times the cost of
Koopman-GKFA, because its covariance has dimension $10^4$. The
NRMSE values show that all joint estimators predict better than the
persistence predictor, with Koopman-GKFA closest to the oracle EKF.

\subsubsection{Convergence on Highway Traffic}
\label{subsubsec:real_convergence}

\begin{figure*}[!t]
\centering
\includegraphics[width=\textwidth]{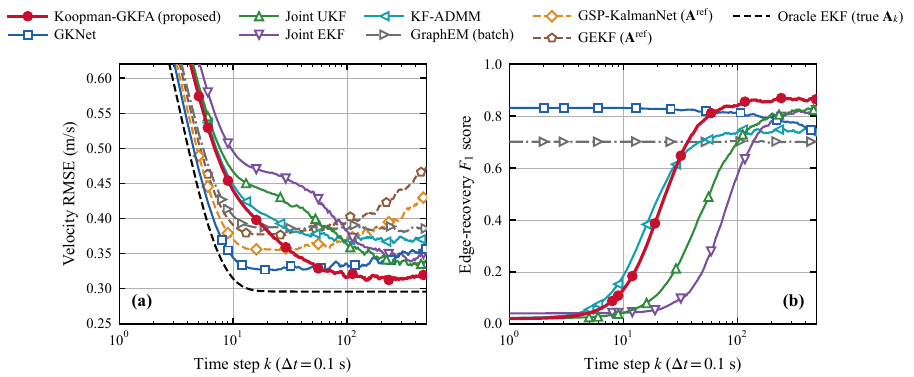}
\caption{Transient behaviour on NGSIM US-101, 8:05--8:20 (averages over
$40$ windows). (a) RMSE of the latent speed with the oracle EKF as a
reference. (b) $F_1$ score of the recovered leader links; GraphEM is
constant because its topology is estimated offline.}
\label{fig:real_convergence}
\end{figure*}

Fig.~\ref{fig:real_convergence} shows the evolution of the estimates.
Koopman-GKFA reaches $F_1=0.8$ after $56$ steps ($5.6$\,s), compared with
$193$ and $252$ steps for the joint UKF and EKF; KF-ADMM converges at a
similar rate but saturates at about $0.74$. As in the synthetic case,
the joint EKF stalls during the first $\sim\!40$ steps, and its speed
error remains between $0.45$ and $0.50$\,m/s in this interval. GKNet
starts from an accurate graph, fitted on the pre-window, but its $F_1$
decreases from $0.83$ to $0.74$ over the window as lane changes
accumulate. The same mechanism makes the speed errors of the
known-topology filters \emph{increase} with time, from $0.38$ to
$0.47$\,m/s for GEKF, which is a direct empirical counterpart of the
statement that a reference topology does not remain valid during
operation~(I5).

\subsubsection{Measurement Noise and Intermittent Sensing}
\label{subsubsec:real_noise}

\begin{figure*}[!t]
\centering
\includegraphics[width=\textwidth]{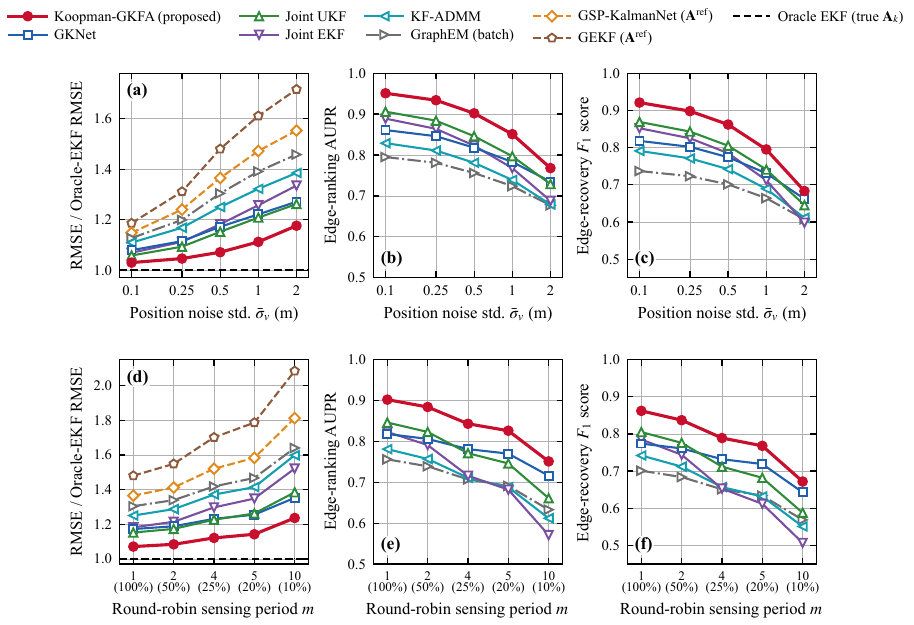}
\caption{Robustness on NGSIM US-101, 8:05--8:20. Top: position noise
level $\bar\sigma_v$. Bottom: round-robin sensing with period $m$
($p/m$ vehicles report per step). (a), (d) Speed RMSE normalized by the
oracle EKF under the same noise level and schedule. (b), (e) AUPR.
(c), (f) $F_1$ score.}
\label{fig:real_noise}
\end{figure*}

Fig.~\ref{fig:real_noise} varies the position noise from $0.1$ to $2$\,m
and the reporting period from $m=1$ to $m=10$, which emulates
slot-based channel access in vehicle-to-everything communication. The
trends agree with the synthetic study. At $\bar\sigma_v=0.1$\,m,
Koopman-GKFA is within $3\%$ of the oracle EKF and attains $F_1=0.921$;
at $\bar\sigma_v=2$\,m its advantage over the best baseline, GKNet,
shrinks to $0.021$ in $F_1$ and $0.033$ in AUPR, in agreement with
Remark~\ref{rem:interpretation}(c). The known-topology filters
deteriorate fastest, their normalized RMSE rising from $1.15$--$1.19$
to $1.55$--$1.72$. Under round-robin sensing, the $F_1$ score of
Koopman-GKFA decreases from $0.862$ to $0.672$ as the reporting
fraction drops from $100\%$ to $10\%$, whereas that of the joint EKF
decreases from $0.785$ to $0.507$. At $m=10$, the margin over GKNet is
$0.029$ in $F_1$; as in Section~\ref{subsubsec:sensing}, each row of
$\hat\bA_k$ then receives only about $10$ measurements within the
effective window, so that the factor $m$ in~\eqref{eq:small_gain_m}
dominates. GKNet is less sensitive to $m$ than the online joint
estimators because its graph is fitted before the test record, but this
graph is not updated afterwards.

\subsubsection{Tracking Lane Changes}
\label{subsubsec:real_tracking}

\begin{figure*}[!t]
\centering
\includegraphics[width=\textwidth]{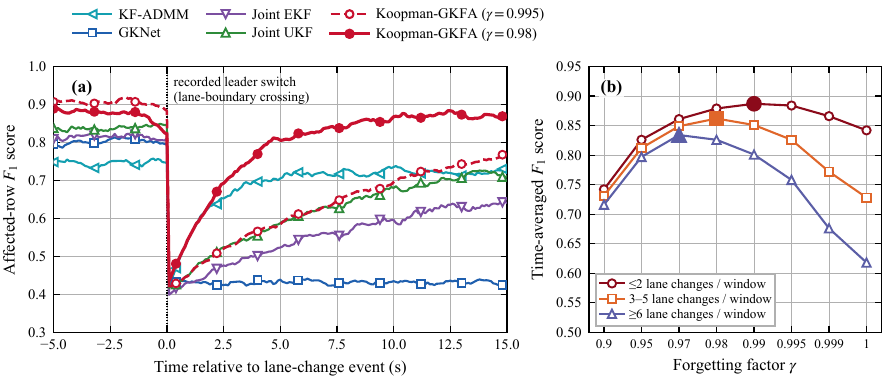}
\caption{Lane changes on NGSIM US-101 (both test periods). (a) $F_1$
score of the rows affected by a lane change (changing vehicle, old and
new follower), aligned at the recorded leader switch and averaged over
all events. (b) Time-averaged $F_1$ score of Koopman-GKFA versus the
forgetting factor $\gamma$ for windows with few, moderate and many lane
changes; filled markers indicate the best $\gamma$.}
\label{fig:real_tracking}
\end{figure*}

Fig.~\ref{fig:real_tracking}(a) aligns $312$ lane-change events at the 
time at which the recorded leader switches. At the switch, the $F_1$
score of the affected rows drops to about $0.42$ for all methods, since
the old leader becomes a false and the new leader a missed edge. With
$\gamma=0.98$, Koopman-GKFA returns to $0.8$ within $4.5$\,s and to
its pre-event level within about $10$\,s. With $\gamma=0.995$, the
pre-event level is higher ($0.905$ versus $0.882$), but the score
recovers only to $0.77$ within $15$\,s. The joint UKF and EKF recover
to $0.72$ and $0.64$, respectively; KF-ADMM recovers within about
$4$\,s, but only to its pre-event level of about $0.75$; and GKNet does not
adapt. A moderate decrease of the $F_1$ score of Koopman-GKFA starts
about $1$\,s \emph{before} the recorded switch. Lane changes take
several seconds~\cite{toledo2007modeling}, and drivers start to react
to the leader in the target lane before they cross the lane boundary,
at which the switch is recorded; the method detects this response,
which the binary ground truth counts as an error.

Fig.~\ref{fig:real_tracking}(b) confirms the trade-off of
Corollary~\ref{cor:tracking} on real data. For each class of windows,
the $F_1$ score is unimodal in $\gamma$, and the optimum moves from
$\gamma=0.99$ for windows with at most two lane changes to $0.98$ and
$0.97$ for moderate and frequent lane changes, i.e., the effective
window shrinks from $19.9$ to $6.6$\,s as the topology varies faster.
Without forgetting, the $F_1$ score falls to $0.618$ for frequent lane
changes. The optimal $\gamma$ is smaller than in
Section~\ref{subsubsec:tracking} because the NGSIM topology changes
within a few tens of samples; what transfers between applications is
the effective window $n_{\rm eff}\Delta t$ relative to the time between
topology changes, not the value of $\gamma$.

\subsubsection{Gene Regulatory Networks with Scarce Data}
\label{subsubsec:real_grn}

\begin{figure*}[!t]
\centering
\includegraphics[width=\textwidth]{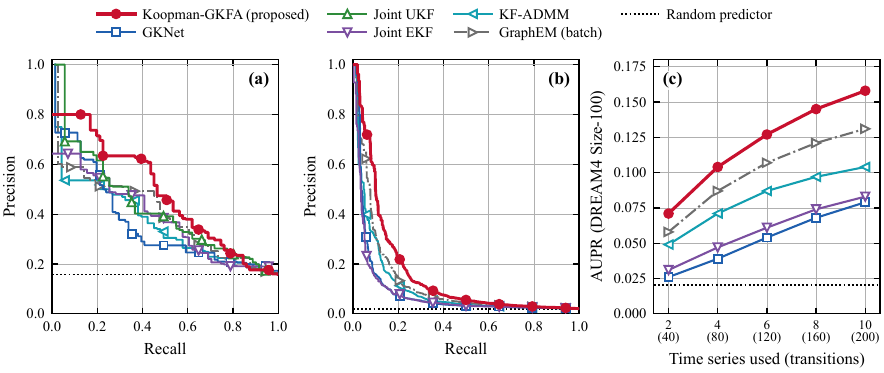}
\caption{DREAM4 In Silico. (a), (b) Interpolated precision--recall
curves pooled over the five Size-10 and Size-100 networks, respectively.
(c) AUPR on Size~100 versus the number of time series used (number of
transitions in parentheses). Dotted lines: random predictor.}
\label{fig:real_grn}
\end{figure*}

Fig.~\ref{fig:real_grn}(a) and (b) show precision--recall curves
pooled over the networks~\cite{davis2006relationship}. On both sizes,
the advantage of Koopman-GKFA is concentrated at low and moderate recall:
on Size~100, its precision at a recall of $0.2$ is about $0.22$,
compared with about $0.13$ for GraphEM. Beyond a recall of about
$0.5$, all curves approach the random level, which indicates that a
substantial part of the gold-standard edges is not identifiable from
these time series, e.g., edges whose regulator is rarely perturbed or
whose effect is masked by the multiplicative combination of
regulators.

Fig.~\ref{fig:real_grn}(c) varies the amount of data from $2$ to $10$
time series. With two series, each row has $40$ transitions but $99$ 
free weights, so the row Gram matrices $\hG_{i,k}$ are singular
(Remark~\ref{rem:design}(b)); the ridge term keeps the problem strongly
convex and the $\ell_1$ term selects a sparse support, which is the
regime of Proposition~\ref{prop:sparse}; its restricted-eigenvalue
condition is not verified here. In this regime, Koopman-GKFA attains an
AUPR of $0.071$, $3.4$ times the random level, against $0.058$ for
GraphEM and $0.026$ for GKNet. GKNet shows the steepest growth with
the amount of data, by a factor of $3.0$ compared with $2.2$ for
Koopman-GKFA, which is consistent with its larger number of trainable
parameters. None of the curves saturates, so the reported values are
limited by the length of the DREAM4 records.

\subsubsection{Ablation Study on Real Data}
\label{subsubsec:real_ablation}

\begin{table*}[!t]
\centering
\caption{Ablation study on real and benchmark data (mean\,$\pm$\,std).
US-101: 8:05--8:20, $40$ windows, RMSE$_v$ in m/s; $t_{0.8}$: time after
a recorded leader switch at which the event-aligned $F_1$ score of the
affected rows returns to $0.8$. DREAM4: five networks per size.
$\gamma=0.98$ on US-101 and $\gamma=0.999$ on DREAM4, except (A6). Best
results in bold and second best underlined.}
\label{tab:real_ablation}
\footnotesize
\setlength{\tabcolsep}{3.0pt}
\renewcommand{\arraystretch}{1.1}
\resizebox{\textwidth}{!}{%
\begin{tabular}{l l cc cc c cc cc c}
\toprule
 & \multirow{2}{*}{Variant} & \multicolumn{2}{c}{US-101, full sensing} & \multicolumn{2}{c}{US-101, round-robin ($m=4$)} & Lane change & \multicolumn{2}{c}{DREAM4 Size~10} & \multicolumn{2}{c}{DREAM4 Size~100} & \multirow{2}{*}{\shortstack{Time\\(ms/step)}}\\
\cmidrule(lr){3-4}\cmidrule(lr){5-6}\cmidrule(lr){7-7}\cmidrule(lr){8-9}\cmidrule(lr){10-11}
 & & RMSE$_v$ & $F_1$ & RMSE$_v$ & $F_1$ & $t_{0.8}$ (s) & AUROC & AUPR & AUROC & AUPR & \\
\midrule
(A0) & Koopman-GKFA (full) & \textbf{0.317$\pm$0.022} & \underline{0.862$\pm$0.035} & \textbf{0.352$\pm$0.027} & \textbf{0.789$\pm$0.044} & \underline{4.5} & \underline{0.768$\pm$0.049} & \underline{0.463$\pm$0.071} & 0.716$\pm$0.018 & 0.158$\pm$0.026 & 3.1\\
\midrule
(A1) & Node-separable lifting & 0.372$\pm$0.030 & 0.703$\pm$0.052 & 0.411$\pm$0.036 & 0.618$\pm$0.058 & $>$15 & 0.763$\pm$0.051 & 0.455$\pm$0.074 & \underline{0.717$\pm$0.019} & \underline{0.159$\pm$0.027} & \underline{2.4}\\
(A2) & w/o re-lifting (lifted KF) & 0.331$\pm$0.024 & 0.842$\pm$0.038 & 0.398$\pm$0.035 & 0.701$\pm$0.053 & 6.3 & 0.759$\pm$0.052 & 0.447$\pm$0.075 & 0.709$\pm$0.020 & 0.149$\pm$0.027 & 3.0\\
(A3) & Filtered-increment regression & 0.326$\pm$0.023 & 0.829$\pm$0.040 & 0.371$\pm$0.031 & 0.736$\pm$0.049 & 11.7 & 0.741$\pm$0.056 & 0.409$\pm$0.080 & 0.694$\pm$0.021 & 0.127$\pm$0.025 & 3.0\\
(A4) & $\ell_1$ only ($\mu_0=0$) & \textbf{0.317$\pm$0.022} & \textbf{0.864$\pm$0.035} & 0.363$\pm$0.033 & 0.752$\pm$0.055 & \textbf{4.4} & 0.761$\pm$0.055 & 0.451$\pm$0.079 & 0.703$\pm$0.022 & 0.139$\pm$0.029 & 4.0\\
(A5) & Ridge only ($\lambda_k=0$) & \underline{0.319$\pm$0.022} & 0.802$\pm$0.041 & \underline{0.354$\pm$0.028} & 0.744$\pm$0.047 & 5.6 & 0.765$\pm$0.050 & 0.452$\pm$0.072 & 0.712$\pm$0.018 & 0.149$\pm$0.025 & 2.6\\
(A6) & No forgetting ($\gamma=1$) & 0.341$\pm$0.026 & 0.704$\pm$0.049 & 0.372$\pm$0.030 & 0.651$\pm$0.054 & $>$15 & \textbf{0.770$\pm$0.048} & \textbf{0.466$\pm$0.070} & \textbf{0.718$\pm$0.018} & \textbf{0.160$\pm$0.026} & 3.1\\
(A7) & Single ADMM iteration ($R=1$) & \underline{0.319$\pm$0.022} & 0.855$\pm$0.036 & \underline{0.354$\pm$0.028} & \underline{0.781$\pm$0.045} & 4.8 & 0.766$\pm$0.050 & 0.459$\pm$0.072 & 0.714$\pm$0.018 & 0.155$\pm$0.026 & \textbf{2.2}\\
\bottomrule
\end{tabular}}
\end{table*}

Table~\ref{tab:real_ablation} repeats the ablation of
Section~\ref{subsubsec:ablation} on the real data. Its most informative
result concerns the representation. The node-separable lifting (A1)
reduces the $F_1$ score on NGSIM from $0.862$ to $0.703$ and never
recovers to $0.8$ after a lane change, whereas on DREAM4 it is within
$0.008$ of the full method in AUPR and even marginally better on
Size~100. As on synthetic data, (A1) also changes the physical
interaction model. On NGSIM, the sender-only structure cannot represent
the receiver dependence of $\tanh(\kappa(v_j-v_i))$. On DREAM4, the Hill
coupling is sender-only, and with linear sensors only the state rows
enter the estimates (Proposition~\ref{prop:ekf}); both structures can
represent $H(x_j)$ in these rows, whereas the cross terms of
Proposition~\ref{prop:separable_fails} affect only non-state
coordinates. In these experiments the interaction-aware lifting is thus
beneficial for the receiver-dependent coupling and brings no measurable
benefit for the sender-only coupling, where its additional features may
add variance.

The inference components behave as on synthetic data, with differences
in magnitude that are explained by the data. Without re-lifting (A2),
the loss is small under full sensing but large under round-robin
sensing, where the RMSE$_v$ increases from $0.352$ to $0.398$\,m/s,
because the lifted coordinates of vehicles that do not report drift
away from $\bphi(\hat x)$. Regressing filtered increments (A3) more than
doubles the recovery time after a lane change and causes the largest
loss on DREAM4 ($0.054$ in AUPR on Size~10): with only $20$ transitions
per series, the information that enters through $\bK_k\bH_k$ alone is
insufficient. Removing the ridge term (A4) is immaterial under full
sensing on NGSIM but costs $0.037$ in $F_1$ under round-robin sensing
and $0.019$ in AUPR on DREAM4 Size~100, where the rows are
underdetermined. Removing the $\ell_1$ term (A5) reduces the $F_1$ score
from $0.862$ to $0.802$ but leaves AUROC and AUPR nearly unchanged: the
$\ell_1$ penalty matters for the selection of the support, not for the
ranking of the weights. Without forgetting (A6), the results on the
static DREAM4 networks are marginally the best in the table, whereas
lane changes are no longer tracked. A single ADMM iteration (A7) reduces
the run time by $29\%$ at a loss of at most $0.008$ in any score.

\subsubsection{Visualization}
\label{subsubsec:real_visual}

\begin{figure*}[!t]
\centering
\includegraphics[width=0.78\textwidth]{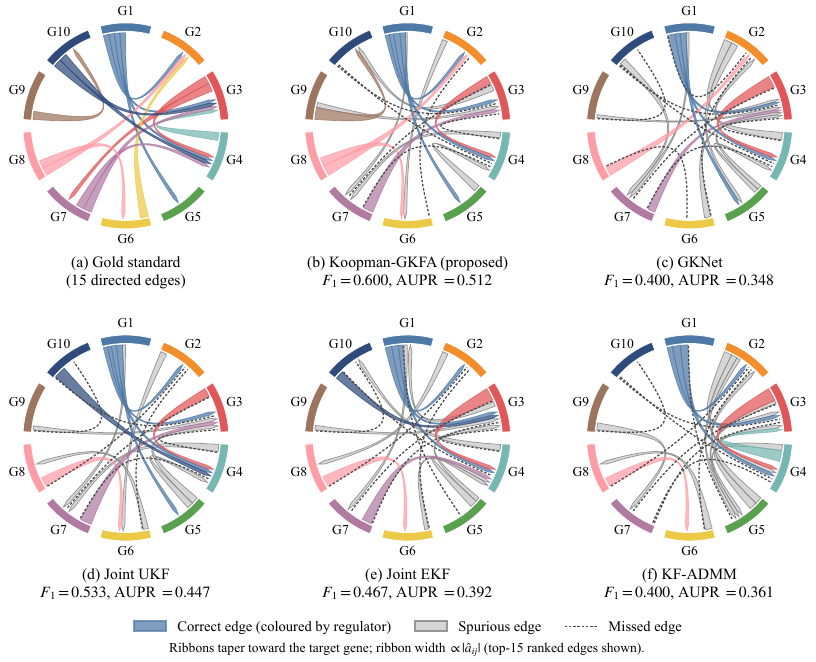}
\caption{Chord diagrams of the $15$ highest-ranked edges for DREAM4
Size-10 network~1. Coloured ribbons: correct edges, coloured by
regulator; grey ribbons: spurious edges; dashed lines: missed
gold-standard edges. Ribbons taper toward the target gene, and their
width is proportional to $|\hat a_{ij}|$. Since $15$ edges are
predicted, precision, recall and $F_1$ coincide.}
\label{fig:real_chord}
\end{figure*}

Fig.~\ref{fig:real_chord} compares the $15$ highest-ranked edges of each
method with the gold standard of DREAM4 Size-10 network~1, which has
the same number of edges. All methods recover most out-edges of the hub
G1, whose variations dominate the expression data. Koopman-GKFA
recovers $9$ of the $15$ edges, the joint UKF $8$, the joint EKF $7$,
and GKNet and KF-ADMM $6$ each. The errors are of different kinds. Of
the six spurious edges of Koopman-GKFA, three shortcut a two-step path
of the gold standard (e.g., G1$\to$G7 through G3), two reverse a true
edge (e.g., G3$\to$G10), and one connects two targets of G1; these are
the confusions that one-step prediction errors cannot fully resolve at
a sampling interval of $50$\,min. In contrast, seven of the nine
spurious edges of KF-ADMM connect the co-regulated targets G2--G5 of
G1, which is characteristic of a linear model that explains
co-expression by direct coupling. The edges G3$\to$G7 and G6$\to$G2 are
missed by all methods, and each of the four edges from G7 and G10 to
G3 and G4 is recovered by at most three methods. Together with the
feedback pair G3$\leftrightarrows$G4, these genes form a densely
interconnected module whose individual couplings are difficult to
separate from $21$-sample records.

% \subsubsection{Summary}
% \label{subsubsec:real_summary}

% On real traffic data and on the DREAM4 benchmark, Koopman-GKFA gives the
% most accurate joint estimates among the evaluated methods, and the
% qualitative behaviour predicted by the analysis carries over: the
% benefit of the interaction-aware lifting is tied to receiver-dependent
% coupling, the forgetting factor trades static accuracy for tracking,
% and the advantage narrows when the measurements become uninformative.
% The margins are smaller than on synthetic data, the gap to the oracle
% is larger because of model mismatch, and on DREAM4 the absolute
% accuracy of all methods remains limited by the data; on \nrev{DREAM4}
% the differences between the best methods are not
% statistically significant. These results indicate that the method
% \nrev{tolerates the model mismatch present in these data sets}, but its accuracy
% is ultimately bounded by the fidelity of the node model and the
% information content of the measurements.

%% ===============================================================
\section{Conclusion}
\label{sec:conclusion}
%% ===============================================================

This paper has studied joint state estimation and topology inference
for nonlinear networks whose node dynamics are learned from data, and
has treated the filter and the topology learner as a feedback
interconnection whose reliability must be established for the loop
rather than for its parts. On the representation side, we showed that
additive node-separable lifting with nonlinear observables cannot
retain a coupling that is linear in the physical edge weights,
constructed an interaction-aware lifting whose network-level residual
is explicit, tight and free of second-order terms on the state
coordinates, and characterized when calibration data
determine the reference topology. On the inference side, we showed that
undetectable lifted modes cannot be stabilized by any gain and that, for
linear sensors, a consistently re-lifted filter is the EKF of the
learned surrogate; the topology is estimated by a prediction-error
elastic-net regression on the free physical weights, solved by an ADMM
with a certified accuracy. On the certification side, we bounded the
topology error under estimated states, closed the loop by a small-gain
theorem with explicit ultimate bounds, showed that one-step contraction
fails when the self-dynamics leave an unmeasured direction
invariant, and extended the guarantee to intermittent
sensing by a multi-step condition; restricted-eigenvalue, tracking and
support-recovery results complete the analysis.

The results also delimit the method. The lifting is first order in the coupling, and with linear sensors its
benefit for state estimation reduces to that of an EKF of the same
surrogate, also for non-affine self-dynamics. A controlled
comparison showed no accuracy advantage over a physical-space surrogate
that shares the state rows of the lifting, and a small disadvantage when
the feature rows are poorly closed, the process noise is large or the
calibration data are scarce. The value of the framework lies in the
learnable, edge-affine model class, which the competitive physical
surrogate itself uses, and, for sensors of nonlinear features, in a
measurement model that is linear in the lifted state and is covered by
the guarantees. The guarantees are conditional: they treat
state errors adversarially, require excitation and multi-step
contraction, and rely on prior bounds on model mismatch, noise and
topology drift that cannot be observed online; the data-computable part
of their constants serves as a diagnostic. Identifiability is a property
of the operating regime and is lost under phase locking. The multi-step data of Remark~\ref{rem:reldeg} are covered only
empirically, and the controlled comparison is limited to one class
of oscillator networks with a known sensor map. Future work
will address bias-corrected (instrumental-variable or smoothing-based)
regressions, an analysis of multi-step prediction errors, controlled comparisons with physical-space surrogates, nonlinear
self-dynamics, time-varying Lyapunov metrics for the loop,
higher-order interaction features for strong coupling, closed-loop
high-probability and edge-wise inference, and decentralized
implementations with learned dictionaries~\cite{lusch2018deep,fan2026learning}.

\appendices

\section{Proof of Proposition~\ref{prop:separable_fails}}
\label{app_prop_sep}
\begin{proof}
(i) Since $f_1(x)=x$ and the sender-only coupling is $g(y)=y$, one has
$\bpsi(f_1(x))=\bpsi(x)=\bI_2\bpsi(x)$ and $\bpsi(g(y))=\bI_2\bpsi(y)$ for all
$x,y$, so $\varepsilon_f=\varepsilon_g=0$.

(ii) The noise-free successor of node $1$ is $x_1+ax_2$, so the
observable $x\mapsto x^2$ evaluated at the successor equals
$x_1^2+2ax_1x_2+a^2x_2^2$. Let $u(x_1)+v(x_2)$ be any additive
prediction of it; under~\eqref{eq:separable_model} one has
$u(x_1)=[\bF_1^\phi\bpsi(x_1)]_m$ and $v(x_2)=a[\bGamma^\phi\bpsi(x_2)]_m$, where
$m$ indexes the observable, for any $N$. The residual is
$r(x_1,x_2)=2ax_1x_2+\tilde u(x_1)+\tilde v(x_2)$ with
$\tilde u(x_1)=x_1^2-u(x_1)$ and $\tilde v(x_2)=a^2x_2^2-v(x_2)$. Define the
mixed difference
$\mathsf{D}r\triangleq r(1,1)-r(1,-1)-r(-1,1)+r(-1,-1)$. For any additive
function $\tilde u(x_1)+\tilde v(x_2)$,
$\mathsf{D}(\tilde u+\tilde v)=\tilde u(1)+\tilde v(1)-\tilde u(1)-\tilde v(-1)
-\tilde u(-1)-\tilde v(1)+\tilde u(-1)+\tilde v(-1)=0$, whereas
$\mathsf{D}(2ax_1x_2)=2a(1+1+1+1)=8a$. Hence $|\mathsf{D}r|=8|a|$. Since
$|\mathsf{D}r|\le4\sup_{[-1,1]^2}|r|$,~\eqref{eq:separable_lower}
follows.
\end{proof}

\section{Proof of Theorem~\ref{thm:lifting} and Corollary~\ref{cor:special}}
\label{app_thm1}
\begin{proof}[Proof of Theorem~\ref{thm:lifting}]
(i) Fix $x$ and write $y\triangleq f_i(x_i)$ and
$\bc_i\triangleq\sum_ja_{ij}g(x_i,x_j)$, so that $\tilde x_i^+=y+\bc_i$.
Adding and subtracting $\bphi(y)$ and
$\bJ_\phi(y)\bc_i=\sum_ja_{ij}\bJ_\phi(f_i(x_i))g(x_i,x_j)$
in~\eqref{eq:delta_def} gives the exact decomposition
\begin{align}
  \bdelta_i(x)=\;&\underbrace{\bigl[\bphi(y+\bc_i)-\bphi(y)-\bJ_\phi(y)\bc_i\bigr]}_{\triangleq\,\bdelta^{(2)}_i}\nonumber\\
  &+\underbrace{\bigl[\bphi(f_i(x_i))-\bF_i\bphi(x_i)\bigr]}_{\triangleq\,\bdelta^{(f)}_i}\nonumber\\
  &+\sum_ja_{ij}\underbrace{\bigl[\bJ_\phi(f_i(x_i))g(x_i,x_j)-\bB_i\bchi(x_i,x_j)\bigr]}_{\triangleq\,\bdelta^{(\chi)}_{ij}}.
  \label{eq:proof_decomp}
\end{align}
By the fundamental theorem of calculus,
$\bphi(y+\bc_i)-\bphi(y)=\int_0^1\bJ_\phi(y+\tau\bc_i)\bc_i\,d\tau$, hence
$\bdelta^{(2)}_i=\int_0^1(\bJ_\phi(y+\tau\bc_i)-\bJ_\phi(y))\bc_i\,d\tau$ and, by the $L_2$-Lipschitz continuity of $\bJ_\phi$,
$\norm{\bdelta^{(2)}_i}\le\int_0^1L_2\tau\norm{\bc_i}^2d\tau=\frac{L_2}{2}\norm{\bc_i}^2$.
By definition, $\norm{\bdelta^{(f)}_i}=\varepsilon_{f,i}(x_i)$ and
$\norm{\bdelta^{(\chi)}_{ij}}=\varepsilon_{\chi,i}(x_i,x_j)$. The triangle
inequality and $\norm{\bc_i}\le\sum_j|a_{ij}|\norm{g(x_i,x_j)}$
give~\eqref{eq:lifted_noise_bound}.

(ii) By (A1), $\bPi\bphi(z)=z$ for all $z$; differentiating gives
$\bPi\bJ_\phi(z)=\bI_n$. Therefore
$\bPi\bdelta^{(2)}_i=(y+\bc_i)-y-\bc_i=\bzero$,
$\bPi\bdelta^{(f)}_i=f_i(x_i)-\bPi\bF_i\bphi(x_i)$ and
$\bPi\bdelta^{(\chi)}_{ij}=g(x_i,x_j)-\bPi\bB_i\bchi(x_i,x_j)$. Applying $\bPi$
to~\eqref{eq:proof_decomp} gives~\eqref{eq:state_block}. If
$\bPi\bF_i\bphi\equiv f_i$ and $\bPi\bB_i\bchi\equiv g$, both brackets
in~\eqref{eq:state_block} vanish for all $x$ and $\bA$. Conversely,
$\bA=\bO$ gives $f_i\equiv\bPi\bF_i\bphi$; then choosing $a_{ij}=1$ for
one $j\neq i$ and all other weights zero gives
$g(x_i,x_j)=\bPi\bB_i\bchi(x_i,x_j)$ for all $(x_i,x_j)$, because $x_i$
and $x_j$ vary independently. Finally, matrices with
$\bPi\bF_i\bphi\equiv f_i$ exist if and only if every component of
$f_i$ is a linear combination of the components of $\bphi$, since the
state rows $\bPi\bF_i$ are unconstrained; the same holds for
$\bPi\bB_i=\bB_x$ and $g$, because $g$ is common to all nodes.

(iii) For $\bphi(x)=[x,x^2]^\top$, $\bJ_\phi(x)=[1,2x]^\top$, so
$\norm{\bJ_\phi(x)-\bJ_\phi(x')}=2|x-x'|$ and $L_2=2$. With
$f_1(x)=x$ and $\bF_1=\bI_2$, $\varepsilon_{f,1}=0$. The ideal interaction
term is $\bJ_\phi(x_1)x_2=[x_2,2x_1x_2]^\top=\diag(1,2)\bchi(x_1,x_2)$, so
$\varepsilon_{\chi,1}=0$. Hence
$\bdelta_1=\bdelta^{(2)}_1=\bphi(x_1+ax_2)-\bphi(x_1)-a[x_2,2x_1x_2]^\top
=[0,a^2x_2^2]^\top$, whose norm $a^2x_2^2$ equals
$\frac{L_2}{2}(|a||x_2|)^2$, the right-hand side
of~\eqref{eq:lifted_noise_bound}.
\end{proof}

\begin{proof}[Proof of Corollary~\ref{cor:special}]
At most $d$ weights in row $i$ are nonzero and each is bounded by
$a_{\max}$, so $\sum_j|a_{ij}|\varepsilon_{\chi,i}\le da_{\max}\varepsilon_\chi$
and $\sum_j|a_{ij}|\norm{g}\le da_{\max}g_\infty$;
\eqref{eq:uniform_residual} follows from~\eqref{eq:lifted_noise_bound}.
In the sender-only case,
$\sum_ja_{ij}\bGamma\bphi(x_j)=\sum_ja_{ij}\bGamma\bpsi_j=[(\bA\otimes\bGamma)\bPsi]_i$
by the definition of the Kronecker product, and the expression of
$\varepsilon_{\chi,i}$ is~\eqref{eq:eps_chi} with $g(x_i,x_j)=g(x_j)$.
\end{proof}

\section{Proofs of Proposition~\ref{prop:lifted_noise} and Lemma~\ref{lem:localization}}
\label{app_noise}
\begin{proof}[Proof of Proposition~\ref{prop:lifted_noise}]
Write $\tilde x\triangleq\tilde x_{i,k+1}$ and $w\triangleq w_{i,k}$. As
in Appendix~\ref{app_thm1}, $\bxi_{i,k}=\bJ_\phi(\tilde x)w+\bvarrho$ with
$\bvarrho=\int_0^1(\bJ_\phi(\tilde x+\tau w)-\bJ_\phi(\tilde x))w\,d\tau$ and
$\norm{\bvarrho}\le\frac{L_2}{2}\norm{w}^2$. Since $\tilde x$ is
$\calF_k$-measurable, $w$ is independent of $\calF_k$ with zero mean, and
$\norm{\bJ_\phi}\le L_\phi$ ensures integrability,
$\expect[\bJ_\phi(\tilde x)w\mid\calF_k]=\bJ_\phi(\tilde x)\expect[w]=\bzero$.
Hence
$\norm{\expect[\bxi_{i,k}\mid\calF_k]}=\norm{\expect[\bvarrho\mid\calF_k]}
\le\frac{L_2}{2}\expect[\norm{w}^2\mid\calF_k]=\frac{L_2}{2}\tr(\bQ_{i,k})$,
which is~\eqref{eq:xi_mean}. The Lipschitz property gives
$\norm{\bxi_{i,k}}\le L_\phi\norm{w}$ and thus~\eqref{eq:xi_var}. By (A1),
$\bPi\bxi_{i,k}=(\tilde x+w)-\tilde x=w$. For $\bphi(x)=[x,x^2]^\top$,
$\bxi_{i,k}=[w,\;2\tilde xw+w^2]^\top$. Since $\tilde x$ is
$\calF_k$-measurable and integrable and $w$ is independent of $\calF_k$
with zero mean, $\expect[\bxi_{i,k}\mid\calF_k]=[0,\sigma^2]^\top$, whose
norm is $\sigma^2=\frac{L_2}{2}\sigma^2$ with $L_2=2$. Conditionally on
$\calF_k$, the second component equals $(w+\tilde x)^2-\tilde x^2\ge-\tilde x^2$;
its conditional law has support bounded from below and is therefore
not Gaussian.
\end{proof}

\begin{proof}[Proof of Lemma~\ref{lem:localization}]
Let $\mathbf{1}_{\rm in}$ and $\mathbf{1}_{\rm out}$ be the indicators of
$\{\norm{x}\le R\}$ and $\{\norm{x}>R\}$. Then
$\expect\norm{\bdelta}^2=\expect[\norm{\bdelta}^2\mathbf{1}_{\rm in}]
+\expect[\norm{\bdelta}^2\mathbf{1}_{\rm out}]\le\varepsilon^2
+(\expect\norm{\bdelta}^4)^{1/2}\prob(\norm{x}>R)^{1/2}$ by the
Cauchy--Schwarz inequality, and
$\expect\norm{\bdelta}^4\le c_0^4\expect(1+\norm{x}^m)^4<\infty$.
Since $\sqrt{\alpha^2+\beta^2}\le\alpha+\beta$ for $\alpha,\beta\ge0$,
\eqref{eq:localization} follows. The tail bound is Markov's
inequality applied to $\norm{x}^2$.
\end{proof}

\section{Proof of Lemma~\ref{lem:homomorphism}}
\label{app_lemma1}
\begin{proof}
(i) By the definition of the Kronecker product, $[\bT]_{(i,j)}=a_{ij}\bGamma$.
If $a_{ij}=0$ the block vanishes; conversely, if $a_{ij}\bGamma=\bO$ and
$\bGamma\neq\bO$, some entry $\gamma_{mn}\neq0$ and $a_{ij}\gamma_{mn}=0$
forces $a_{ij}=0$. Absolute homogeneity of the Frobenius norm gives
$\normF{a_{ij}\bGamma}=|a_{ij}|\normF{\bGamma}$, and summation gives the
last identity.

(ii) By the reverse triangle inequality,
$\bigl|\normF{[\hat\bT]_{(i,j)}}-\normF{a_{ij}\bGamma}\bigr|\le
\normF{[\hat\bT]_{(i,j)}-a_{ij}\bGamma}\le\delta$. If $(i,j)\in\calE$,
then $\normF{[\hat\bT]_{(i,j)}}\ge a_{\min}\normF{\bGamma}-\delta>\tau$; if
$(i,j)\notin\calE$, then $a_{ij}=0$ and
$\normF{[\hat\bT]_{(i,j)}}\le\delta<\tau$. Hence $\hat\calE_\tau=\calE$.

(iii) The function
$a\mapsto\normF{[\hat\bT]_{(i,j)}-a\bGamma}^2=\normF{[\hat\bT]_{(i,j)}}^2
-2a\inner{[\hat\bT]_{(i,j)}}{\bGamma}_{\rm F}+a^2\normF{\bGamma}^2$ is
minimized at $\hat a_{ij}$. Moreover
$\hat a_{ij}-a_{ij}=\inner{[\hat\bT]_{(i,j)}-a_{ij}\bGamma}{\bGamma}_{\rm F}/\normF{\bGamma}^2$,
and the Cauchy--Schwarz inequality gives the bound.

(iv) $\diag(1,0)+\diag(0,1)=\bI_2$, while $\diag(1,0)\neq c\bI_2$ and
$\diag(0,1)\neq c\bI_2$ for every $c\in\mathbb{R}$.
\end{proof}

\begin{oldrevblock}
\section{Proofs of Propositions~\ref{prop:undetectable} and~\ref{prop:ekf}}
\label{app_filter}
\begin{proof}[Proof of Proposition~\ref{prop:undetectable}]
(i) For every $\bK$,
$(\bI-\bK\bH)\bL\bv=\bL\bv-\bK\bH\bL\bv=\lambda\bv-\lambda\bK\bH\bv=\lambda\bv$, so
$\lambda$ is an eigenvalue of the real matrix $(\bI-\bK\bH)\bL$. Hence
$\rho((\bI-\bK\bH)\bL)\ge|\lambda|\ge1$, and every induced matrix norm
dominates the spectral radius.

(ii) Let $\bm\Pi_k\triangleq(\bI-\bK_k\bH)\bL\cdots(\bI-\bK_1\bH)\bL$. By the
argument of (i), applied at every step, $\bm\Pi_k\bv=\lambda^k\bv$. Write
$\bv=\bv_R+\imath\bv_I$ with real $\bv_R,\bv_I$. Since $\bm\Pi_k$ is real,
$\bm\Pi_k\bv_R=\mathrm{Re}(\lambda^k\bv)$ and $\bm\Pi_k\bv_I=\mathrm{Im}(\lambda^k\bv)$,
so that
$\norm{\bm\Pi_k\bv_R}^2+\norm{\bm\Pi_k\bv_I}^2=|\lambda|^{2k}\norm{\bv}^2\ge\norm{\bv}^2>0$
for all $k$. Hence at least one of the two real solutions
$\be_k=\bm\Pi_k\bv_R$ and $\be_k=\bm\Pi_k\bv_I$ does not converge to zero.

(iii) With $s\triangleq\omega_i\Delta t$, the last two rows of $\bF_i$ are
$[0,0,\cos s,-\sin s]$ and $[0,0,\sin s,\cos s]$, and $\bv_i$ has zero
first two entries. Hence the first two entries of $\bF_i\bv_i$ vanish and
the last two are
$\cos s+\imath\sin s=e^{\imath s}$ and $\sin s-\imath\cos s=-\imath e^{\imath s}$,
i.e., $\bF_i\bv_i=e^{\imath s}\bv_i$. Moreover $\bH_i\bv_i=0$. Embedding
$\bv_i$ in the $i$-th block of $\bv\in\mathbb{C}^{Np}$ gives
$\bL\bv=e^{\imath s}\bv$ and $\bH\bv=\bzero$.
\end{proof}

\begin{proof}[Proof of Proposition~\ref{prop:ekf}]
Fix $k$ and write $\bP^x\triangleq\Pip\bP_{k|k-1}\Pip^\top$ and
$\bK^x\triangleq\Pip\bK_k$. Since $\Pip$ is linear and
$\Pip\hat\bM_{k-1}=\partial_x\Pip[\bL\bPhi(x)+\bU(x)\hat\ba_{k-1}]_{x=\hat x_{k-1}}$,
\eqref{eq:kf_pred_cov} gives
$\bP^x=\bm{\mathcal{J}}\bP_{k-1}\bm{\mathcal{J}}^\top+\Pip\bar\bQ_{k-1}\Pip^\top$
with $\bm{\mathcal{J}}\triangleq\partial_x\mathfrak{f}(x;\hat\ba_{k-1})|_{x=\hat x_{k-1}}$.
By $\bH_k=\bC_k\Pip$, $\bH_k\bP_{k|k-1}\bH_k^\top=\bC_k\bP^x\bC_k^\top$ and
$\Pip\bP_{k|k-1}\bH_k^\top=\bP^x\bC_k^\top$, so~\eqref{eq:kalman_gain} yields
$\bK^x=\bP^x\bC_k^\top(\bC_k\bP^x\bC_k^\top+\bR_k)^{-1}$. By~\eqref{eq:kf_pred_state},
$\Pip\hat\bPsi_{k|k-1}=\mathfrak{f}(\hat x_{k-1};\hat\ba_{k-1})$ and
$\bH_k\hat\bPsi_{k|k-1}=\bC_k\mathfrak{f}(\hat x_{k-1};\hat\ba_{k-1})$, so
\eqref{eq:kf_update_state} reads
$\hat x_k=\mathfrak{f}(\hat x_{k-1};\hat\ba_{k-1})+\bK^x(z_k-\bC_k\mathfrak{f}(\hat x_{k-1};\hat\ba_{k-1}))$.
Finally, $\Pip(\bI-\bK_k\bH_k)=\Pip-\bK^x\bC_k\Pip=(\bI-\bK^x\bC_k)\Pip$, so
\eqref{eq:kf_update_cov} reads
$\bP_k=(\bI-\bK^x\bC_k)\bP^x(\bI-\bK^x\bC_k)^\top+\bK^x\bR_k(\bK^x)^\top$. These
are the EKF prediction, gain, mean update and Joseph-form covariance
update for the stated model; note that the update depends on
$\bP_{k|k-1}$ only through $\bP^x$, and on $\hat\bPsi_{k-1}$ only through
$\hat x_{k-1}$ by~\eqref{eq:relift}. If $\Pip\bdelta\equiv\bzero$, then
by~\eqref{eq:delta_def} and (A1),
$\Pip[\bL\bPhi(x)+\bU(x)\ba]=\Pip\bPhi(\tilde x^+)=\tilde x^+$, i.e.,
$\mathfrak{f}_i(x;\ba)=f_i(x_i)+\sum_ja_{ij}g(x_i,x_j)$.
\end{proof}
\end{oldrevblock}

\section{Proofs of Lemma~\ref{lem:prox} and Theorem~\ref{thm:admm_convergence}}
\label{app_thm2}
\begin{proof}[Proof of Lemma~\ref{lem:prox}]
$\mathcal{R}$ and $\frac{\beta}{2}\norm{\ba-\bv}^2$ are sums of functions
of single entries, so the proximal problem decouples. For \oldrev{each free
coordinate $m$} the scalar function
$\varphi(a)=\frac{\lambda_k}{\beta}|a|+\frac12(a-v_{m})^2$ is strictly
convex with unconstrained minimizer
$a^\circ=\sign(v_{m})\max(|v_{m}|-\lambda_k/\beta,0)$. A strictly convex
function of one variable is decreasing on $(-\infty,a^\circ]$ and
increasing on $[a^\circ,\infty)$; hence its minimizer over
$[l_{m},u_{m}]$ is $a^\circ$ if $a^\circ\in[l_{m},u_{m}]$, $l_{m}$ if
$a^\circ<l_{m}$, and $u_{m}$ if $a^\circ>u_{m}$, i.e.,
$\clip_{[l_{m},u_{m}]}(a^\circ)$.
\end{proof}

\begin{proof}[Proof of Theorem~\ref{thm:admm_convergence}]
\oldrev{Write $F(\ba)\triangleq\frac12\sum_t\omega_{k,t}\norm{\hat\by_t-\hD_t\ba}^2+\frac{\mu_0}{2}\norm{\ba}^2
=\frac12\ba^\top\hat\bG\ba-\hat\bb^\top\ba+{\rm const}$, with $\hat\bG=\hG_{\mu,k}$,} and let
$\mathcal{R}$ be as in Section~\ref{subsubsec:block_soft}. Define the
resolvents $J_F(\bs)\triangleq\argmin_{\ba}\{F(\ba)+\frac\beta2\norm{\ba-\bs}^2\}
=(\hat\bG+\beta\bI)^{-1}(\hat\bb+\beta\bs)$ and
$J_{\mathcal{R}}\triangleq\prox_{\beta^{-1}\mathcal{R}}$, the reflections
$R_F\triangleq2J_F-I$, $R_{\mathcal{R}}\triangleq2J_{\mathcal{R}}-I$, and the
Douglas--Rachford operator $T\triangleq\frac12(I+R_FR_{\mathcal{R}})$.

\emph{Step 1 (ADMM as a fixed-point iteration).} By~\eqref{eq:T_update}
and~\eqref{eq:C_update}, $\ba^{(r+1)}=J_F(\bc^{(r)}-\bu^{(r)})$ and
$\bc^{(r+1)}=J_{\mathcal{R}}(\bt^{(r)})$. By~\eqref{eq:D_update},
$\bu^{(r+1)}=\bt^{(r)}-J_{\mathcal{R}}(\bt^{(r)})$, hence
$\bc^{(r+1)}-\bu^{(r+1)}=2J_{\mathcal{R}}(\bt^{(r)})-\bt^{(r)}=R_{\mathcal{R}}(\bt^{(r)})$
and therefore $\ba^{(r+2)}=J_F(R_{\mathcal{R}}(\bt^{(r)}))$. Consequently,
\begin{align*}
  \bt^{(r+1)}&=\ba^{(r+2)}+\bu^{(r+1)}
  =\bt^{(r)}+J_F(R_{\mathcal{R}}\bt^{(r)})-J_{\mathcal{R}}(\bt^{(r)})\\
  &=\tfrac12\bigl(\bt^{(r)}+R_FR_{\mathcal{R}}\bt^{(r)}\bigr)=T\bt^{(r)},
\end{align*}
where the third equality uses
$R_FR_{\mathcal{R}}\bt=2J_F(R_{\mathcal{R}}\bt)-2J_{\mathcal{R}}\bt+\bt$.

\emph{Step 2 (fixed points are minimizers).} Since
$T\bt=\bt+J_F(R_{\mathcal{R}}\bt)-J_{\mathcal{R}}\bt$, $\bt\in{\rm Fix}\,T$ iff
$J_F(R_{\mathcal{R}}\bt)=J_{\mathcal{R}}\bt\triangleq\by$. The optimality
conditions of the two resolvents read $\beta(\bt-\by)\in\partial\mathcal{R}(\by)$
and $\nabla F(\by)+\beta(\by-(2\by-\bt))=\bzero$, i.e.,
$\nabla F(\by)=-\beta(\bt-\by)$. Hence
$\bzero\in\nabla F(\by)+\partial\mathcal{R}(\by)\subseteq\partial(F+\mathcal{R})(\by)$
and $\by$ is a minimizer. Conversely, if $\hat\ba$ is a minimizer, then
$-\nabla F(\hat\ba)\in\partial\mathcal{R}(\hat\ba)$ because $F$ is finite
and differentiable everywhere, and
$\bt^\star\triangleq\hat\ba-\beta^{-1}\nabla F(\hat\ba)$ satisfies
$J_{\mathcal{R}}\bt^\star=\hat\ba$ and $J_F(R_{\mathcal{R}}\bt^\star)
=J_F(\hat\ba+\beta^{-1}\nabla F(\hat\ba))=\hat\ba$ (the optimality condition
$\nabla F(\ba)+\beta(\ba-\hat\ba)-\nabla F(\hat\ba)=\bzero$ holds at
$\ba=\hat\ba$ and determines the unique minimizer). Thus
$\bt^\star\in{\rm Fix}\,T$.

\emph{Step 3 (part (i)).} \oldrev{$F\ge0$ and $F+\mathcal{R}$ is proper and lower
semicontinuous. It is coercive: if $\mu_0>0$, $F$ is strongly convex; if
$\lambda_k>0$, $F+\mathcal{R}\ge\lambda_k\norm{\ba}_1$ on $\calA$. Hence a minimizer exists;} by Step~2,
${\rm Fix}\,T\neq\emptyset$. Resolvents of proper, closed, convex
functions are firmly nonexpansive, hence their reflections are
nonexpansive~\cite{bauschke2017convex}; so $R_FR_{\mathcal{R}}$ is
nonexpansive and $T$ is $\frac12$-averaged. By the
Krasnosel'ski\u{\i}--Mann theorem~\cite[Thm.~5.15]{bauschke2017convex},
$\bt^{(r)}\to\bar\bt\in{\rm Fix}\,T$. By continuity,
$\bc^{(r+1)}=J_{\mathcal{R}}\bt^{(r)}\to J_{\mathcal{R}}\bar\bt$ and
$\ba^{(r+2)}=J_F(R_{\mathcal{R}}\bt^{(r)})\to J_F(R_{\mathcal{R}}\bar\bt)=J_{\mathcal{R}}\bar\bt$,
which is a minimizer by Step~2; hence $\ba^{(r)}-\bc^{(r)}\to\bzero$.

\emph{Step 4 (part (ii)).} $R_F$ is affine with linear part
$2\beta(\hat\bG+\beta\bI)^{-1}-\bI=(\hat\bG+\beta\bI)^{-1}(\beta\bI-\hat\bG)$,
a symmetric matrix with eigenvalues $(\beta-g)/(\beta+g)$, $g\in[\mu_k,\bar\mu_k]$.
The function $g\mapsto|\beta-g|/(\beta+g)$ is quasiconvex on
$[\mu_k,\bar\mu_k]$, so its maximum is attained at an endpoint and equals
$\delta(\beta)$; since $\mu_k>0$, $\delta(\beta)<1$. Hence $R_F$ is
$\delta(\beta)$-Lipschitz, $R_FR_{\mathcal{R}}$ is $\delta(\beta)$-Lipschitz,
and $T$ is a contraction with modulus $\rho(\beta)=(1+\delta(\beta))/2$. By
Banach's theorem, ${\rm Fix}\,T=\{\bt^\star\}$ and
$\norm{\bt^{(r)}-\bt^\star}\le\rho(\beta)^r\norm{\bt^{(0)}-\bt^\star}$. The
minimizer is unique by strong convexity, and $\hat\ba_k=J_{\mathcal{R}}\bt^\star
=J_F(R_{\mathcal{R}}\bt^\star)$. Since $J_{\mathcal{R}}$ is nonexpansive and $J_F$
is $\beta/(\beta+\mu_k)$-Lipschitz,
$\norm{\bc^{(r+1)}-\hat\ba_k}\le\norm{\bt^{(r)}-\bt^\star}$ and
$\norm{\ba^{(r+2)}-\hat\ba_k}\le\norm{R_{\mathcal{R}}\bt^{(r)}-R_{\mathcal{R}}\bt^\star}
\le\norm{\bt^{(r)}-\bt^\star}$, which gives~\eqref{eq:linear_rate}.
For $\beta\in[\mu_k,\bar\mu_k]$, the first term in $\delta(\beta)$ is
increasing and the second decreasing in $\beta$, and both are larger
outside this interval; they coincide at $\beta^\star=\sqrt{\mu_k\bar\mu_k}$,
where $\delta(\beta^\star)=(\sqrt{\kappa_k}-1)/(\sqrt{\kappa_k}+1)$ and
$\rho(\beta^\star)=\sqrt{\kappa_k}/(\sqrt{\kappa_k}+1)$. Finally,
$\log(1/\rho(\beta^\star))=\log(1+\kappa_k^{-1/2})\ge(\sqrt{\kappa_k}+1)^{-1}$,
so $\rho(\beta^\star)^r\le\varepsilon/\norm{\bt^{(0)}-\bt^\star}$ for
$r\ge(\sqrt{\kappa_k}+1)\log(\norm{\bt^{(0)}-\bt^\star}/\varepsilon)$.

\emph{Step 5 (part (iii)).} From $\norm{\bt^{(r-1)}-\bt^\star}\le
\norm{\bt^{(r-1)}-\bt^{(r)}}+\norm{T\bt^{(r-1)}-T\bt^\star}\le
\norm{\bt^{(r)}-\bt^{(r-1)}}+\rho\norm{\bt^{(r-1)}-\bt^\star}$ we get
$\norm{\bt^{(r-1)}-\bt^\star}\le\norm{\bt^{(r)}-\bt^{(r-1)}}/(1-\rho)$, and
$\norm{\bc^{(r+1)}-\hat\ba_k}\le\norm{\bt^{(r)}-\bt^\star}\le\rho\norm{\bt^{(r-1)}-\bt^\star}$
yields~\eqref{eq:certificate}. Under the stopping rule
$\norm{\bt^{(r)}-\bt^{(r-1)}}\le(1-\rho)\varepsilon_{\rm opt}$ the right-hand
side is at most $\rho\varepsilon_{\rm opt}$.
\end{proof}

\begin{oldrevblock}
\section{Proofs of Lemma~\ref{lem:regressor} and Theorem~\ref{thm:topology_error}}
\label{app_thm3}
\begin{proof}[Proof of Lemma~\ref{lem:regressor}]
$\bU(x)$ is block diagonal with blocks $\bU_i(x)\in\mathbb{R}^{N\times|\calI_i|}$,
so $\norm{\bU(x)}=\max_i\norm{\bU_i(x)}\le\max_i\normF{\bU_i(x)}
=\max_i(\sum_{j\in\calI_i}\norm{\bB_i\bchi(x_i,x_j)}^2)^{1/2}\le\sqrt{\bar p}\max_i\norm{\bB_i}b_\chi$.
For~\eqref{eq:lU}, let $\epsilon_j\triangleq\norm{x_j-x'_j}$ and
$\bepsilon\triangleq[\epsilon_1,\ldots,\epsilon_p]^\top$, so that
$\norm{\bepsilon}=\norm{x-x'}$. The $i$-th block of $(\bU(x)-\bU(x'))\ba$
has norm at most
$\max_m\norm{\bB_m}L_\chi\sum_j|a_{ij}|(\epsilon_i+\epsilon_j)
\le\max_m\norm{\bB_m}L_\chi\bigl(\norm{\bA}_\infty\epsilon_i+[|\bA|\bepsilon]_i\bigr)$.
Taking the Euclidean norm over $i$ and using
$\norm{|\bA|}\le\sqrt{\norm{\bA}_\infty\norm{\bA}_{1\to1}}$ (Schur
test~\cite{horn2012matrix}) gives~\eqref{eq:lU}. Finally,
$\norm{\hD_t}\le\norm{\bH_t}\norm{\hU_{t-1}}\le\bar hb_U$, and the column of
$\hD_t$ associated with $(i,j)$ is $\bH_{t}$ applied to the vector whose
$i$-th block is $\bB_i\bchi(\hat x_{i,t-1},\hat x_{j,t-1})$, whose norm is at
most $\bar h\norm{\bB_i}b_\chi\le b_c$.
\end{proof}

\begin{lemma}[Self-normalized inequality]
\label{lem:selfnormalized}
Let $\omega_t\ge0$, $\bD_t\in\mathbb{R}^{m_t\times p_\calI}$, $\mu_0\ge0$, and
$\bG\triangleq\sum_t\omega_t\bD_t^\top\bD_t+\mu_0\bI\succ\bO$. Then, for all
$\by_t\in\mathbb{R}^{m_t}$,
$\bigl\|\sum_t\omega_t\bD_t^\top\by_t\bigr\|_{\bG^{-1}}\le\bigl(\sum_t\omega_t\norm{\by_t}^2\bigr)^{1/2}$.
\end{lemma}
\begin{proof}
Let $\bq\triangleq\bG^{-1}\sum_t\omega_t\bD_t^\top\by_t$, so that
$\norm{\bq}_{\bG}=\norm{\sum_t\omega_t\bD_t^\top\by_t}_{\bG^{-1}}$. By the
Cauchy--Schwarz inequality,
$\norm{\bq}_{\bG}^2=\sum_t\omega_t\inner{\bD_t\bq}{\by_t}
\le(\sum_t\omega_t\norm{\bD_t\bq}^2)^{1/2}(\sum_t\omega_t\norm{\by_t}^2)^{1/2}$,
and $\sum_t\omega_t\norm{\bD_t\bq}^2=\bq^\top(\bG-\mu_0\bI)\bq\le\norm{\bq}_{\bG}^2$.
Dividing by $\norm{\bq}_{\bG}$ (the case $\bq=\bzero$ is trivial) gives
the claim.
\end{proof}

\begin{proof}[Proof of Theorem~\ref{thm:topology_error}]
\emph{Step 1 (regression residual).} Since $\bPsi_{t-1}=\bPhi(x_{t-1})$,
\eqref{eq:global_lifted_state}--\eqref{eq:global_lifted_meas} give
$z_t=\bH_t\bL\bPhi(x_{t-1})+\bH_t\bU(x_{t-1})\ba_{t-1}+\bH_t\bzeta_{t-1}+v_t$.
With~\eqref{eq:regression_data} and
$\bzeta_{t-1}=\bdelta_{t-1}+\expect[\bxi_{t-1}\mid\calF_{t-1}]+\tilde\bxi_{t-1}$,
the residual $\bm\eta_{k,t}\triangleq\hat\by_t-\hD_t\ba_k$ splits as
$\bm\eta_{k,t}=\by^{\rm e}_t+\by^{\rm b}_{k,t}+\bm\varepsilon_t$ with
\begin{align*}
  \by^{\rm e}_t&\triangleq\bH_t\bL\bigl(\bPhi(x_{t-1})-\bPhi(\hat x_{t-1})\bigr)\nonumber\\
  &\quad+\bH_t\bigl(\bU(x_{t-1})-\hU_{t-1}\bigr)\ba_{t-1},\\
  \by^{\rm b}_{k,t}&\triangleq\bH_t\bigl(\hU_{t-1}(\ba_{t-1}-\ba_k)+\bdelta_{t-1}
  +\expect[\bxi_{t-1}\mid\calF_{t-1}]\bigr),
\end{align*}
so that $\norm{\by^{\rm b}_{k,t}}=\eta^{\rm b}_{k,t}$. Since $\bPhi$ is
$C^1$, $\bPhi(x)-\bPhi(\hat x)=\int_0^1\bJPhi(\hat x+\tau(x-\hat x))\,d\tau\,(x-\hat x)$,
hence
$\norm{\bH_t\bL(\bPhi(x_{t-1})-\bPhi(\hat x_{t-1}))}\le\ell_{HL}\norm{\be_{t-1}}$
by~\eqref{eq:ce}; with~\eqref{eq:lU},
$\norm{\by^{\rm e}_t}\le c_e\norm{\be_{t-1}}$.

\emph{Step 2 (basic inequality).} Let $F$ be as in Appendix~\ref{app_thm2}
and $\bvartheta_k\triangleq-\nabla F(\ba_k)=\sum_t\omega_{k,t}\hD_t^\top\bm\eta_{k,t}-\mu_0\ba_k$.
Optimality of $\hat\ba_k$ gives $-\nabla F(\hat\ba_k)\in\partial\mathcal{R}(\hat\ba_k)$;
since $\ba_k\in\calA$ by (C1), the subgradient inequality yields
$\inner{\nabla F(\hat\ba_k)}{\bDelta_k}\le\mathcal{R}(\ba_k)-\mathcal{R}(\hat\ba_k)
=\lambda_k(\norm{\ba_k}_1-\norm{\hat\ba_k}_1)$. With $S\triangleq\supp(\ba_k)$,
$\norm{\hat\ba_k}_1\ge\norm{\ba_{k,S}}_1-\norm{\bDelta_{k,S}}_1+\norm{\bDelta_{k,S^c}}_1$,
hence
\begin{equation}
  \lambda_k(\norm{\ba_k}_1-\norm{\hat\ba_k}_1)\le\lambda_k\bigl(\norm{\bDelta_{k,S}}_1-\norm{\bDelta_{k,S^c}}_1\bigr).
  \label{eq:cone_ineq}
\end{equation}
Since $F$ is quadratic with Hessian $\hG_{\mu,k}$,
\begin{align}
  \norm{\bDelta_k}_{\hG_{\mu,k}}^2&=\inner{\nabla F(\hat\ba_k)-\nabla F(\ba_k)}{\bDelta_k}\nonumber\\
  &\le\lambda_k\norm{\bDelta_{k,S}}_1+\inner{\bvartheta_k}{\bDelta_k}.
  \label{eq:basic}
\end{align}
By (C2)--(C3), $\lambda_k\norm{\bDelta_{k,S}}_1\le\lambda_k\sqrt s\norm{\bDelta_k}
\le\lambda_k\sqrt s\norm{\bDelta_k}_{\hG_{\mu,k}}/\sqrt\mu$, and
$\inner{\bvartheta_k}{\bDelta_k}\le\norm{\bvartheta_k}_{\hG_{\mu,k}^{-1}}\norm{\bDelta_k}_{\hG_{\mu,k}}$.
Dividing by $\norm{\bDelta_k}_{\hG_{\mu,k}}$ (the case $\bDelta_k=\bzero$
is trivial),
$\norm{\bDelta_k}_{\hG_{\mu,k}}\le\lambda_k\sqrt s/\sqrt\mu+\norm{\bvartheta_k}_{\hG_{\mu,k}^{-1}}$.

\emph{Step 3 (bounding $\bvartheta_k$).} By the triangle inequality and
Lemma~\ref{lem:selfnormalized} applied to $\by^{\rm e}$ and $\by^{\rm b}$,
\[
  \norm{\bvartheta_k}_{\hG_{\mu,k}^{-1}}\le c_e[\norm{\be_{\cdot-1}}]_k
  +[\eta^{\rm b}]_k+\frac{\norm{\bvarsigma_k}}{\sqrt{\mu}}+\frac{\mu_0\norm{\ba_k}}{\sqrt{\mu}},
\]
where we used $\norm{\bv}_{\hG_{\mu,k}^{-1}}\le\norm{\bv}/\sqrt{\mu_k}\le\norm{\bv}/\sqrt\mu$
for the last two terms. With $\norm{\ba_k}\le\bar a$, Step~2
gives~\eqref{eq:topology_bound}; $\norm{\bDelta_k}\le
\norm{\bDelta_k}_{\hG_{\mu,k}}/\sqrt\mu$ follows from (C3).

\emph{Step 4 (martingale term).} $\hD_t$ is $\calF_{t-1}$-measurable
(it depends on $\hat x_{t-1}$ and on the schedule), $\bm\varepsilon_t$ is
$\calF_t$-measurable, and
$\expect[\bm\varepsilon_t\mid\calF_{t-1}]=\bH_t\expect[\tilde\bxi_{t-1}\mid\calF_{t-1}]+\expect[v_t]=\bzero$
because $v_t$ is independent of $\calF_{t-1}$. Hence, for $t'<t$,
$\expect\inner{\hD_{t'}^\top\bm\varepsilon_{t'}}{\hD_t^\top\bm\varepsilon_t}
=\expect\inner{\hD_t\hD_{t'}^\top\bm\varepsilon_{t'}}{\expect[\bm\varepsilon_t\mid\calF_{t-1}]}=0$,
and, since the weights are deterministic,
$\expect\norm{\bvarsigma_k}^2=\sum_t\omega_{k,t}^2\expect\norm{\hD_t^\top\bm\varepsilon_t}^2$.
Moreover, $\tilde\bxi_{t-1}$ is a function of $\calF_{t-1}$ and $w_{t-1}$,
and $v_t$ is independent of both with zero mean, so the cross term
vanishes and
$\expect\norm{\bm\varepsilon_t}^2\le\bar h^2\expect\norm{\tilde\bxi_{t-1}}^2
+\tr(\bR_t)\le\sigma_\varepsilon^2$, using
$\expect\norm{\tilde\bxi_{t-1}}^2\le\expect\norm{\bxi_{t-1}}^2\le L_\phi^2\sigma_w^2$
(Proposition~\ref{prop:lifted_noise}, summed over nodes). With
$\norm{\hD_t}\le b_D$ and~\eqref{eq:weights},
$\expect\norm{\bvarsigma_k}^2\le b_D^2\sigma_\varepsilon^2/n_{{\rm eff},k}$,
which is~\eqref{eq:martingale_bound}.

\emph{Step 5 (bias term).} By (C2), $\norm{\ba_{t-1}-\ba_k}\le(k-t+1)\nu$,
so $\eta^{\rm b}_{k,t}\le\bar hb_U(k-t+1)\nu+\norm{\bH_t(\bdelta_{t-1}+\expect[\bxi_{t-1}\mid\calF_{t-1}])}$.
Minkowski's inequality in $L^2$ of the product of the probability space
and the window measure $\{\omega_{k,t}\}$ gives
$\Lnorm{[\eta^{\rm b}]_k}\le\bar hb_U\nu\ell^{(2)}_k+\beta_{\rm rep}$. By
Proposition~\ref{prop:lifted_noise}, $\norm{\expect[\bxi_{t-1}\mid\calF_{t-1}]}
\le\frac{L_2}{2}(\sum_i\tr(\bQ_{i,t-1})^2)^{1/2}\le\frac{L_2}{2}\sigma_w^2$, which
gives the bound on $\beta_{\rm rep}$ in~\eqref{eq:bias_bound}. For
$\bH_t=\bC_t\Pip$, $\bH_t\expect[\bxi_{t-1}\mid\calF_{t-1}]=\bC_t\expect[\Pip\bxi_{t-1}\mid\calF_{t-1}]
=\bC_t\expect[w_{t-1}]=\bzero$ by Proposition~\ref{prop:lifted_noise} and
the independence of $w_{t-1}$ from $\calF_{t-1}$, and
$\norm{\bH_t\bdelta_{t-1}}\le\bar c\norm{\Pip\bdelta_{t-1}}$; moreover
$\bH_t\tilde\bxi_{t-1}=\bC_tw_{t-1}$. This gives~\eqref{eq:linear_sensor},
and $\varepsilon_x=0$ under the stated conditions on the matrices
by Theorem~\ref{thm:lifting}(ii).
\end{proof}

\section{Proofs of Lemma~\ref{lem:subgauss} and Proposition~\ref{prop:sparse}}
\label{app_sparse}
\begin{proof}[Proof of Lemma~\ref{lem:subgauss}]
\emph{Step 1 (conditional sub-Gaussianity).} Let $\bu$ be
$\calF_{t-1}$-measurable. Conditionally on $\calF_{t-1}$, the noise-free
successor $\tilde x_t$ of $x_{t-1}$ is fixed, and $w_{t-1}\sim\mathcal{N}(\bzero,\bQ_{t-1})$
and $v_t\sim\mathcal{N}(\bzero,\bR_t)$ are independent of each other. Since
$\bxi_{t-1}=\bPhi(\tilde x_t+w_{t-1})-\bPhi(\tilde x_t)$, one has
$\bu^\top\bH_t\tilde\bxi_{t-1}=G(\bm\gamma)-\expect[G(\bm\gamma)\mid\calF_{t-1}]$ with
$G(\bm\gamma)\triangleq\bu^\top\bH_t\bPhi(\tilde x_t+\bQ_{t-1}^{1/2}\bm\gamma)$
and $\bm\gamma$ standard Gaussian. $G$ is Lipschitz with constant
$\norm{\bH_t^\top\bu}L_\phi\norm{\bQ_{t-1}^{1/2}}\le\norm{\bu}\bar hL_\phi\sqrt{\bar q}$,
so the Gaussian concentration inequality for Lipschitz
functions~\cite[Ch.~5]{boucheron2013concentration} gives
$\expect[e^{\bu^\top\bH_t\tilde\bxi_{t-1}}\mid\calF_{t-1}]\le e^{\norm{\bu}^2\bar h^2L_\phi^2\bar q/2}$.
Also $\expect[e^{\bu^\top v_t}\mid\calF_{t-1}]=e^{\bu^\top\bR_t\bu/2}\le e^{\norm{\bu}^2\bar r/2}$.
By conditional independence,
$\expect[e^{\bu^\top\bm\varepsilon_t}\mid\calF_{t-1}]\le e^{\norm{\bu}^2\sigma_{\rm sg}^2/2}$.

\emph{Step 2 (one coordinate).} Fix $m\in\calI$, let $\bd_{t,m}$ be the
$m$-th column of $\hD_t$ ($\calF_{t-1}$-measurable, $\norm{\bd_{t,m}}\le b_c$
by Lemma~\ref{lem:regressor}), and let $s\in\mathbb{R}$. Order
$\calT_k=\{t_1<\cdots<t_n\}$. Since $\bm\varepsilon_{t'}$ is
$\calF_{t_n-1}$-measurable for $t'<t_n$, conditioning on $\calF_{t_n-1}$
and Step~1 with $\bu=s\omega_{k,t_n}\bd_{t_n,m}$ give
$\expect e^{s[\bvarsigma_k]_m}\le e^{s^2\omega_{k,t_n}^2b_c^2\sigma_{\rm sg}^2/2}\,
\expect e^{s\sum_{l<n}\omega_{k,t_l}\bd_{t_l,m}^\top\bm\varepsilon_{t_l}}$. By
induction,
$\expect e^{s[\bvarsigma_k]_m}\le\exp(s^2b_c^2\sigma_{\rm sg}^2/(2n_{{\rm eff},k}))$.
The Chernoff bound with $s=\tau n_{{\rm eff},k}/(b_c^2\sigma_{\rm sg}^2)$ gives
$\prob(|[\bvarsigma_k]_m|>\tau)\le2\exp(-\tau^2n_{{\rm eff},k}/(2b_c^2\sigma_{\rm sg}^2))$.

\emph{Step 3 (union bound).} Summing over the $p_\calI$ coordinates and
setting the right-hand side equal to $\delta$ gives~\eqref{eq:tau}.
\end{proof}

\begin{proof}[Proof of Proposition~\ref{prop:sparse}]
The vector $\bvartheta_k$ is the one of Appendix~\ref{app_thm3}, Step~2.
By~\eqref{eq:basic}, \eqref{eq:cone_ineq} and H\"older's inequality,
\[
  \norm{\bDelta_k}_{\hG_{\mu,k}}^2\le\lambda_k\bigl(\norm{\bDelta_{k,S}}_1-\norm{\bDelta_{k,S^c}}_1\bigr)
  +\norm{\bvartheta_k}_\infty\norm{\bDelta_k}_1 .
\]
(Here~\eqref{eq:cone_ineq} is used in place of its upper bound
$\lambda_k\norm{\bDelta_{k,S}}_1$.) With
$\norm{\bvartheta_k}_\infty\le\lambda_k/2$ and
$\norm{\bDelta_k}_1=\norm{\bDelta_{k,S}}_1+\norm{\bDelta_{k,S^c}}_1$,
\[
  0\le\norm{\bDelta_k}_{\hG_{\mu,k}}^2\le\tfrac{3\lambda_k}{2}\norm{\bDelta_{k,S}}_1
  -\tfrac{\lambda_k}{2}\norm{\bDelta_{k,S^c}}_1,
\]
so $\norm{\bDelta_{k,S^c}}_1\le3\norm{\bDelta_{k,S}}_1$, i.e.,
$\bDelta_k\in\calC(S_k)$. Then~\eqref{eq:RE} and
$\norm{\bDelta_{k,S}}_1\le\sqrt s\norm{\bDelta_k}$ give
$\kappa_{\rm RE}\norm{\bDelta_k}^2\le\frac{3\lambda_k}{2}\sqrt s\norm{\bDelta_k}$,
hence the first bound in~\eqref{eq:RE_bound}; the second follows from
$\norm{\bDelta_k}_1\le4\norm{\bDelta_{k,S}}_1\le4\sqrt s\norm{\bDelta_k}$.
For the bound on $\norm{\bvartheta_k}_\infty$, write
$\bvartheta_k=\sum_t\omega_{k,t}\hD_t^\top(\by^{\rm e}_t+\by^{\rm b}_{k,t})+\bvarsigma_k-\mu_0\ba_k$.
For every coordinate $m$,
$|\sum_t\omega_{k,t}\bd_{t,m}^\top\by_t|\le b_c\sum_t\omega_{k,t}\norm{\by_t}\le b_c[\norm{\by}]_k$
by the Cauchy--Schwarz inequality and $\sum_t\omega_{k,t}=1$; with
$\norm{\by^{\rm e}_t}\le c_e\norm{\be_{t-1}}$ and $\norm{\by^{\rm b}_{k,t}}=\eta^{\rm b}_{k,t}$
the claimed bound follows. Combining it with Lemma~\ref{lem:subgauss}
shows that~\eqref{eq:lambda_choice} implies
$\lambda_k\ge2\norm{\bvartheta_k}_\infty$ with probability at least
$1-\delta$.
\end{proof}

\section{Proof of Theorem~\ref{thm:error_bound}}
\label{app_thm4}
\begin{lemma}[Vanishing initial window]
\label{lem:shift}
For each of the weight choices $W<\infty$; $W=\infty$, $\gamma<1$; and
$W=\infty$, $\gamma=1$, and every fixed $T$,
$\sum_{t\in\calT_k,\,t\le T}\omega_{k,t}\to0$ as $k\to\infty$. Consequently,
for every bounded nonnegative sequence $\{X_t\}$,
$\limsup_k(\sum_{t\in\calT_k}\omega_{k,t}X_{t-1}^2)^{1/2}\le\limsup_tX_t$.
\end{lemma}
\begin{proof}
If $W<\infty$, $\calT_k\cap\{1,\ldots,T\}=\emptyset$ for $k\ge T+W$. If
$W=\infty$ and $\gamma<1$,
$\omega_{k,t}=(1-\gamma)\gamma^{k-t}/(1-\gamma^k)$, so the sum is at most
$T\gamma^{k-T}(1-\gamma)/(1-\gamma^k)\to0$. If $W=\infty$ and $\gamma=1$,
$\omega_{k,t}=1/k$ and the sum is at most $T/k$. For the second claim,
let $\bar X\triangleq\limsup_tX_t$, $B\triangleq\sup_tX_t$ and $\epsilon>0$,
and choose $T$ such that $X_{t-1}\le\bar X+\epsilon$ for $t>T$. Then
$\sum_{t\in\calT_k}\omega_{k,t}X_{t-1}^2\le B^2\sum_{t\le T}\omega_{k,t}+(\bar X+\epsilon)^2$,
whose $\limsup$ is at most $(\bar X+\epsilon)^2$; let $\epsilon\to0$.
\end{proof}

\begin{proof}[Proof of Theorem~\ref{thm:error_bound}]
\emph{Step 1 (state recursion).} Let $t\ge k_0+1$. By~\eqref{eq:kf_update_state}
and $x_t=\Pip\bPsi_t$,
$\be_t=\Pip\bigl[(\bI-\bK_t\bH_t)(\bPsi_t-\hat\bPsi_{t|t-1})-\bK_tv_t\bigr]$, and
by~\eqref{eq:global_lifted_state} and~\eqref{eq:kf_pred_state} (the filter
uses $\hat\ba^{\rm alg}_{t-1}$),
\begin{align*}
  &\bPsi_t-\hat\bPsi_{t|t-1}=\bL\bigl(\bPhi(x_{t-1})-\bPhi(\hat x_{t-1})\bigr)+\bzeta_{t-1}
  \nonumber\\
  &\qquad+\hU_{t-1}(\ba_{t-1}-\hat\ba^{\rm alg}_{t-1})+\bigl(\bU(x_{t-1})-\hU_{t-1}\bigr)\ba_{t-1}.
\end{align*}
The map $y\mapsto\Pip(\bI-\bK_t\bH_t)\bL\bPhi(y)$ is $C^1$ with derivative
$\Pip(\bI-\bK_t\bH_t)\bL\bJPhi(y)$, so the mean value inequality and (C4)
bound the first contribution by $\rho_L\norm{\be_{t-1}}$; by (C4)
and~\eqref{eq:lU} the second is at most $\kappa_K\ell_U\norm{\be_{t-1}}$.
Writing $\ba_{t-1}-\hat\ba^{\rm alg}_{t-1}=-\bDelta_{t-1}-(\hat\ba^{\rm alg}_{t-1}-\hat\ba_{t-1})$,
the third is at most
$\norm{\Pip(\bI-\bK_t\bH_t)\hU_{t-1}\hG_{\mu,t-1}^{-1/2}}\norm{\bDelta_{t-1}}_{\hG_{\mu,t-1}}
+\kappa_Kb_U\varepsilon_{\rm opt}\le\kappa_U\norm{\bDelta_{t-1}}_{\hG_{\mu,t-1}}+\kappa_Kb_U\varepsilon_{\rm opt}$
by (C4)--(C5). The remaining term
$n_t\triangleq\norm{\Pip(\bI-\bK_t\bH_t)\bzeta_{t-1}-\Pip\bK_tv_t}$ satisfies
$\Lnorm{n_t}\le\kappa_K(\varepsilon_{\rm rep}+L_\phi\sigma_w)+\kappa_v\sigma_v$,
because $\Lnorm{\bzeta_{t-1}}\le\Lnorm{\bdelta_{t-1}}+\Lnorm{\bxi_{t-1}}$ and
$\expect\norm{\bxi_{t-1}}^2\le L_\phi^2\sigma_w^2$ by~\eqref{eq:xi_var}.
Taking $L^2$ norms and using Minkowski's inequality,
\begin{equation}
  X_t\le\rho_xX_{t-1}+\kappa_UY_{t-1}+f_x,\qquad t\ge k_0+1.
  \label{eq:X_rec}
\end{equation}

\emph{Step 2 (topology recursion).} Taking $L^2$ norms
in~\eqref{eq:topology_bound}, with deterministic weights and
\eqref{eq:martingale_bound}--\eqref{eq:bias_bound},
\begin{equation}
  Y_k\le c_e\Bigl(\sum_{t\in\calT_k}\omega_{k,t}X_{t-1}^2\Bigr)^{1/2}+f_{y,k},\qquad k\ge k_0,
  \label{eq:Y_rec}
\end{equation}
where $f_{y,k}$ is~\eqref{eq:fy} with $\bar\lambda$, $\bar\ell^{(2)}$ and
$\bar n_{\rm eff}$ replaced by $\lambda_k$, $\ell^{(2)}_k$ and $n_{{\rm eff},k}$.

\emph{Step 3 (boundedness).} $\bar F_y\triangleq\sup_{k\ge k_0}f_{y,k}<\infty$,
because $\sup_k\lambda_k<\infty$, $n_{{\rm eff},k}\ge1$, and
$\sup_k\ell^{(2)}_k<\infty$ when $W<\infty$ or $\gamma<1$ (the drift term
is absent when $\nu=0$). Let $B_k\triangleq\max_{0\le t\le k}X_t$, which is
finite for $k\le k_0$ by (C1). For $t\ge k_0+1$,~\eqref{eq:Y_rec} gives
$Y_{t-1}\le c_eB_{t-1}+\bar F_y$, and~\eqref{eq:X_rec} gives
$X_t\le\rho_{\rm cl}B_{t-1}+G$ with $G\triangleq f_x+\kappa_U\bar F_y$. If
$B_{t-1}\le M\triangleq\max\{B_{k_0},G/(1-\rho_{\rm cl})\}$, then
$X_t\le\rho_{\rm cl}M+(1-\rho_{\rm cl})M=M$, hence $B_t\le M$. By induction,
$\sup_kX_k\le M$ and $\sup_{k\ge k_0}Y_k\le c_eM+\bar F_y$.

\emph{Step 4 (ultimate bound).} Let $\bar X\triangleq\limsup_kX_k$ and
$\bar Y\triangleq\limsup_kY_k$, both finite by Step~3. By
Lemma~\ref{lem:shift} and \rev{$\limsup_kf_{y,k}\le f_y$, which follows
because the $\limsup$ of a sum is at most the sum of the $\limsup$'s
and $\limsup_k\lambda_k=\bar\lambda$, $\limsup_k\ell^{(2)}_k=\bar\ell^{(2)}$,
$\limsup_kn_{{\rm eff},k}^{-1/2}=\bar n_{\rm eff}^{-1/2}$},~\eqref{eq:Y_rec} gives
$\bar Y\le c_e\bar X+f_y$, and~\eqref{eq:X_rec} gives
$\bar X\le\rho_x\bar X+\kappa_U\bar Y+f_x$. Substituting the former into
the latter, $(1-\rho_{\rm cl})\bar X\le f_x+\kappa_Uf_y$, and then
$\bar Y\le c_e\bar X+f_y\le(c_ef_x+(1-\rho_x)f_y)/(1-\rho_{\rm cl})$. These
two bounds are the components of $(\bI-\bM)^{-1}[f_x,f_y]^\top$, since
$\det(\bI-\bM)=1-\rho_x-\kappa_Uc_e$ and
$(\bI-\bM)^{-1}=\det(\bI-\bM)^{-1}\bigl[\begin{smallmatrix}1&\kappa_U\\c_e&1-\rho_x\end{smallmatrix}\bigr]$.

\emph{Step 5 (equivalence in~\eqref{eq:small_gain}).} The eigenvalues
of $\bM$ solve $\lambda^2-\rho_x\lambda-\kappa_Uc_e=0$; the spectral radius
is $\frac12(\rho_x+\sqrt{\rho_x^2+4\kappa_Uc_e})$, which is below one iff
$\sqrt{\rho_x^2+4\kappa_Uc_e}<2-\rho_x$, i.e., iff $\rho_x+\kappa_Uc_e<1$.

\emph{Step 6 (topology in the Euclidean norm).}
$\Lnorm{\hat\ba^{\rm alg}_k-\ba_k}\le\Lnorm{\bDelta_k}+\varepsilon_{\rm opt}
\le Y_k/\sqrt\mu+\varepsilon_{\rm opt}$ by (C3) and (C5).
\end{proof}
\end{oldrevblock}

\begin{revblock}
\section{Proofs of Propositions~\ref{prop:calib} and~\ref{prop:onestep}}
\label{app_calib}
\begin{proof}[Proof of Proposition~\ref{prop:calib}]
(i) For every $i$, $k$ and $c\neq0$,
$(\bB_i/c)\sum_j(ca_{ij})\bchi(x_{i,k},x_{j,k})=\bB_i\sum_ja_{ij}\bchi(x_{i,k},x_{j,k})$,
so every summand of $\mathcal{J}$ is unchanged, and
$\bPi(\bB_i/c)=\bB_x/c$ for all $i$, with $\bB_x/c\neq\bO$. Let
$(\{\bF_i,\bB_i\},\bA)$ be a minimizer of the penalized objective with
$\bA\neq\bO$, and set $\alpha\triangleq\sum_{i\neq j}|a_{ij}|>0$ and
$\beta_B\triangleq\sum_i\normF{\bB_i}^2$. Along the feasible curve
$c\mapsto(\{\bF_i,\bB_i/c\},c\bA)$, $c>0$, the objective equals
$\varphi(c)\triangleq C_0+\varrho\beta_Bc^{-2}+\lambda\alpha c$, where $C_0$
collects the terms that do not depend on $c$; the curve is
admissible because the admissible set of $\bA$ is a cone. Since $c=1$ minimizes
$\varphi$ on $(0,\infty)$ and $\varphi$ is differentiable there,
$\varphi'(1)=-2\varrho\beta_B+\lambda\alpha=0$, which
is~\eqref{eq:scale_balance}.

(ii) By (A1), $\bPi\bphi(x_{i,k+1})=x_{i,k+1}$, and by (A4) with $\bB_x$
fixed, the state rows of the residual in~\eqref{eq:calibration} are
$x_{i,k+1}-\bTheta\bphi(x_{i,k})-\bB_x\sum_{j\in\calI_i}a_j\bchi(x_{i,k},x_{j,k})$
with $\bTheta=\bPi\bF_i$. By hypothesis this equals
$\mathcal{L}_i(\bPi\bF^o_i,\ba^o_i)-\mathcal{L}_i(\bTheta,\ba)$, because
$\mathcal{L}_i$ is linear. It vanishes for all $k$ iff
$(\bTheta,\ba)-(\bPi\bF^o_i,\ba^o_i)\in\ker\mathcal{L}_i$. Hence the
zero-residual pair is unique iff $\ker\mathcal{L}_i=\{\bzero\}$. If the
sequence associated with some $j$ lies in the span stated in the
proposition, there is $(\bTheta',\ba')$ with $a'_j=1$ and
$\mathcal{L}_i(\bTheta',\ba')=\bzero$, so $\ker\mathcal{L}_i\neq\{\bzero\}$.
\end{proof}

\begin{proof}[Proof of Proposition~\ref{prop:onestep}]
Let $\bK^x_t\triangleq\Pip\bK_t$. Since $\bH_t=\bC_t\Pip$,
$\Pip(\bI-\bK_t\bH_t)=\Pip-\bK^x_t\bC_t\Pip=(\bI-\bK^x_t\bC_t)\Pip$. Hence
$\Pip(\bI-\bK_t\bH_t)\bL\bJPhi(\bm\xi)\bu=(\bI-\bK^x_t\bC_t)\bu=\bu-\bK^x_t\bC_t\bu=\bu$,
and the induced norm is at least one. By (C4),
$\rho_L\ge\norm{\Pip(\bI-\bK_t\bH_t)\bL\bJPhi(\bm\xi)}\ge1$, so
$\rho_x=\rho_L+\kappa_K\ell_U\ge1$ and $\rho_{\rm cl}\ge\rho_x\ge1$. For the
Kuramoto lifting, the $i$-th block of $\Pip\bL\bPhi(x)$ is
$\bPi\bF_i\bphi(\theta_i)=\theta_i+\omega_i\Delta t$
(Example~\ref{ex:constructions}), whose derivative is one, so
$\Pip\bL\bJPhi\equiv\bI_p$; if node $i$ is not measured at time $t$, the
$i$-th unit vector lies in $\ker\bC_t$.
\end{proof}

\section{Proofs of Theorem~\ref{thm:multistep} and Corollary~\ref{cor:roundrobin}}
\label{app_multistep}
\begin{lemma}[Products over convex hulls]
\label{lem:hull}
Let $\calS_1,\ldots,\calS_r$ be bounded sets of matrices of compatible
dimensions and let $\overline{\rm co}\,\calS_j$ denote the closed convex
hull. Then
\begin{align*}
    &\sup\Bigl\{\norm{\bS_r\cdots\bS_1}:\bS_j\in\overline{\rm co}\,\calS_j\Bigr\}\nonumber\\
  &\qquad=\sup\Bigl\{\norm{\bS_r\cdots\bS_1}:\bS_j\in\calS_j\Bigr\}.
\end{align*}
  
\end{lemma}
\begin{proof}
Denote by $\sigma_l$ the supremum in which $\bS_j\in\calS_j$ for $j\le l$
and $\bS_j\in\overline{\rm co}\,\calS_j$ for $j>l$; then $\sigma_0$ and
$\sigma_r$ are the left- and right-hand sides, $\sigma_l<\infty$ by
boundedness, and $\sigma_{l}\le\sigma_{l-1}$ because the constraint set
shrinks. Fix $l\ge1$ and all factors except $\bS_l$, with
$\bS_j\in\calS_j$ for $j<l$ and $\bS_j\in\overline{\rm co}\,\calS_j$ for
$j>l$. The map $\bS_l\mapsto\norm{\bS_r\cdots\bS_l\cdots\bS_1}$ is convex
and continuous, because it is a norm composed with a linear map. A
convex function on a convex combination
$\sum_q\theta_q\bS_q$, $\bS_q\in\calS_l$, is at most
$\max_q$ of its values at $\bS_q$, and continuity extends this bound to
the closure. Hence its supremum over $\overline{\rm co}\,\calS_l$ is at
most its supremum over $\calS_l$, which is at most $\sigma_l$. Taking the
supremum over the fixed factors gives $\sigma_{l-1}\le\sigma_l$. Thus
$\sigma_0=\sigma_1=\cdots=\sigma_r$.
\end{proof}

\begin{proof}[Proof of Theorem~\ref{thm:multistep}]
\emph{Step 1 (error recursion).} As in Appendix~\ref{app_thm4}, Step~1,
for $t\ge k_0+1$,
$\be_t=\Pip(\bI-\bK_t\bH_t)\bigl[\bL(\bPhi(x_{t-1})-\bPhi(\hat x_{t-1}))
+(\bU(x_{t-1})-\hU_{t-1})\ba_{t-1}\bigr]+\bm u_t+\bm n_t$, with
$\bm u_t\triangleq\Pip(\bI-\bK_t\bH_t)\hU_{t-1}(\ba_{t-1}-\hat\ba^{\rm alg}_{t-1})$
and $\bm n_t\triangleq\Pip(\bI-\bK_t\bH_t)\bzeta_{t-1}-\Pip\bK_tv_t$. Since
$x\mapsto\bL\bPhi(x)+\bU(x)\ba_{t-1}$ is $C^1$ with derivative
$\bN(x;\ba_{t-1})$, the fundamental theorem of calculus along
$\bm\xi_\tau\triangleq\hat x_{t-1}+\tau\be_{t-1}$ gives
$\be_t=\bA_t\be_{t-1}+\bm u_t+\bm n_t$ with
\[
  \bA_t\triangleq\Pip(\bI-\bK_t\bH_t)\int_0^1\bN(\bm\xi_\tau;\ba_{t-1})\,d\tau .
\]
The integrand is continuous in $\tau$ by (A2) and $\bchi\in C^1$, so the
integral is a limit of Riemann sums, which are convex combinations of
values $\bN(\bm\xi_\tau;\ba_{t-1})$; since $\ba_{t-1}\in\calA_\star$, it
follows that $\bA_t\in\overline{\rm co}\,\calS_t$. As in
Appendix~\ref{app_thm4}, (C4)--(C5) give, for $t\ge k_0+1$,
\begin{align}
  \norm{\bm u_t}&\le\kappa_U\norm{\bDelta_{t-1}}_{\hG_{\mu,t-1}}+\kappa_Kb_U\varepsilon_{\rm opt},
  \nonumber\\
  \Lnorm{\bm n_t}&\le\kappa_K(\varepsilon_{\rm rep}+L_\phi\sigma_w)+\kappa_v\sigma_v.
  \label{eq:un_bounds}
\end{align}

\emph{Step 2 ($m$-step comparison inequality).} Let
$\bm\Phi(k,s)\triangleq\bA_k\bA_{k-1}\cdots\bA_{s+1}$ for $s<k$ and
$\bm\Phi(k,k)\triangleq\bI$. Unrolling the recursion of Step~1, for
$k\ge k_0+m$,
\[
  \be_k=\bm\Phi(k,k-m)\be_{k-m}+\sum_{s=k-m+1}^{k}\bm\Phi(k,s)(\bm u_s+\bm n_s).
\]
By Lemma~\ref{lem:hull} and (C4$_m$)(a),
$\norm{\bm\Phi(k,k-m)}\le\rho_m$ almost surely; by Lemma~\ref{lem:hull}
and (C4$_m$)(b), $\norm{\bm\Phi(k,s)}\le\varpi$ for $k-m+1\le s\le k$
(for $s=k$ trivially, since $\varpi\ge1$). Taking $L^2$ norms, using
Minkowski's inequality, $\norm{\bm\Phi(k,s)\bm y}\le\norm{\bm\Phi(k,s)}\norm{\bm y}$
and~\eqref{eq:un_bounds},
\begin{equation}
  X_k\le\rho_mX_{k-m}+\varpi\sum_{s=k-m+1}^{k}\bigl(\kappa_UY_{s-1}+f_x\bigr),
  \, k\ge k_0+m,
  \label{eq:X_rec_m}
\end{equation}
with $f_x$ as in~\eqref{eq:fx}. For $k_0<k<k_0+m$ (only if $m\ge2$),
(C4$_m$)(b) with $r=1$ gives $\norm{\bA_k}\le\varpi$ and hence
$X_k\le\varpi X_{k-1}+\kappa_UY_{k-1}+f_x$.

\emph{Step 3 (boundedness).} The topology recursion~\eqref{eq:Y_rec} of
Appendix~\ref{app_thm4} holds unchanged, since its derivation does not
involve (C4). Let $B_k\triangleq\max_{0\le t\le k}X_t$ and
$\bar F_y\triangleq\sup_{k\ge k_0}f_{y,k}<\infty$. For $s-1\ge k_0$,
\eqref{eq:Y_rec} gives $Y_{s-1}\le c_eB_{s-2}+\bar F_y$. By (C1) and the
second inequality of Step~2, $B_{k_0+m-1}<\infty$. For $k\ge k_0+m$,
every index in~\eqref{eq:X_rec_m} satisfies $s-1\ge k_0$ and
$s-2\le k-1$, and $X_{k-m}\le B_{k-1}$; hence
\[
  X_k\le\rho^{(m)}_{\rm cl}B_{k-1}+m\varpi G_m,\qquad
  G_m\triangleq f_x+\kappa_U\bar F_y .
\]
With $M_m\triangleq\max\{B_{k_0+m-1},m\varpi G_m/(1-\rho^{(m)}_{\rm cl})\}$,
the induction of Appendix~\ref{app_thm4}, Step~3, gives
$\sup_kX_k\le M_m$ and $\sup_{k\ge k_0}Y_k\le c_eM_m+\bar F_y$.

\emph{Step 4 (ultimate bound).} Let $\bar X\triangleq\limsup_kX_k$ and
$\bar Y\triangleq\limsup_kY_k$, finite by Step~3. Taking $\limsup$
in~\eqref{eq:X_rec_m}, and using that the $\limsup$ of a finite sum is
at most the sum of the $\limsup$'s, $\bar X\le\rho_m\bar X+m\varpi(\kappa_U\bar Y+f_x)$.
As in Appendix~\ref{app_thm4}, Step~4, $\bar Y\le c_e\bar X+f_y$.
Substituting, $(1-\rho^{(m)}_{\rm cl})\bar X\le m\varpi(f_x+\kappa_Uf_y)$
and $\bar Y\le c_e\bar X+f_y\le\bigl(m\varpi c_ef_x+(1-\rho_m)f_y\bigr)/(1-\rho^{(m)}_{\rm cl})$.
These are the components of $(\bI-\bM_m)^{-1}[m\varpi f_x,f_y]^\top$,
since $\det(\bI-\bM_m)=1-\rho^{(m)}_{\rm cl}$ and
$(\bI-\bM_m)^{-1}=\det(\bI-\bM_m)^{-1}\bigl[\begin{smallmatrix}1&m\varpi\kappa_U\\c_e&1-\rho_m\end{smallmatrix}\bigr]$.
The equivalence in~\eqref{eq:small_gain_m} follows as in
Appendix~\ref{app_thm4}, Step~5, with $(\rho_x,\kappa_U)$ replaced by
$(\rho_m,m\varpi\kappa_U)$, using $\rho_m<1$.

\emph{Step 5 (the case $m=1$).} For $m=1$, (C4$_m$)(b) is void, so
$\varpi=1$ is admissible. For $\bS=\Pip(\bI-\bK_t\bH_t)\bN(\bm\xi;\ba)\in\calS_t$,
(C4) and~\eqref{eq:lU} give
$\norm{\bS}\le\norm{\Pip(\bI-\bK_t\bH_t)\bL\bJPhi(\bm\xi)}
+\norm{\Pip(\bI-\bK_t\bH_t)}\,\norm{\partial_{\bm\xi}[\bU(\bm\xi)\ba]}
\le\rho_L+\kappa_K\ell_U=\rho_x$, because the norm of the derivative of a
map is bounded by its Lipschitz constant $\ell_U(\bA)\le\ell_U$. Hence
(C4$_1$)(a) holds with $\rho_1\le\rho_x$, and $\bM_1$ coincides with $\bM$
of~\eqref{eq:loop_matrix} with $\rho_x$ replaced by $\rho_1$.
\end{proof}

\begin{proof}[Proof of Corollary~\ref{cor:roundrobin}]
Let $\bE_t\triangleq\bC_t^\top\bC_t=\diag(\mathbf{1}\{i\in\calV^{\rm obs}_t\})$.
Then $\Pip\bK_t\bH_t=k\bC_t^\top\bC_t\Pip=k\bE_t\Pip$, so
$\Pip(\bI-\bK_t\bH_t)=(\bI_p-k\bE_t)\Pip$ and, since $\Pip\Pip^\top=\bI_p$,
$\kappa_K=\sup_t\norm{\bI_p-k\bE_t}\le1$. By
Proposition~\ref{prop:onestep}, $\Pip\bL\bJPhi\equiv\bI_p$, hence every
element of $\calS_t$ has the form $\bP_t(\bI_p+\bJ)$ with
$\bP_t\triangleq\bI_p-k\bE_t$ and $\norm{\bJ}\le\varepsilon_c$. For
$\bS_t=\bP_t(\bI_p+\bJ_t)=\bP_t+\bP_t\bJ_t$, expanding a product of $r$
such factors gives $\prod_t\bP_t$ plus $2^r-1$ products each containing
$j\ge1$ factors $\bP_t\bJ_t$, of norm at most $\varepsilon_c$, and $r-j$
factors $\bP_t$, of norm at most one. Therefore
\[
  \Bigl\|\prod_{t}\bS_t\Bigr\|\le\Bigl\|\prod_t\bP_t\Bigr\|
  +\sum_{j=1}^{r}\binom{r}{j}\varepsilon_c^j
  =\Bigl\|\prod_t\bP_t\Bigr\|+(1+\varepsilon_c)^r-1 .
\]
For $r\le m-1$, $\norm{\prod_t\bP_t}\le1$ gives the bound $(1+\varepsilon_c)^{m-1}$
on $\varpi$. For $r=m$, $\prod_t\bP_t$ is diagonal with $i$-th entry
$(1-k)^{c_i}$, where $c_i\ge1$ is the number of times node $i$ is
measured in the $m$ steps; hence $\norm{\prod_t\bP_t}\le1-k$, which gives
the bound on $\rho_m$. The final statement is
Proposition~\ref{prop:onestep}.
\end{proof}
\end{revblock}

\section{Proofs of Corollaries~\ref{cor:tracking} and~\ref{cor:consistency}}
\label{app_cor1}
\begin{proof}[Proof of Corollary~\ref{cor:tracking}]
For $W=\infty$, $\omega_{k,t}=(1-\gamma)\gamma^{k-t}/(1-\gamma^k)$. Hence
$n_{{\rm eff},k}=(\sum_{s=0}^{k-1}\gamma^s)^2/\sum_{s=0}^{k-1}\gamma^{2s}\to(1+\gamma)/(1-\gamma)$,
and $(\ell^{(2)}_k)^2\to\sum_{s\ge0}(1-\gamma)\gamma^s(s+1)^2=\expect(S+1)^2$
for $S$ geometric with $\prob(S=s)=(1-\gamma)\gamma^s$. With
$\expect S=\gamma/(1-\gamma)$ and $\expect S^2=\gamma(1+\gamma)/(1-\gamma)^2$,
$\expect(S+1)^2=(1+\gamma)/(1-\gamma)^2$. \oldrev{With $\lambda_k=\lambda_0n_{{\rm eff},k}^{-1/2}$,
$\bar\lambda=\lambda_0\bar n_{\rm eff}^{-1/2}$, so the $\gamma$-dependent part of
$f_y$ is $\frac{\lambda_0\sqrt s+b_D\sigma_\varepsilon}{\sqrt\mu}\sqrt{\frac{1-\gamma}{1+\gamma}}
+\bar hb_U\nu\frac{\sqrt{1+\gamma}}{1-\gamma}$, and $f_x$ does not depend on $\gamma$.
Substitution into~\eqref{eq:error_bound}} gives~\eqref{eq:tradeoff}.
(i) For $\nu=0$, $\sqrt{(1-\gamma)/(1+\gamma)}$ is decreasing in $\gamma$;
for $\gamma=1$, $W=\infty$, $n_{{\rm eff},k}=k\to\infty$.
(ii) Using $\sqrt{(1-\gamma)/(1+\gamma)}\le\sqrt{1-\gamma}$ and
$\sqrt{1+\gamma}\le\sqrt2$, $h(\gamma)\le\tilde h(u)\triangleq c_s\sqrt u+\sqrt2c_\nu\nu/u$
with $u=1-\gamma$. \rev{$\tilde h$ is not convex on $(0,\infty)$ in
general, since
$\tilde h''(u)=-\frac{c_s}{4}u^{-3/2}+2\sqrt2c_\nu\nu\,u^{-3}$ changes sign;
we use its first derivative instead. For $u>0$,
$\tilde h'(u)=\frac{c_s}{2\sqrt u}-\frac{\sqrt2c_\nu\nu}{u^2}
=u^{-2}\bigl(\tfrac{c_s}{2}u^{3/2}-\sqrt2c_\nu\nu\bigr)$, and the bracket is
continuous and strictly increasing in $u$, negative as $u\downarrow0$,
with the unique zero $u^\star=(2\sqrt2c_\nu\nu/c_s)^{2/3}$. Hence
$\tilde h'<0$ on $(0,u^\star)$ and $\tilde h'>0$ on $(u^\star,\infty)$, so
$u^\star$ is the global minimizer of $\tilde h$ on $(0,\infty)$ and, when
$u^\star<1$, on $(0,1)$; when $u^\star\ge1$, $\tilde h$ is decreasing on
$(0,1)$ and $\inf_{u\in(0,1)}\tilde h(u)=\lim_{u\uparrow1}\tilde h(u)$.} At
$u^\star$, $\sqrt2c_\nu\nu/u^\star=\frac12c_s\sqrt{u^\star}$, so
$\tilde h(u^\star)=\frac32c_s(u^\star)^{1/2}=\frac{3}{\sqrt2}c_s^{2/3}(c_\nu\nu)^{1/3}$.
\end{proof}

\begin{proof}[Proof of Corollary~\ref{cor:consistency}]
(i) By~\eqref{eq:state_block}, $\bPi\bdelta_i$ depends on $(\bF_i,\bB_i)$
only through $\bPi\bF_i$ and $\bPi\bB_i$, which multiply $\bphi^{(m)}$ and
$\bchi^{(m)}$, respectively. Any feasible pair for index $m$, padded
with zeros, is feasible for $m+1$ and yields the same $\bPi\bdelta_i$;
hence the minimum is nonincreasing. If $f_i=\bC_i\bphi^{(m)}$ and
$g=\bD\bchi^{(m)}$, Theorem~\ref{thm:lifting}(ii) gives $\bPi\bdelta_i\equiv0$.

(ii) Let $n=1$, $f_i(x)=\alpha x$, $g(x_i,x_j)=x_j$, and let node $i$ have
in-neighbours $j\neq l$. For $\bphi^{(1)}(x)=x$, $\bchi(x_i,x_j)=x_j$,
$\bF_i=\alpha$ and $\bB_i=1$, the residual vanishes identically. For
$\bphi^{(2)}(x)=[x,x^2]^\top$, the second component of $\bdelta_i$ equals
$(\alpha x_i+\sum_ma_{im}x_m)^2-[\bF_i\bphi(x_i)]_2-\sum_ma_{im}[\bB_i\bchi(x_i,x_m)]_2$.
For fixed $x_i$ and all other states fixed, it is a function of
$(x_j,x_l)$ of the form $2a_{ij}a_{il}x_jx_l+\tilde u(x_j)+\tilde v(x_l)$,
because every pairwise term depends on at most one of $x_j,x_l$. The
mixed difference on any rectangle $[x_j,x_j+h]\times[x_l,x_l+h]$
annihilates $\tilde u+\tilde v$ and equals $2a_{ij}a_{il}h^2\neq0$.
Hence the residual, which is continuous by (A2)--(A3), vanishes on no
open rectangle, so the set where it is nonzero is open and dense; for
any state distribution with a density that is positive on an open set,
$\varepsilon_{\rm rep}>0$ for every choice of $(\bF_i,\bB_i)$ and of the
pairwise dictionary.
\end{proof}

\section{Proof of Corollary~\ref{cor:sparsity}}
\label{app_cor2}
\begin{proof}
If $\norm{\hat\ba^{\rm alg}_k-\ba_k}_\infty<\tau_\star$, then for
$(i,j)\in\supp(\ba_k)$, $|\hat a^{\rm alg}_{ij,k}|\ge|a_{ij,k}|-\norm{\hat\ba^{\rm alg}_k-\ba_k}_\infty>a_{\min}-\tau_\star\ge\tau$,
and for $(i,j)\in\calI\setminus\supp(\ba_k)$,
$|\hat a^{\rm alg}_{ij,k}|=|\hat a^{\rm alg}_{ij,k}-a_{ij,k}|<\tau_\star\le\tau$;
hence $\hat\calE_k=\supp(\ba_k)$. (i) By Markov's inequality and
$\norm{\cdot}_\infty\le\norm{\cdot}$,
$\prob(\hat\calE_k\neq\supp(\ba_k))\le\prob(\norm{\hat\ba^{\rm alg}_k-\ba_k}\ge\tau_\star)
\le\Lnorm{\hat\ba^{\rm alg}_k-\ba_k}^2/\tau_\star^2$, and
Theorem~\ref{thm:error_bound} gives
$\limsup_k\Lnorm{\hat\ba^{\rm alg}_k-\ba_k}\le\bar Y/\sqrt\mu+\varepsilon_{\rm opt}$,
which yields~\eqref{eq:support_recovery}. (ii) On the event of
Proposition~\ref{prop:sparse}, which has probability at least $1-\delta$,
$\norm{\hat\ba^{\rm alg}_k-\ba_k}_\infty\le\norm{\bDelta_k}+\varepsilon_{\rm opt}
\le3\lambda_k\sqrt s/(2\kappa_{\rm RE})+\varepsilon_{\rm opt}<\tau_\star$.
\end{proof}

\bibliographystyle{unsrt}
\bibliography{ref}
\end{document}